\documentclass{svjour3}                     
\smartqed  
\usepackage{graphicx}
\usepackage{subfigure}
\usepackage{xspace}
\usepackage[T1]{fontenc}
\usepackage{comment}
\usepackage{graphicx}
\usepackage[headings]{fullpage}
\usepackage{amsmath, amssymb, amsfonts}
\usepackage{commath}
\usepackage{times}
\usepackage{color}
\usepackage{enumerate}
\usepackage{theorem}
\usepackage{hyperref}

\usepackage{verbatim} 

\theorembodyfont{\upshape}

\newtheorem{algorithm}{Algorithm}

\newcommand{\eref}[1]{\mbox{\rm(\ref{#1})}}

\newcommand{\real}{\mathbb{R}}
\newcommand{\complex}{\mathbb{C}}

\newcommand {\cA} {{\cal A}}

\newcommand {\cC} {{\cal C}}

\newcommand {\cF} {{\cal F}}

\newcommand {\cH} {{\cal H}}
\newcommand {\cI} {{\cal I}}

\newcommand {\cM} {{\cal M}}
\newcommand {\cP} {{\cal P}}

\newcommand {\cT} {{\cal T}}

\newcommand{\bc}{\mathbf{c}}

\newcommand{\be}{\mathbf{e}}
\newcommand{\bbf}{\mathbf{f}}
\newcommand{\bg}{\mathbf{g}}
\newcommand{\bh}{\mathbf{h}}

\newcommand{\br}{\mathbf{r}}

\newcommand{\bu}{\mathbf{u}}

\newcommand{\bP}{\mathbf{P}}
\newcommand{\bQ}{\mathbf{Q}}

\newcommand{\bU}{\mathbf{U}}
\newcommand{\bV}{\mathbf{V}}

\newcommand{\bomega}{{\boldsymbol{\omega}}}

\newcommand{\hf}[2]{{h_{\mu_{#1}}\!\del{#2}}}
\newcommand{\hfc}[2]{{h^*_{\mu_{#1}}\!\del{#2}}}

\newcommand{\hfd}[1]{{{\bf h}_{\mu_{#1}}}}

\newcommand{\ibasis}[1]{{\widetilde{e}_{#1}}}

\newcommand{\bigo}[1]{{\cal O}({#1})}

\newcommand{\triplenorm}[1]{\left| \! \right| \!\! | {#1} | \!\! \left| \! \right|}

\newcommand{\nwidth}[2]{d_{#1}}
\newcommand{\greedy}[2]{\sigma_{#1}}

\newcommand{\proF}[2]{\cP_{#1}}
\newcommand{\proFtwo}[2]{\widetilde{\cP}_{#1}}

\DeclareMathOperator*{\argsup}{\arg\!\sup}

\begin{document}


\title{Two-step greedy algorithm for reduced order quadratures}

\author{Harbir Antil, Scott E. Field, Frank Herrmann, Ricardo H. Nochetto, Manuel Tiglio}


\institute{Harbir Antil \at
              Department of Mathematical Sciences. George Mason University, Fairfax, VA 22030, USA. 
              Tel.: +703-993-2390\\
              Fax: +703-993-1491\\
              \email{hantil@gmu.edu}
          \and
           Scott E. Field \at
              Department of Physics, Joint Space Science Institute, Maryland Center for Fundamental Physics. University of Maryland, College Park, MD 20742, USA.
              \and
           Frank Herrmann \at
              Center for Scientific Computation and Mathematical Modeling, Department of Physics, Maryland Center for Fundamental Physics. University of Maryland, College Park, MD 20742, USA.
              \and
              Ricardo H. Nochetto \at 
              Department of Mathematics, and Institute of Physical Science and Technology, University of Maryland,
              College Park, MD 20742, USA.
              \and 
              Manuel Tiglio \at 
              Center for Scientific Computation and Mathematical Modeling, Department of Physics, Joint Space Science Institute, Maryland Center for Fundamental Physics. University of Maryland, College Park, MD 20742, USA. \\
}

\maketitle 

\begin{abstract}

We present an algorithm to generate application-specific, global 
{\em reduced order quadratures} (ROQ) for multiple fast evaluations of weighted inner products between parameterized functions. 
If a reduced basis (RB) or any other projection-based model reduction technique is applied,
the dimensionality of integrands is reduced dramatically; however, the cost of approximating the integrands by projection still scales as the size of the original problem. In contrast, using discrete empirical interpolation (DEIM) points as ROQ nodes leads to a computational cost which  depends linearly on the dimension of the reduced space. 
Generation of a reduced basis via a greedy procedure requires a training set, which for products of functions can be
very large. Since this direct approach can be impractical in many applications, we propose instead a two-step greedy targeted towards approximation of such products. We present numerical experiments demonstrating the accuracy and the efficiency of the two-step approach. The presented ROQ are expected to display very fast convergence whenever there is regularity with respect to parameter variation. We find that for the particular application here considered, one driven by gravitational wave physics, the two-step approach speeds up the offline computations to build the ROQ by more than two orders of magnitude. Furthermore, the resulting ROQ rule is found to converge  exponentially with the number of nodes, and a factor of $\sim 50$ savings, without loss of accuracy, is observed in evaluations of inner products when ROQ are used as a downsampling strategy for equidistant samples using the trapezoidal rule. While the primary focus of this paper is on quadrature rules for inner products of parameterized functions, our method can be easily adapted to integrations of single parameterized functions, and some examples of this type are considered.

\end{abstract}





\section{Introduction}  \label{sec:intro}

Many application areas deal with parameterized problems. Throughout this paper, we consider these to be a set of functions 
\[
   \cF_W:= \{ h_{\mu} : \Omega \rightarrow \complex \ | \ \mu \in \cP, h_\mu \in \cC\del{\Omega} \},
\] 
where $\Omega$, $\cP$ denote the physical and parameter domains,
respectively, $\complex$ is the set of complex numbers, $\cC\del{\Omega}$ is the set of continuous functions 
on a compact domain $\Omega \subset \real$ and $\cF_W \subset \cH_W$
denotes a compact subset of the Hilbert space $\cH_W: = L^2_{W}(\Omega)$. In general,
both $\Omega$ and $\cP $ are multi-dimensional and can be irregular 
domains. Here we assume $\cP$ to be  
compact in $\real^N$ and take the scalar product and norm between two functions $f,g \in \cH_W$ to be
\begin{equation}
\langle f, g \rangle_{L^2_W} := \int_{\Omega}  f^*(x) g(x) W(x) dx , \qquad 
            \| f \|_{L^2_W}^2:= \langle f, f \rangle_{L^2_W}  ,
            \label{eq:dp}
\end{equation}
with $0 < W \in \cC\del{\Omega}$ some weight function and $f^*$ denoting complex conjugation. We absorb the weight above into each integrand function through\footnote{For generic target applications we envision these functions not to be polynomials, and the weights not to be standard ones. In particular, our proposed reduced order quadrature construction does not rely on classical orthogonal polynomial theory for which the precise form of the weight $W$ is essential.}
$$
f(x) \rightarrow W^{\frac{1}{2}}(x) f(x)\, , 
$$
with $x \in \Omega$,  and define the set of functions
\begin{align*}
\cF = \{ h_{\mu} \sqrt{W} \in \cC\del{\Omega} \ | \ h_{\mu} \in \cF_W \}\, , 
\end{align*}
which is now a compact subset of the Hilbert space $\cH:=  L^2(\Omega)$. This allows us, without loss of generality, 
to use the standard $L^2$ inner product and the corresponding norm for elements of $\cH$, i.e., 
\begin{equation} \label{eq:dp_full}
    \langle f,g \rangle = \int_\Omega f^*(x) g(x) \, dx, \qquad \| f \|^2:= \langle f, f \rangle 
 \, .
\end{equation} 
Furthermore,  the discrete $L^2$ inner product and the corresponding norm are here denoted by
\begin{align}  \label{eq:dp_disc}
    \langle f, g \rangle_{\tt d} = \sum_{k=1}^M \omega_k f^*(x_k) g(x_k), \qquad \| f \|_{\tt d } ^2
                                                        =  \langle f, f \rangle_{\tt d} \, ,
\end{align}
where $\{  x_k, \omega_k \}_{k=1}^M $ are arbitrary quadrature points and weights.
Additionally, we often denote discrete objects and vectors
through bold notation, for example, $\bbf = \del{ f(x_1), \dots, f(x_M) }^T$ and $\bomega = \del{ \omega_1 , \dots, \omega_M }^T$. 
More details about the notation used throughout this paper are given in Table~\ref{table:notation}.

There are several challenges associated with parameterized problems and Reduced Order Modeling (ROM). One of them, when the set of functions $\cF$ is not known a priori but requires expensive numerical simulations, is how to select ``on the fly'' which parameters to solve for in a nearly optimal way, and a compact --application-specific spectral-- representation for every element in $\cF$. Reduced Basis (RB) \cite{PateraBook} is a leading candidate for such problems. We refer to Section \ref{sec:RB} for details about it; for the purposes of this paper it suffices to describe in that section the algorithm and some of its key features.
Another often desired aspect of any approach to parameterized problems is, once a reduced basis has been built, the ability to accelerate the computation of a particular quantity of interest. This can be fast online prediction of new solutions by solving a reduced problem~\cite{PateraBook} or, as in the case of this paper, fast online evaluations of inner products between elements of $\cF$.
Additional benefits may be gained from an interpolation-based approximation. The leading candidates for this are the Empirical Interpolation Method (EIM) \cite{Maday_2009,Barrault2004667} and its discrete counterpart, the Discrete Empirical Interpolation Method (DEIM)~\cite{chaturantabut:2737,Chaturantabut5400045}. Again, for the purposes of this paper it suffices to summarize the algorithm in Section \ref{sec:EIM} and some of its properties.

The specific goal of this paper is to build application-specific global quadratures for {\bf inner products of functions} (such as those appearing in convolutions, matched filtering,  Markov Chain Monte Carlo simulations, etc.) combining both ROM  and the DEIM, which for briefness we refer to as {\it reduced order quadratures} (ROQ). Here we explicitly consider RB for ROM, but the approach remains the same for other choices of bases generated using, for example, Proper Orthogonal or Singular Value Decompositions (POD, SVD)~\cite{MHinze_SVolkwein_2005a}.

The construction of ROQ is somewhat similar in spirit to that one of Gaussian quadratures: the function to be integrated is approximated by a truncated expansion in a reduced basis,
this expansion is replaced by interpolation (here DEIM) at special nodal points and, finally, the interpolant integrated to compute effective quadrature weights for any scalar product of functions in the set of interest $\cF$.  For many applications of interest, there are several advantages of ROQ with respect to Gaussian quadratures, though. First, regularity with respect to parameter variation, not with respect to the integration variable, is exploited to provide very fast (in many cases of interest, exponentially fast) convergence. 
For example, for the gravitational wave test cases considered in this paper, ROQ prove to be more efficient than Gaussian quadratures (see Fig.~\ref{fig:RBO_Quad}). In addition, the method is adaptive by nature, and the quadrature nodes are hierarchical. Our approach also gives, when acquiring data from experiments or observations, a nearly optimal way of downsampling it for matched filtering~\cite{manolakis2000statistical,Wainstein_L:1962} or other purposes, while at the same time preserving the accuracy in computing integrals with the full data set.  All these benefits of ROQ come at the cost of 
a potentially expensive offline procedure to pre-assemble both the basis and the ROQ rule. For many applications where multiple fast evaluations are required, especially if they are needed online or in real time, this tradeoff is indeed desirable. 

To the best of our knowledge, leveraging the advantages of the Empirical
Interpolation Method for fast numerical integration was first suggested by
Maday et al~\cite{Maday_2009} (see also
Refs.~\cite{Eftang:2011,Aanonsen2009}) and further investigations were
carried out by Aanonsen~\cite{Aanonsen2009}. However, our approach
differs from previous ones in several respects. First, the empirical
interpolant coefficients are usually found by carrying out a potentially
costly (compared to a competitive quadrature rule) matrix-vector
multiplication. Here we absorb this cost into an {\em
  offline} computation of the reduced order quadrature weights and
explicitly write the ROQ rule as a vector-vector
product (see Eq.~\eqref{eqn:RO_quad} in Section~ \ref{sec:overview}). Additionally, the RB-DEIM model
reduction used here allows for a natural factorization of the
parameter and physical dependences. 
Finally and most important, our main
objective here is fast computation of {\bf inner products} of the form 
$\langle h_{\mu_i},h_{\mu_j}\rangle$. While this is closely related to the integration of single functions in $\cF$, for large problems there are practical and serious obstacles towards a reduced basis representation of products of functions,  
denoted here as 
\begin{align}  \label{eq:product_set}
     \widetilde{\cF}  = \{ h_{\mu_i}^* h_{\mu_j} \in \cC\del{\Omega} \ | \ h_{\mu_i}, h_{\mu_j} \in \cF \}\, , 
\end{align}
in terms of 
very large training spaces. To address this we suggest a simple two-step approach targeted towards such products which reuses the algorithms necessary for reducing the underlying lower order space $\cF$, with dramatic savings in the offline stage. Without this two-step procedure, 
building ROQ would be simply not possible in many realistic application problems driving our work, even if carried out offline and using large supercomputers. 

This paper is organized as follows. Section~\ref{sec:overview} provides an overview of the main ideas to be developed 
 as well as brief preliminary discussions on quadrature rules, RB-DEIM, and reduced order quadratures. 
Section~\ref{sec:RB_DEIM} collects necessary results and algorithms required for developing ROQ. In particular, in Section~\ref{sec:RB}  we describe a RB-greedy model reduction algorithm to generate a set of basis vectors which are subsequently interpolated at a set of points chosen by 
the DEIM algorithm as summarized in Section~\ref{sec:EIM}. The main difficulties in constructing a reduced basis for $\widetilde{\cF}$, as well as our approach for overcoming them, are considered in Section~\ref{sec:Products}. Fast integrations using the reduced model and ROQ are constructed in Section~\ref{sec:ROQ_Alg}, along with some error estimates in Section~\ref{sec:convergence_rates}. Straightforward extensions of the ROQ algorithm for resampled functions are discussed in Section~\ref{sec:ROQ_newPoints}. Two numerical experiments are documented in Section~\ref{sec:numerics}. To compare with well known results in one and two dimensions, our first experiment considers polynomials on the standard interval $[-1,1]$ as basis, and integration of single functions. The second example draws from a non-trivial problem in gravitational wave physics and showcases the potential savings of ROQ to compute scalar products in large problems. In Section~\ref{sec:conclusion} we summarize the results of this paper, and comment on their possible extensions.

\section{Overview and main ideas}  \label{sec:overview}

The main goal of this paper is to efficiently compute approximations to integrals such as those in Eq.~\eqref{eq:dp_full}. We refer to the exact, continuum one as $I_{c}$, 
\begin{equation}   \label{eq:Ic} 
 I_c (i,j) :=  \langle h_{\mu_i}, h_{\mu_j} \rangle \, ,
\end{equation}
where $\{\mu_i, \mu_j \}$ are any two values in the parameter space $\cP$. 
Next we discuss different approximations to 
$I_c$. 

\smallskip
\noindent
{\bf Integration by quadrature:} 
Let us recall a typical quadrature rule, since the proposed ROQ of this paper follow a somewhat similar pattern. 
With $\left\{ x_k, \omega_k \right\}_{k=1}^M$ denoting an arbitrary set of quadrature points and weights, respectively, we can approximate the integral \eqref{eq:Ic} as 
\begin{align}  \label{eq:disc_full_L2prod}
    I_{c}(i,j) \approx I_{\tt d}(i,j) := \langle h_{\mu_i}, h_{\mu_j} \rangle_{\tt d}
        = \sum_{k = 1}^M \omega_k  \hfc{i}{x_k} \ \hf{j}{x_k} \, .
\end{align}
Many integration rules can be written in the form of Eq.~\eqref{eq:disc_full_L2prod}, including the extended trapezoidal rule, Gaussian quadratures, and integration using domain decomposition. Among these many choices an optimal strategy should allow for a given target integration accuracy with the smallest number of operations. 
However, if $M$ is large (perhaps the functions are sharply peaked, have different length scales and/or many cycles) and/or many evaluations of $I_{c}$ are needed (perhaps in real-time), computing the approximations $I_{\tt d}$ can become very costly. Problems which depend on a high dimensional spaces constitute such cases. In addition, if the target functions $\{ h_{\mu_j} \}$ come from acquired data, each function $h_{\mu_j}$ is usually sampled at equally spaced or scattered points and the integrations are usually carried out using an extended trapezoidal rule, with slow convergence in general. 

\smallskip
\noindent
{\bf Integration in a reduced space:} 
A preliminary reduced order approximation to $I_{\tt d}(i,j)$ 
would involve using an expansion in basis vectors $\{ e_\ell  \}_{\ell=1}^n$.
Without loss of generality we can assume that the basis is orthonormal with respect to the discrete \footnote{In carrying out the RB-greedy algorithm we resolve all the integrals involved within machine precision, so for all practical purposes within that context the discrete and continuum scalar products agree with each other.} inner product 
$\langle \cdot, \cdot \rangle_{\tt d}$. Then any function $h_{\mu} \in \cF$ can be approximated by 
\begin{align*}
       h_\mu \approx \proF{n}{} h_{\mu} := \sum_{\ell=1}^n \langle e_\ell , h_{\mu} \rangle_{\tt d}  e_\ell  \, ,
\end{align*}
where we have introduced the orthogonal projection operator $\proF{n}{}$, and an approximation to \eqref{eq:disc_full_L2prod} would be 
\begin{equation}
    I_{\tt d}(i,j)  \approx  \langle \proF{n}{} h_{\mu_i}, \proF{n}{} h_{\mu_j}\rangle_{\tt d}
  = \sum_{k=1}^n \langle  h_{\mu_i}, e_k\rangle_{\tt d} \langle e_k, h_{\mu_j}\rangle_{\tt d } \, . \label{eq:Irb}
\end{equation} 
Now, computing the approximation \eqref{eq:Irb} requires knowledge of the projection coefficients $\left\{ \langle  h_{\mu_i}  , e_k\rangle_{\tt d}  \right\}$. Modulo the latter, performing an integration via Eq.~\eqref{eq:Irb} is of improved computational cost ($n$ multiplications and $n-1$ additions) whenever the number of basis elements is smaller than the number of quadrature points, $n < M $. A greedy construction of a {\it reduced basis} (see Sec.~\ref{sec:RB}), as opposed to a standard basis choice such as Jacobi polynomials, involves the use of a problem-dependent {\it training space}. If the functions to be integrated, namely $h_{\mu_i}$ and $h_{\mu_j}$, are members of this training space, the projection coefficients 
have been precomputed offline while building the basis (see Section \ref{sec:RB}). The more interesting case in practice is that one in which $h_{\mu_i}$ and $h_{\mu_j}$ were not members of the training space, or one is unable to store the set of projection coefficients. 

\smallskip
\noindent
{\bf Integration with reduced order quadratures (ROQ):} 
We advocate an application-specific quadrature scheme for efficiently evaluating the discrete integral in \eqref{eq:disc_full_L2prod}, with essentially no loss of accuracy,  with a computational cost that depends only on the reduced basis space dimension, and without requiring projection coefficients. 
This new approach combines the flexibility of quadrature rules with powerful dimensionality reduction. 
We seek to approximate integrands of the form $g_{ij} = h_{\mu_i}^* h_{\mu_j}$,
\begin{align} \label{eq:RB_approxProducts}
       g_{ij} \approx \proFtwo{m}{} g_{ij} := \sum_{\ell=1}^m \langle \ibasis{\ell} ,g_{ij} \rangle_{\tt d}  \ibasis{\ell}  \, ,
\end{align}
with $g_{ij} \in \widetilde{\cF}$, and a truncated expansion is carried with a set of basis vectors 
$\{ \widetilde{e}_\ell  \}_{\ell=1}^m$ approximating elements in $\widetilde{\cF}$. As before, we are faced with computing 
$\langle \ibasis{\ell} ,g_{ij} \rangle_{\tt d}$, the cost of which can in principle be expensive. The empirical interpolation method (EIM) offers an attractive alternative: consider these $m$ basis 
$\{ \widetilde{e}_\ell  \}_{\ell=1}^m$, 
let $\{ \widetilde{p}_\ell \}_{\ell=1}^m \subset \{ x_k \}_{k=1}^M$ be a set of points (the generation of which will be described later), and let the EIM interpolant of $g_{ij}$ be
\begin{align*} 
  \widetilde{{\cal I}}_m[g_{ij}] := \sum_{\ell=1}^m \widetilde{c}_\ell \ibasis{\ell}\ \quad \mathrm{s.t.} \quad  \sum_{\ell=1}^m \widetilde{c}_\ell \ibasis{\ell}(\widetilde{p}_k) = g_{ij} (\widetilde{p}_k), \quad k = 1, \dots, m \, ,
\end{align*}
where the coefficients $\{ \widetilde{c}_\ell \}_{\ell=1}^m$
solve the interpolation problem. 
If the interpolant is accurate then $g_{ij} \approx {\widetilde{\cal I}_m}[ g_{ij} ]$
and, substituting this into Eq.\eqref{eq:disc_full_L2prod}, our approximation to the quadrature rule $\{x_k, \omega_k\}_{k=1}^M$ becomes
\begin{align} 
      I_{d}(i,j) \approx \sum_{k=1}^M \omega_k \widetilde{\cal I}_m[g_{ij}](x_k)  
 & = \sum_{k = 1}^M \omega_k \sum_{\ell=1}^m  \widetilde{c}_\ell \ibasis{\ell}(x_k) 
    = \sum_{\ell=1}^m \left[ \sum_{k = 1}^M \omega_k \ibasis{\ell}(x_k) \right] \widetilde{c}_\ell \, . \label{eq:roq}
\end{align}
Written in this way, the bracketed expression in the last term of Eq.~(\ref{eq:roq}) is seen to be the integration of the basis. 
Later on (cf.~Eq.(\ref{eq:ROQ})) we show that this expression can be rewritten in a more recognizable form
\begin{align} \label{eqn:RO_quad}
   I_{\tt d}(i,j)  \approx  I_{\mathrm{ROQ}}(i,j):= \sum_{\ell=1}^m \omega_\ell^\mathrm{ROQ} h_{\mu_i}^* (\widetilde{p}_\ell) h_{\mu_j}(\widetilde{p}_\ell) =  \sum_{\ell=1}^m \omega_\ell^\mathrm{ROQ} g_{ij}(\widetilde{p}_\ell)  \, .   
\end{align}
Eq.~\eqref{eqn:RO_quad} is our proposed Reduced Order Quadrature, which has an online evaluation cost of 
$\bigo{m}$. 
The interpolation points $\{\widetilde{p}_\ell\}_{\ell=1}^m \subset \{x_k\}_{k=1}^M$ are outputs of the EIM algorithm. Both these points and the weights $\omega_\ell^\mathrm{ROQ}$ only depend on the underlying quadrature rule given by $\{x_k, \omega_k\}_{k=1}^M$ and on the choice of reduced basis vectors $\ibasis{\ell}$; they are precomputed off-line and do not depend on the integrand $g_{ij}$. Note that the ROQ rule \eqref{eqn:RO_quad} achieves our efficiency requirements whenever $m < M$. In many cases one can expect exponential convergence of $I_\mathrm{ROQ} \rightarrow I_{\tt d}$ with respect to $m$, and for $I_{\tt d} \approx I_c$  we expect this convergence rate to be inherited by the limit to the continuum $I_\mathrm{ROQ} \rightarrow I_c$. Furthermore, even though this might be problem-dependent, in our numerical experiments (see Section~\ref{sec:SPAWaveforms}) we have found that $m \sim 2n \ll M$ (the equality $m=2n$ exactly holds for polynomial bases). For the case $m=M$, when all degrees of freedom have been exhausted, consistency requires both quadrature rules $\{\widetilde{p}_\ell, \omega_\ell^\mathrm{ROQ}\}_{\ell=1}^m$ and $\{x_k, \omega_k\}_{k=1}^M$ to be identical. Indeed, in this case the quadratures points trivially coincide and one can show ${\bf \omega^\mathrm{ROQ}} = {\bf \omega}$ (see Corollary~\ref{cor:consistency}).

A straightforward and largely self contained blueprint to generate the quadrature rule 
$\{\widetilde{p}_\ell, \omega_\ell^\mathrm{ROQ}\}_{\ell=1}^m$ is given in Algorithm~\ref{algo:ROQ}.

\smallskip
\noindent
{\bf Approximation of products with a two-step greedy:} The ROQ rule requires an approximation of integrands, which is where the main challenge lies; namely, in integrating products of target functions and not individual ones. A greedy construction of a reduced basis space approximating $\cF$ involves the use of a {\em training set} $\cT_K := \{ \mu_i\}_{i=1}^K \subset \cP$ of size $K$ which densely samples the continuum $\cP$. If a faithful approximation of $\cF$ requires a training set $\cT_K$ of size
$\bigo{K}$ then a training set for $\widetilde{\cF}$ would in principle be of size $\bigo{K^2}$. Since for large problems $\bigo{K}$ can already be computationally challenging, this direct approach might be unfeasible in some applications, even when building the basis is an offline and parallelizable calculation.

A greedy algorithm for $\cF$ returns a set of $n$ {\em greedy points} $\{\mu_\ell\}_{\ell=1}^n$ and orthonormal reduced basis $\{e_\ell\}_{\ell=1}^n$. The basic idea of the two-step greedy approach is to take advantage of the observation that all integrands can be accurately approximated by 
$$
h_{\mu_{i}}^* h_{\mu_{j}} \approx \left(\proF{n}{} h_{\mu_{i}}\right)^* \proF{n}{} h_{\mu_{j}} = \left(\sum_{k=1}^n \langle e_k, h_{\mu_{i}} \rangle_{\tt d} e_k\right)^*\left(\sum_{l=1}^n \langle e_l, h_{\mu_{j}} \rangle_{\tt d} e_l\right) \, ,
$$
which involves a sum of $n^2$ terms 
of the form $e_k^*e_l $. We then carry out a second layer of dimensional reduction -- a second greedy for the products $\{ e_k^*e_l \}_{k,l=1}^n$.
The training set for this second greedy is of size $\bigo{n^2}$, as opposed to $\bigo{K^2}$ 
(with, usually, $n^2 \ll K^2$). One may also consider a training set for $\widetilde{\cF}$ given by the Cartesian product of greedy points $\{\mu_i \}_{i=1}^n$ thereby sampling $\widetilde{\cF}$ is a smarter way. Both ideas are more thoroughly explored in Sec.~\ref{sec:Products}.

\begin{table}
\centering
\begin{tabular}{||rcl||}
\hline
$\cF$: & & set of target functions.\\
$\mu$: & & parameter in set $\cP$. \\
$h_{\mu}$: & & a sample target function in $\cF$.\\
$\widetilde{\cF}$: & & set of product of target functions.\\
$g_{\mu}$: & & a sample target function in $\widetilde{\cF}$.\\
$\widetilde{\mu}$: & & parameter in Cartesian product set $\cP \times \cP$\\
$\{e_\ell\}_{\ell=1}^n$: & & basis whose span $F_n$ approximates $\cF$.\\
$\proF{n}{}$: & & orthogonal projection to $F_n$. \\
${\cal I}_n[]$: & & empirical interpolation associated with $F_n$. \\
$\set{ \ibasis{\ell} }_{\ell=1}^m$: & & basis whose span $\widetilde{F}_m$ approximates $\widetilde{\cF}$.\\
$\proFtwo{m}{}$: & & orthogonal projection to $\widetilde{F}_m$. \\
$\widetilde{{\cal I}}_m[]$: & & empirical interpolation associated with $\widetilde{F}_m$. \\
$M$: & & number of samples to evaluate discrete inner products.\\
$n$: & & number of EIM points (and reduced basis) to approximate elements in $\cF$.\\
$m$: & & number of EIM points (and reduced basis) to approximate elements in $\tilde{\cF}$. Number of ROQ points.\\
$\cF_{\mathrm{train}}$: & & training space sampling $\cF$. \\
$K$: & & number of elements in $\cF_{\mathrm{train}}$.\\
$K^2$: & & number of training space elements in principle needed to {\em directly} sample $\widetilde{\cF}$.\\
$n^2$: & & number of training space elements used in our two-step greedy approach to sample $\widetilde{\cF}$. In general $n^2 \ll K^2$.\\
\hline
\end{tabular}
\caption{{\sc Notation.}
Notice that tildes refer to approximations or quantities related to $\widetilde{\cF}$, defined in Eq.~\eqref{eq:product_set}.
\label{table:notation}}
\end{table}

\section{Reduced basis and the empirical interpolation method} \label{sec:RB_DEIM}

The ROQ rule proposed here requires an
accurate interpolation representation for all integrands in the
space of interest. In the next two subsections we describe
the necessary RB-greedy and EIM algorithms to achieve this. Here we favor an RB-greedy
approach, since it is able to handle very large problems, such as
those that we are interested in. Our proposal is also applicable to other
basis choices such as those found through a Proper Orthogonal/Singular Value
Decomposition~\cite{MHinze_SVolkwein_2005a}. In Sec.~\ref{sec:Products} we propose  
a two-step RB-greedy approach for large sets of products of functions. Without loss of generality we assume  all 
functions to be normalized with respect to the discrete inner product $\langle \cdot, \cdot \rangle_{\tt d}$. 

\subsection{\bf Greedy construction of a reduced basis} \label{sec:RB}

The RB approach, combined with a greedy algorithm to generate the basis, provides a way to construct an application-specific 
expansion, where the basis elements are members of the space under consideration itself. Here the sets of interest are $\cF$ and $\widetilde{\cF}$. For definiteness in notation we will describe the approach for $\cF$, since an approach for $\widetilde{\cF}$ or any other space is identical; we will refer to this approach applied to $\widetilde{\cF}$ as a {\em direct}
or one-step RB-greedy.

A greedy approach is highly efficient in practice. It yields a nested basis set that is hierarchically constructed.
As is customary for spectral
expansions, having an orthonormal basis $\{e_\ell\}_{\ell=1}^n$ simplifies computing 
the orthogonal projection $\proF{n}{} h_{\mu }$ with respect to $\langle \cdot, \cdot \rangle_{\tt d}$ of a function
$h_\mu \in \cF$ onto the span $F_n$ of the basis, 
\begin{align*}
   F_n := \mathrm{span}\{ e_1,\dots,e_n\} \, .
\end{align*}
We recall (see, for example, \cite{Pinkus}) that the Kolmogorov $n$-width of $\cF$ in $\cH$ 
\begin{align} \label{eq:KNwidth}
     \nwidth{n}{M} (\cF;\cH) &:= \inf_{ \dim X_n \le n} \hspace{5pt} \sup_{h_{\mu} \in \cF } \hspace{5pt} \inf_{f \in X_n}
                                   \left\|   h_{\mu} - f  \right\|  
                               =  \inf_{ \dim X_n \le n} \hspace{5pt} \sup_{h_{\mu} \in \cF } \left\|   h_{\mu} - \cP_n h_{\mu}  \right\| \, ,
\end{align}                                   
measures the error of the best n-dimensional subspace $X_n \subset \cH$ approximating $\cF$, and the last equality follows from $\cH$ being a Hilbert space.
Computation of the $n$-width $\nwidth{n}{M}$, or any basis achieving it, is in most practical applications not possible (however, see \cite{Pinkus} for some cases where it can be done). 
Nevertheless, obtaining a convergence rate for the $n$-width provides valuable information towards understanding the approximability of a space by 
greedy algorithms.

We summarize a greedy strategy to build a reduced basis $\{ e_\ell\}_{\ell=1}^n$, with an approximation error
\begin{equation} \label{eq:RB-error}
    \greedy{n}{M} (\cF;\cH) := \sup_{h_{\mu} \in \cF }  \left\|   h_{\mu} -  \proF{n}{} h_{\mu }  \right\| \, ,
\end{equation}
which is nearly optimal with respect to the Kolmogorov $n$-width defined in Eq.~\eqref{eq:KNwidth}. 
Then we will state some of the available convergence rates for the {\it greedy error} $\greedy{n}{M}(\cF;\cH)$ based on \cite{Binev10convergencerates,DeVore2012}, which make precise the quoted near 
optimality, under the assumption that up to machine precision $\left\|   h_{\mu} -  \proF{n}{} h_{\mu }  \right\| \simeq \left\|   h_{\mu} -  \proF{n}{} h_{\mu }  \right\|_{\tt d}$.
While the greedy error \eqref{eq:RB-error} has been defined over $\cF$, in practice one 
samples the continuum using a training set $\cT_K := \{ \mu_i\}_{i=1}^K \subset \cP$ 
of size $K$ and a {\em training space} $\cF_{\mathrm{train}}$ of associated normalized functions $\{ h_{\mu_i} \}_{i=1}^K$.
Typically, if there is redundancy, the number of 
reduced basis needed to represent the training space is much smaller than the number of samples: $n \ll K$. 

Next let $\epsilon > 0$ be a user-specified error tolerance. Its role is to guarantee that the approximation error ensured by the RB-Greedy algorithm is strictly bounded as
\begin{equation}
    \greedy{n}{M} (\cF_{\mathrm{train}};\cH) := \sup_{h_{\mu} \in \cF_{\mathrm{train}} } \norm{h_\mu - \proF{n}{} h_\mu}_{\tt d} \le \epsilon \, . \label{eq:greedy_error1}
\end{equation}
To a-priori ensure that the training space $\cF_{\mathrm{train}}$ is a faithful approximation of the continuum $\cF$ is, in general, difficult. Construction and adaptive management of the training space $\cF_{\mathrm{train}}$ is an area of active research; see, for example, \cite{Eftang:2010,Eftang:2011,Bui:2007,Haasdonk:2010}. A good choice is dictated by the problem, further details for our application are given in the numerical results Section \ref{sec:SPAWaveforms}. To check that the reduced basis space $F_n$ is a faithful approximation of $\cF$ we typically do convergence tests with respect to the number of samples $K$ in $\cT_K$, and Monte Carlo reconstruction studies of the continuum; see for example \cite{Caudill:2011kv,Field:2011mf,PhysRevD.86.084046} and Section~\ref{sec:SPAWaveforms}. For any given tolerance $\epsilon >0$, the Algorithm~\ref{algo:RB-greedy} of Appendix~\ref{app:algs} ensures the strict bound \eqref{eq:greedy_error1} over the entire training space. However, a-posteriori validation has allowed us to establish a more impressive bound in all of our applications in gravitational wave physics so far \cite{Caudill:2011kv,Field:2011mf,PhysRevD.86.084046}, 
\begin{align}\label{sigma-epsilon}
\|h_\mu - \proF{n}{} h_\mu\|_{\tt d} \le ~ \greedy{n}{M} (\cF;\cH) \lesssim \epsilon  \quad \forall \mu \in \cP \, .
\end{align}
Such observation is obviously problem-dependent, and in general it should read 
$\greedy{n}{M} (\cF;\cH) \le \widetilde{\epsilon}$ 
where in principle $\tilde{\epsilon} \geq \epsilon$.

For any $\epsilon > 0$ and training set $\cT_K$ the greedy algorithm to build a reduced basis is given in App.~\ref{app:algs}. This algorithm returns a nested, hierarchical set of $n$ greedy points $\{\mu_\ell\}_{\ell=1}^n$ and orthonormal reduced basis $\{e_\ell\}_{\ell=1}^n$ which are nearly optimal with respect to the $n$-width~\eqref{eq:KNwidth}, in the following sense. Recall from \cite[Corollary 3.3]{DeVore2012} that if the $n$-width defined in \eqref{eq:KNwidth} decays exponentially, then
\begin{align}  \label{eq:kolmogorov-width-exp}
    \nwidth{n}{M} (\cF;\cH) \le C e^{-c_0 n^\alpha} \quad \rightarrow \quad \greedy{n}{M} (\cF;\cH) &\le \sqrt{2C} e^{-c_1 n^{\alpha}} \, , 
\end{align}
whence the greedy error in \eqref{eq:RB-error} decays exponentially, where $C$, $c_0$, $\alpha$, and $c_1 := 2^{-1-2\alpha} c_0$ are positive constants.
Similar statements can be made if the Kolmogorov $n$-width decays with a polynomial order,
\begin{align}  \label{eq:kolmogorov-width-alg}
    \nwidth{n}{M} (\cF;\cH) \le C_a n^{-\beta} \quad \rightarrow \quad \, \greedy{n}{M} (\cF;\cH) &\le 2^{5\beta+1} C_a n^{-\beta} \, , 
\end{align}
where $C_a$, $\beta$ are positive constants.

In Section~\ref{sec:SPAWaveforms} we summarize the observed exponential decay of the greedy error for a particular family of gravitational wave solutions.
We have also found that for a fixed representation error and parameter domain $\cP$ we can reconstruct any member function in $\cF$ with a {\em finite number of basis} elements. That is, as the size of $\cT_K $ increases, a finite asymptotic number of basis is needed to represent it; 
see \cite{Cannon:2010qh,Cannon:2011xk,Field:2011mf,Caudill:2011kv,PhysRevD.86.084046} for further studies.

\begin{remark}
  While it is obvious that $\nwidth{n}{M} (\cF;\cH) \le \greedy{n}{M} (\cF;\cH)$, a direct comparison obtained in \cite[Corollary 3.3]{DeVore2012} reads $\sigma_n \del{\cF;\cH} \le \sqrt{2 \ d_{n/2}\del{\cF;\cH} }$ for any $n$-width decay rate. If $d_{n/2}\del{\cF;\cH} $ decays exponentially, then one gets \eqref{eq:kolmogorov-width-exp}. Similarly, if $d_{n/2}\del{\cF;\cH}$ decays with a polynomial order, then one has \eqref{eq:kolmogorov-width-alg}.
\end{remark}

We now turn to some immediate applications of the reduced space
expressed with matrix-vector notation, which will prove convenient
throughout the remainder of this paper. Let 
\[
                     \bV := [ \be_1, \dots, \be_n ] \in \complex^{M \times n} \, ,
                \]
where the $\ell^\mathrm{th}$ column of the matrix $\bV$ is $\be_\ell = \del{ e_\ell(x_1),\dots,e_\ell(x_M) }^T \in \complex^M$; 
i.e., it corresponds to the  $\ell^\mathrm{th}$ reduced basis $e_\ell$ sampled at the
set of quadrature points $\{x_i\}_{i=1}^M$ used in the greedy algorithm\footnote{For certain applications it may be desirable to evaluate the functions at a different set of points $\{y_i\}_{i=1}^{M'}$,
see Section~\ref{sec:ROQ_newPoints} for details.}. 
In matrix-vector notation $\proF{n}{} {\bf h_{\mu }} = \bV \left[\bV^{\dagger} (\bomega \circ \hfd{}) \right] $ and the reduced discrete integral \eqref{eq:Irb} becomes 
\begin{align} \label{eqn:RB_discrete}
  I_{\tt d}(i,j) \approx  \langle \proF{n}{} h_{\mu_i}, \proF{n}{} h_{\mu_j} \rangle_{\tt d} 
                                                  & = \sum_{k=1}^n \langle h_{\mu_i}, e_k \rangle_{\tt d} \langle e_k, h_{\mu_j} \rangle_{\tt d} 
                =    \left(  \bV^{\dagger} (\bomega \circ \hfd{i}) \right)^{\dagger}
                                 \left(  \bV^{\dagger} (\bomega \circ \hfd{j}) \right)   \, ,
\end{align}
where $\bV^{\dagger}$ denotes the conjugate transpose of $\bV$ and $\circ$ denotes the Hadamard (pointwise) product between the vector components. Here we have reduced the overall {\em dimensionality} of the problem to $n$, but two issues remain. First, the number of $\{ h_{\mu}(x_k) \}_{k=1}^M$ functional evaluations depends on $M$ (Sec.~3.3 of Ref.~\cite{chaturantabut:2737} discusses the issue for a variety of nonlinear parametrized functions). Second, the dominant operation count for the approximation is seen to be of order $\bigo{Mn}$. We will return to these issues shortly.

\subsection{\bf The empirical interpolation method} \label{sec:EIM}

The main idea behind the EIM is to replace an expensive approximation by projection with a relatively inexpensive interpolation, without sacrificing accuracy. The  algorithm identifies a set of basis-specific interpolation points through a greedy selection criteria. When applied to several physical dimensions, irregular shaped domains, and rather generic parameterized spaces, the  EIM \cite{Maday_2009,Aanonsen2009,Eftang:2011} and its discrete counterpart, the DEIM \cite{chaturantabut:2737,Chaturantabut5400045}, and points have shown remarkable robustness and efficiency.

Suppose we have a set of basis vectors.  In particular, the algorithm applies to the reduced basis for the space $F_n$ or the space $\widetilde{F}_m$ to be constructed in Sec.~\ref{sec:Products}. The algorithm selects the DEIM points in physical space $\Omega$ and builds an associated interpolation matrix $\bP$. This matrix $\bP$ interpolates the columns of the matrix $\bV$ introduced in Sec.~\ref{sec:RB} (see Appendix~\ref{app:algs}.2). The columns of $\bP$ are unit vectors with a single unit entry at the location of the empirical interpolation points and zero elsewhere. For example, the first column of $\bP$  has a unit entry at the location $\mathrm{arg}\hspace{-1pt}\max | \be_1 |$ (see step 2 of Algorithm~\ref{algo:DEIM}). Furthermore, $\bP^T  \hfd{}$ extracts an $n$-subvector which is exactly equivalent to evaluating $h_{\mu}$ at $n$ empirical interpolation points $\{ p_i \}_{i=1}^n$. The algorithm to generate $\bP$ and $\{ p_i \}_{i=1}^n$ is given in App.~\ref{app:algs}.

With $\{ p_i \}_{i=1}^n$ the DEIM approximation is~\cite{chaturantabut:2737}
\begin{align} \label{eim_disc}
    \hfd{} \approx {\cal I}_n[\hfd{}] := \bV ( \bP^T\bV)^{-1} \bP^T \hfd{} \, ,
\end{align}      
which is indeed an interpolant $\bP^T {\cal I}_n[\hfd{}] =  \bP^T (\bV (\bP^T\bV)^{-1} \bP^T \hfd{} ) = \bP^T \hfd{}$. Note that $\bP^T\bV$ is invertible when the columns of $\bV$ are linearly independent. Substituting ${\cal I}_n[\hfd{}]$ from Eq.~\eqref{eim_disc} into Eq.~\eqref{eq:disc_full_L2prod} gives
\begin{equation}\label{approx-integral}
I_{\tt d}(i, j)
\approx \langle {\cal I}_n[\hfd{i}], {\cal I}_n[\hfd{j}] \rangle_{\tt d} 
\approx 
  \left(  ( \bP^T\bV)^{-1} \bP^T \hfd{i}  \right)^{\dagger} \left(  ( \bP^T\bV)^{-1} \bP^T \hfd{j} \right)   \, .
\end{equation}
As we only have to evaluate $\bP^T \hfd{} \in \complex^n$, with $n < M$ in many applications,
we have reduced the complexity (that is, the number of required functional evaluations) to $\bigo{n}$
(compare with \eqref{eqn:RB_discrete}). However, due to the two matrix-vector multiplications 
an evaluation cost of $\bigo{n^2}$ has
appeared. 
Such unacceptable scaling arises when
approximating the individual functions $h_{\mu}$ instead of the
products (integrands) $h_{\mu_i}^*h_{\mu_j}$ directly. We consider the
generation of a reduced basis for products of functions next.

\subsection{\bf Two-step greedy approach for the approximation of product of target functions} \label{sec:Products}

\begin{align}  \label{eq:greedy_for_products}
  \nearrow  \hspace{0.5cm}   & \xrightarrow{\mbox{\it Path \#1}} \hspace{0.4cm} \searrow{}  \nonumber \\ 
     \widetilde{\cF}  \xrightarrow[\mbox{\it Path \#2 (i)} ]{} & F_{n^2}  \xrightarrow[\mbox{\it Path \#2 (ii)}]{}   
      \widetilde{F}_m
\end{align}

As discussed at the end of Secs.~\ref{sec:RB} and~\ref{sec:EIM}, we seek to approximate all possible integrands $g_{ij}=h_{\mu_{i}}^* h_{\mu_{j}} \in \widetilde{\cF}$, where $\widetilde{\cF}$ is a subset of the Hilbert space $\cH = L^2(\Omega)$. 
A natural approach, which we refer to as a {\em direct} one, is to build a reduced basis for $\widetilde{\cF}$
through a greedy algorithm as described in
Algorithm~\ref{algo:RB-greedy} and is marked as {\it Path \#1} in \eqref{eq:greedy_for_products}. In that case the results of Secs.~\ref{sec:RB} and~\ref{sec:EIM} are directly applicable, with $\cF$ and $F_n$ replaced by 
$\widetilde{\cF}$ and $\widetilde{F}_{m}$ respectively. 
In Sections \ref{sec:overview} and \ref{sec:SPAWaveforms} we argue that this direct approach can be impractical for some applications. Advanced algorithms for sampling or building the training space, perhaps adaptively~\cite{Eftang:2010,Eftang:2011,Bui:2007,Haasdonk:2010}, might overcome this issue, though, and allow for a direct approach to such large problems.

In Sec.~\ref{sec:overview} we motivated a {\it two-step greedy approach} via {\it Path \#2} in \eqref{eq:greedy_for_products}. 
In the first step ({\it Path \#2 (i)} in \eqref{eq:greedy_for_products})
we generate, through a first greedy as in Algorithm~\ref{algo:RB-greedy}, a reduced basis whose span $F_n$ approximates $\cF$. 
Next one might construct a set $F_{n^2}$, consisting of all 
$n^2$ products $e_k^*e_l  /\| e_k^*e_l \|_{\tt d}$. Another option, and indeed the one used in our numerical experiments, is to use a training set $\cT_{n}^2 := \{ ( \mu_i , \mu_j ) \}_{i,j=1}^n$ and associated normalized products $h_{\mu_i}^*h_{\mu_j} / \| h_{\mu_i}^*h_{\mu_j} \|_{\tt d} $ to define $F_{n^2}$, where $\{\mu_i \}_{i=1}^n$ are the greedy points identified by the first greedy algorithm when approximating ${\cal F}$ by  $F_n$. In other words, we do not necessarily use an orthonormal reduced basis for $F_n$ (see Remark~\ref{remark:TwoStepChoices}, below).
In either case, we approximate $F_{n^2}$ through a second RB-greedy  ({\it Path 2 (ii)} in \eqref{eq:greedy_for_products}), carried out again as in Algorithm~\eqref{algo:RB-greedy} but with $F_{n^2}$ as the training space. The  result is an orthonormal set of $m$ reduced basis $\{\ibasis{i}\}_{i=1}^m$ such that
$\widetilde{F}_m := \mathrm{span}\{ \ibasis{1},\dots,\ibasis{m}\}  $
accurately approximates $F_{n^2}$, and in turn the full set $\widetilde{\cF}$.

To summarize, the second step of our proposed two-step greedy algorithm to build $\widetilde{F}_m$ is given by Algorithm~\ref{algo:RB-greedy_product}, and for simplicity we choose the tolerance for the second greedy step to be equal to the first one ($\epsilon$ in Eq.(\ref{eq:greedy_error1})). To better present the algorithm it will be useful to introduce new notation for elements of $\cT_n^2$. Define $\widetilde{\mu}_{k(i,j)} := (\mu_i, \mu_j)$ to be a re-indexing of 
the array $\{(\mu_i, \mu_j)\}_{i,j=1}^n$ into 
$\{ \widetilde{\mu}_k \}_{k=1}^{n^2}$, where $k$ is a one-to-one function which takes two integers $i$ and $j$ and returns a unique integer $k(i,j)$. 

\begin{remark} \label{remark:TwoStepChoices}
The first greedy returns both the greedy points $\{ \mu_i \}_{i=1}^n$, an orthonormal basis $\{ e_i \}_{i=1}^n$, and the  {\em greedy basis functions} $\{h_{\mu_i}\}_{i=1}^n$, serving as a non-orthonormal basis set for the same space $F_n$. As discussed above, from either $\{h_{\mu_i}\}_{i=1}^n$ or $\{ e_i \}_{i=1}^n$ one can build the set 
$F_{n^2}$. In our numerical experiments we choose $\{ h_{\mu_i} \}_{i=1}^n$ as basis, partly because we need not store the orthonormal basis vectors $\{ e_i \}_{i=1}^n$ but only the greedy points $\{ \mu_i \}_{i=1}^n$.
Clearly the implementation and interpretation of Algorithm \ref{algo:RB-greedy_product} will depend on this choice. For example, when using $\{h_{\mu_i}\}_{i=1}^n$ it makes sense to discuss the selected greedy parameter points $\widetilde{\mu}_k$ while for $\{ e_i \}_{i=1}^n$ no such interpretation holds; instead one should view Step 5d as returning a column index when $F_{n^2}$ is thought of as an $M \times n^2$ matrix. 
\end{remark}

\begin{algorithm} [Two-step RB-Greedy approximation for $\widetilde{F}_m$]  \label{algo:RB-greedy_product} $\quad$\\
$\left[\{\widetilde{e}_\ell\}_{\ell=1}^m,\{\widetilde{\mu}_\ell\}_{\ell=1}^m\right]$ = Two-step RB-Greedy 
$\del{ \epsilon, \{\mu_i \}_{i=1}^n, \{ e_i \}_{i=1}^n }$   \\
{\bf Comment:} $\{ e_i \}_{i=1}^n$ are either orthonormal or greedy basis functions (cf. remark~\ref{remark:TwoStepChoices})
\begin{itemize}
\setlength{\itemsep}{2pt}
\item[(1)] Let $\cT_n^2 = \{ ( \mu_i , \mu_j ) \}_{i,j=1}^n$
\item[(2)] Let $F_{n^2} = \{ e_i^*e_j / \| e_i^*e_j \|_{\tt d}  \}_{i,j=1}^n$ 
\item[(3)] Set $m=0$ and define $\greedy{0}{M}(F_{n^2}; \cH) := 1$
\item[(4)] Choose an arbitrary $g \in F_{n^2}$ and set $\widetilde{e}_1 := g$ $\quad$ {\bf Comment:} $\|g\|_{\tt d} = 1$
\item[(5)] do, while $\greedy{m}{M}(F_{n^2}; \cH) \ge \epsilon$
\begin{itemize}
\setlength{\itemsep}{2pt}
\item[(a)] $m = m + 1$
\item[(b)] $\greedy{m}{M}(g) := \left\|  g - \proFtwo{m}{} g \right\|_{\tt d}$ for all $g \in F_{n^2}$
\item[(c)] $\greedy{m}{M}(F_{n^2}; \cH) =  \sup_{ g \in F_{n^2}} \left\{ \greedy{m}{M}(g) \right\}$ 
                 $\quad$ 
\item[(d)] $\widetilde{\mu}_{m+1} :=  \argsup_{g \in F_{n^2}}  \left\{ \greedy{m}{M} (g) \right\} \quad \mbox{(greedy sweep)} $
\item[(e)] $\widetilde{e}_{m+1} := g_{\widetilde{\mu}_{m+1}} - \proFtwo{m}{} g_{\widetilde{\mu}_{m+1}} \quad \mbox{(Gram-Schmidt)} $
\item[(f)] $\widetilde{e}_{m+1} := \widetilde{e}_{m+1} / \| \widetilde{e}_{m+1} \|_{\tt d} \quad \mbox {(normalization)}$
\end{itemize}
\end{itemize}
\end{algorithm}
\vspace{10pt}

A DEIM approximation for the integrand leads to an efficient and simple reduced order quadrature rule~\eqref{eqn:RO_quad} and is the subject of the next Section.

\section{Reduced order quadratures (ROQ)} \label{sec:ROQ}

In classical theory of Gaussian quadratures one seeks to maximize the exactness for the integration of polynomials. Here the goal is to empirically maximize the accuracy of inner products between elements of $\cF$.
Notice that an approximation of (discrete) inner products may be carried out in either a reduced space $F_n \approx \cF$ or $\widetilde{F}_m \approx \widetilde{\cF}$.
In Sec.~\ref{sec:EIM} we noted that working in $F_n$ leads to a pessimistic $\bigo{n^2}$ cost for inner products (see the discussion below Eq.~(\ref{approx-integral})). Thus, for applications where one needs fast evaluations of many scalar products the approximation space $\widetilde{F}_m$ is preferable. For the construction of $\widetilde{F}_m$ we refer to \eqref{eq:greedy_for_products}. 

Section~\ref{sec:ROQ_Alg} focuses on a ROQ stemming from $\widetilde{F}_m$, although only very minimal changes are needed to adapt Algorithm~\ref{algo:ROQ} for obvious variations. Indeed, one such variation is considered in our first numerical experiment. In Section~\ref{sec:ROQ_newPoints} we consider extensions of the ROQ construction to situations where one might only be able to evaluate integrands at a set of points $\{y_i\}_{i=1}^{M'}$ which do not correspond to the nodes used in the quadrature rule $\{ x_k, \omega_k \}_{k=1}^M$. 
In Section~\ref{sec:rates_deim} we present the DEIM
interpolation error estimates, followed by some ROQ error estimates in Section~\ref{sec:rates_ROQ}. 

\subsection{\bf ROQ Algorithm} \label{sec:ROQ_Alg}

Suppose we are given a set of functions $\cF$ and an arbitrary quadrature rule $\{ x_k, \omega_k \}_{k=1}^M$ for 
the discrete inner product~\eqref{eq:disc_full_L2prod}. The following algorithm generates the ROQ nodes and weights.

\begin{algorithm}[Construction of Reduced Order Quadratures]  \label{algo:ROQ}  $\quad$\\
$\left[\{ \widetilde{p}_\ell, \omega_\ell^\mathrm{ROQ} \}_{\ell=1}^m\right]$ 
         = ROQ ($\epsilon$, $\cT_K$, $\{ x_k, \omega_k \}_{k=1}^M$)   
\begin{enumerate}[(1)]
  \item {\bf Approximation of $\widetilde{\cF}$}: We have two options to approximate
                 $\widetilde{\cF}$ by $\widetilde{F}_m$, (cf. \eqref{eq:greedy_for_products}).
            
            \vspace{7pt}
            \begin{enumerate}[\it \text{Path} \#1]  
                   \item Consider the training set $\cT_K^2 = \{ ( \mu_i , \mu_j ) \}_{i,j=1}^K$ and apply the Greedy 
                            Algorithm~\ref{algo:RB-greedy}, 
                             with obvious modifications, to generate $\widetilde{F}_m = \mbox{span} 
                             \{ \widetilde{e}_\ell \}_{\ell=1}^m$ and greedy points
                             $\{ \widetilde{\mu}_\ell \}_{\ell=1}^m  \subset \cT_K^2$ such that 
                             any product 
                             $h_{\mu_i}^* h_{\mu_j} \in \widetilde{\cF}$ can be approximated by its projection 
                             $\widetilde{\cP}_m$ as:
                             \begin{align*}
                                 \norm{h_{\mu_i}^* h_{\mu_j} - \proFtwo{m}{} \del{ h_{\mu_i}^* h_{\mu_j} } }_{\tt d} & \le \epsilon \, ,
                                  \qquad \forall \del{\mu_i, \mu_j} \in \cT_K^2 \, .
                            \end{align*}
                            $Note$: This direct approach can be expensive; we therefore advocate the following two-step procedure. 
                          
                   \vspace{7pt}     
                   \item 
                   
                        \begin{enumerate}[$(i)$]
                              \item {\bf Approximation of $\cF$}:
                                       Consider the training set $\cT_{K} = \{ \mu_i \}_{i=1}^K$ and apply the Greedy 
                                       Algorithm~\ref{algo:RB-greedy} to generate $F_n = \mbox{span} \{ e_\ell \}_{\ell=1}^n$
                                       and greedy points $\cT_n = \{ \mu_\ell \}_{\ell=1}^n \subset \cT_K$ such that any $h_{\mu} \in \cF$ can be 
                                       approximated by its 
                                       projection $\cP_n$ as:
                                       \begin{align*}
                                             \norm{h_{\mu} - \cP_n  h_{\mu} }_{\tt d} & \le \epsilon \, , \qquad \forall \mu \in \cT_{K} \, .
                                     \end{align*}
                              \item 
                                    {\bf Construction of $F_{n^2}$}:  
                                     Define the training set 
                                     $\cT_n^2 = \cT_n \times \cT_n = \{  ( \mu_i , \mu_j ) \}_{i,j=1}^n$ and training space
                                     $F_{n^2} = \set{ e_i^*e_j / \| e_i^*e_j \|_{\tt d}}_{i,j=1}^n$.                                     
                                      Another choice for $F_{n^2}$ is 
                                      $\set{ h_{\mu_i}^*h_{\mu_j} / \| h_{\mu_i}^*h_{\mu_j} \|_{\tt d}}_{i,j=1}^n$,
                                      we refer to Remark~\ref{remark:TwoStepChoices} for justification.  
                              \item {\bf Approximation of $\widetilde{\cF}$}:
                                        Apply the Greedy Algorithm~\ref{algo:RB-greedy_product} to generate 
                                        $\widetilde{F}_m = \mbox{span} \{ \widetilde{e}_\ell \}_{\ell=1}^m$ and 
                                        $\{ \widetilde{\mu}_\ell \}_{\ell=1}^m \subset \cT_{n}^2 \subset \cT_{K}^2$ 
                                        such that any product 
                                        $e_i^* e_j \in F_{n^2}$ can be approximated by its projection $\widetilde{\cP}_m$ as:
                                        \begin{align*}
                                              \norm{e_i^* e_j - \proFtwo{m}{} \del{ e_i^* e_j } }_{\tt d} & \le \epsilon \, .
                                        \end{align*}
                            \end{enumerate}
                                  
            \end{enumerate}
          \item Define the matrix $\widetilde\bV = [ \widetilde{\be}_1, \dots, \widetilde{\be}_m ] \in \complex^{M \times m}$                
               where, for example, the $\ell^\mathrm{th}$ column is 
              $\widetilde{\be}_\ell = \del{ \widetilde{e}_\ell(x_1), \dots , \widetilde{e}_\ell(x_M) }^T$.
              
         \item {\bf Generation of empirical interpolation points}: 
               Apply the DEIM Algorithm \ref{algo:DEIM} with $\widetilde\bV$ as an input to compute the DEIM points
               $\{ \widetilde{p}_\ell \}_{\ell=1}^m \subset \{ x_i \}_{i=1}^M$ and the interpolation matrix $\widetilde{\bP}$.
         \item Compute the ROQ weights by
             \begin{align} \label{eqn:ROQ_weights}
                 \del{\bomega^\mathrm{ROQ}}^T := \bomega^T \widetilde{\bV} ( \widetilde{\bP}^T \widetilde{\bV})^{-1}  \, . 
             \end{align}       
\end{enumerate}
\end{algorithm}

\vspace{10 pt}

Step 4 deserves a bit more explanation. With the interpolation matrix $\widetilde{\bP}$ from Step 3 we first build an expression for the DEIM interpolant of some product $g \in \widetilde{\cF}$ 
\begin{align} \label{eq:eim_disc}
{\bf g} \approx \widetilde{\cal I}_m[{\bf g}] := \widetilde{\bV} ( \widetilde{\bP}^T \widetilde{\bV})^{-1} \widetilde{\bP}^T {\bf g} \, .
\end{align}  
The reduced order weights are then found by substituting this expression into Eq.~\eqref{eq:disc_full_L2prod} to produce the desired approximation with $\bg = \bh_{\mu_i}^* \circ \bh_{\mu_j}$, 
\begin{align}
I_{\tt d}(i,j) = 
\bomega^T {\bf g} \approx \bomega^T \widetilde{\cal I}_m[{\bf g}] = \left[ \bomega^T \widetilde{\bV} ( \widetilde{\bP}^T \widetilde{\bV})^{-1} \right] \widetilde{\bP}^T {\bf g} = \sum_{i=1}^m \omega_i^\mathrm{ROQ} g(\widetilde{p}_i) \, . \label{eq:ROQ}
\end{align}
Notice that the term $\widetilde{\bP}^T {\bf g}$ is just evaluation of $g$ at the DEIM points $\{ \widetilde{p}_i \}_{i=1}^m$ (which need not be ordered, but can be if desired). Furthermore, the weights are explicitly parameter independent and for less accuracy we might consider using the first $m' \le m$ functional evaluations of $g$
\begin{align*}
\bomega^T \widetilde{\cal I}_{m'}[{\bf g}] \approx \sum_{i=1}^{m{}'} \alpha_i^\mathrm{ROQ} g(\widetilde{p}_i) \, ,
\end{align*}
which, as this is equivalent to carrying out the computation in an $m{}'$-dimensional subspace of $\widetilde{F}_m$, is also a valid integration rule. In general the weights $\alpha_i^\mathrm{ROQ}$ associated with this $m'$ point quadrature rule will not be a subset of the weights $\omega_i^\mathrm{ROQ}$ computed for the $m$ point rule. Notice that the hierarchical nature of the empirical interpolation set allows one to build application-specific nested quadratures of arbitrary order and depth (i.e. more than one embedded method). For example, suppose $m=20$; one might combine the $10$-point, $15$-point, and $20$-point ROQ rules to yield a nested quadrature scheme.

\begin{remark}[FLOP-Count] \label{rmk:flop-count}
Let $K$ be the number of parameters in the training space to approximate $\cF$ by $F_n$. If $\widehat{m}$ and $m$ 
denote the number of reduced basis for direct ({\em Path \#1}) and two-step greedy ({\em Path \#2}) respectively, then 
numerically we observe that $m \approx \widehat{m}$. 
The cost of ROQ Algorithm~\ref{algo:ROQ} using a direct and a two-step greedy is $\bigo{K^2M\widehat{m}}$ and
$\bigo{n^2Mm + KMn}$, respectively. This implies that for $K \gg n$ (as is observed in our numerical experiments), the two-step greedy is highly efficient. 
For the details of FLOP counts we refer to Appendix~\ref{app:flop-count}. 
\end{remark}

Next we state some results which 
highlight the importance of using an accurate quadrature rule $\{ x_k, \omega_k \}_{k=1}^M$; 
these will be verified in Section~\ref{sec:Leg}. 

\begin{theorem} \label{thm:exact-integration}
Integration of the basis functions $\{ \widetilde{e}_\ell \}_{\ell=1}^m$
with either of the quadrature rules $\{ x_i, \omega_i \}_{i=1}^M$ or $\{ \widetilde{p}_\ell, \omega_\ell^\mathrm{ROQ} \}_{\ell=1}^m$ yield identical results.
\end{theorem}
\begin{proof}
The integration of basis $\{ \widetilde{e}_i \}_{i=1}^m$ using ROQ (cf. \eqref{eqn:ROQ_weights}, \eqref{eq:ROQ}
in vector notation) is
\[
   I_{\mathrm{ROQ}} =  \del{\bomega^{\mathrm{ROQ}}}^T \widetilde{\bP}^T \widetilde\be_i 
        = \bomega^T \widetilde{\bV} ( \widetilde{\bP}^T \widetilde{\bV})^{-1} \widetilde\bP^T \widetilde\be_i \, , \quad i = 1, \dots, m \, .
\]
Recalling that $\widetilde{\bV} = [\widetilde\be_1, \dots, \widetilde\be_m]$, integration of all the basis 
$\{ \widetilde{e}_i \}_{i=1}^m$ can be written as $\del{\bomega^{\mathrm{ROQ}}}^T  \widetilde{\bP}^T \widetilde\bV =  \bomega^T \widetilde\bV$, 
where $\bomega^T \widetilde{\bV}$ is precisely the integration of all the basis vectors $\{ \widetilde{e}_i \}_{i=1}^m$ with the quadrature rule $\{ x_k, \omega_k \}_{k=1}^M$.
\end{proof}
\begin{corollary}[Consistency of ROQ rule] \label{cor:consistency}
When $m=M$ the quadrature rules $\{ x_i, \omega_i \}_{i=1}^M$ and $\{ \widetilde{p}_\ell, \omega_\ell^\mathrm{ROQ} \}_{\ell=1}^m$ are identical up-to ordering of the quadrature points.
\end{corollary}

\begin{proof}
Using Theorem~\ref{thm:exact-integration}, we have $\del{\bomega^{\mathrm{ROQ}}}^T  \widetilde{\bP}^T \widetilde\bV =  \bomega^T \widetilde\bV$. Apply $\widetilde\bV^{-1}$ on the right to deduce $\del{\bomega^{\mathrm{ROQ}}}^T  \widetilde{\bP}^T =  \bomega^T$.
\end{proof}

\begin{remark}
If the quadrature rule $\{ x_i, \omega_i \}_{i=1}^M$ integrates the basis functions $\{ \widetilde{e}_\ell \}_{\ell=1}^m$
exactly then ROQ built using Algorithm~\ref{algo:ROQ} also integrates those basis functions exactly.
\end{remark}

\begin{remark}
As a special case consider $\widetilde{\cF}$ to be a space of polynomials of degree $m-1$ with weight $W=1$ and suppose the basis is specified by an ordered set of Legendre polynomials. Then the $m$-point ROQ rule will exactly integrate any polynomial of at most degree $m-1$ provided the quadrature rule $\{ x_k, \omega_k \}_{k=1}^M$ is at least $m-1$ exact. 
\end{remark}
             
\subsection{\bf Convergence Estimates}\label{sec:convergence_rates}

\subsubsection{\bf DEIM Convergence} \label{sec:rates_deim} 
We shall first state a DEIM interpolation error estimate with respect to the discrete $L^{\infty}$-norm. Then we will recall the DEIM error estimate with respect to the discrete $L^2$-norm from 
\cite{chaturantabut:2737} and, finally, combine this estimate with those for the RB-greedy basis construction from \cite{DeVore2012}. %
\begin{theorem}[discrete empirical interpolation method]\label{thm:deim-estimate}
Let the set of reduced basis $\{ e_\ell \}_{\ell=1}^n$ be orthonormal with respect to the discrete inner product defined in \eqref{eq:dp_disc} and $h^{opt}_\mu \in F_n$ be the optimal approximation of $h_\mu$ with respect to the discrete $L^\infty$-norm. Then for every
$\mu \in \cP$ and $x_k$, with $k=1,\dots,M$, 
\begin{align} \label{eq:deim-error-Linf}
\max_{ 1 \le k \le M} \left| h_\mu(x_k) - {\cal I}_n[h_\mu](x_k) \right|
      \le  \left( 1+  \Lambda_{n,\infty} \right) \max_{ 1 \le k \le M}  \left| h_{\mu}(x_k) - h^{opt}_\mu(x_k) \right| \, ,
\end{align}
where $\Lambda_{n,\infty} = \triplenorm{\cI_n}_\infty = \triplenorm{  \bV (\bP^T \bV)^{-1}\bP^T  }_{\infty} $ denotes the Lebesgue constant.
 
\noindent In the case of the discrete $L^2$-norm, 
$h_{\mu}^{opt} = \proF{n}{} h_\mu$  and we have for all $\mu \in \cP$
\begin{align} \label{eq:deim-error-L2}
\norm{ h_\mu - {\cal I}_n[h_\mu]  }_{\tt d} \le  \Lambda_{n,2}
\norm{ h_{\mu} - \proF{n}{} h_\mu }_{\tt d}\, , 
\end{align}
where $\Lambda_{n,2} = \triplenorm{\cI_n}_2 = \triplenorm{  \bV (\bP^T \bV)^{-1}\bP^T  }_{2}$  
\footnote{The 2-norm and $\infty$-norm of a matrix $\bQ \in \complex^{M\times n}$ is defined by 
\[
     \triplenorm{\bQ}_2 := \max_{\bu \neq 0} \frac{ \| \bQ \bu \|_{\tt d }} { \| \bu \|_{\tt d} } = \del{ \lambda_{\max} (\bQ^{\dagger} \bQ) }^{1/2}\, , \qquad
     \triplenorm{\bQ}_\infty :=  \max_{1 \le k \le M} \sum_{\ell=1}^{n} | q_{k\ell} | \, ,
\]
where $q_{k\ell}$ denotes the $(k,\ell)$ entry of $\bQ$.}

\end{theorem}
\begin{proof}
 The proof of \eqref{eq:deim-error-Linf} and \eqref{eq:deim-error-L2} are straightforward and are omitted.
   Furthermore,
    \eqref{eq:deim-error-L2} follows from \cite[Lemma 3.2]{chaturantabut:2737}, but we 
    state it for completeness. Since $\cI_n [h_\mu^{opt}](x_k) = h_{\mu}^{opt}(x_k)$ for $k=1,\dots,M$, we get
    \begin{align*}
        \norm{h_\mu - \cI_n\sbr{h_\mu}}_{\tt d}
          &= \norm{ \del{ \mathbb{I} - \cI_n } [ h_\mu - h_{\mu}^{opt}  ] }_{\tt d} \\
          &\le  \triplenorm{ \cI_n }_2  \norm{ h_\mu - h_{\mu}^{opt}  }_{\tt d} \, ,
    \end{align*}
    where the last equality follows from the fact that 
    $\triplenorm{ \mathbb{I} - \cI_n }_2 =  \triplenorm{ \cI_n }_2$ (see, for example \cite{JXu_LZikatanov_2003a,DBSzyld_2006a}).
\end{proof}

\begin{remark} \label{remark:DEIM_bounds}
Notice that the DEIM Algorithm \ref{algo:DEIM} seeks to minimize the interpolation error as measured by a discrete $L^\infty$-norm (cf.~Step 6, Algorithm \ref{algo:DEIM}), for which an error bound of the type $\eqref{eq:deim-error-Linf}$ is closely related. The RB-greedy Algorithm \ref{algo:RB-greedy}, however, exactly minimizes the term $\norm{ h_{\mu} - \proF{n}{} h_\mu }_{\tt d}$ on the right hand side of the error bound \eqref{eq:deim-error-L2} and is therefore computable.
Observing that $\Lambda_{n,2} \le \triplenorm{ \bV }_2  \triplenorm{ (\bP^T \bV)^{-1} }_2$ (because $\triplenorm{ \bP^T }_2=1$), one might consider directly minimizing the error $\norm{h_\mu - \cI_n\sbr{h_\mu}}_{\tt d}$ by minimizing the norm $\triplenorm{ (\bP^T \bV)^{-1} }_2$ . Indeed, whenever $\bV^\dagger \bV = {\bf 1}_{n \times n}$ we have exactly 
$\Lambda_{n,2} = \triplenorm{ (\bP^T \bV)^{-1} }_2$ which is more efficient to compute than $\triplenorm{  \bV (\bP^T \bV)^{-1}\bP^T  }_{2}$. Throughout the numerical experiments section~\ref{sec:numerics}, in addition to the error bound \eqref{eq:deim-error-L2} we 
monitor 
\begin{align} \label{eq:deim-error-L2_2}
  \| h_\mu - {\cal I}_n[h_\mu] \|_{\tt d} 
 \le  \triplenorm{ \bV }_2 \triplenorm{ (\bP^T \bV)^{-1} }_2 \norm{ h_{\mu} - \proF{n}{} h_\mu }_{\tt d} \, ,
\end{align}
which we have written here for future reference. 

\end{remark}

\begin{corollary}[RB-Greedy-DEIM error bounds]\label{cor:deim-rb-estimate}
The following practical RB-greedy-DEIM estimate holds
\begin{align*}
\| h_\mu - {\cal I}_n[h_\mu] \|_{\tt d} &\lesssim \Lambda_{n,2} \epsilon , \quad \forall h_\mu \in \cF .
\end{align*}
Furthermore, if the Kolmogorov $n$-width decays exponentially (with order $\alpha$) as in 
\eqref{eq:kolmogorov-width-exp}, then 
 the error \eqref{eq:deim-error-L2} decays with the same order,
\begin{align*}
\| h_\mu - {\cal I}_n[h_\mu] \|_{\tt d} \lesssim \sqrt{2C}  \Lambda_{n,2}  e^{-c_1 n^{\alpha}} \, ,
\end{align*}
whereas if the Kolmogorov $n$-width decays algebraically with order $\beta$ as in \eqref{eq:kolmogorov-width-alg}, 
then \eqref{eq:deim-error-L2} again decays with the same order $\beta$
\begin{align*}
\| h_\mu - {\cal I}_n[h_\mu] \|_{\tt d} \lesssim 2^{5\beta+1} C_a \Lambda_{n,2} n^{-\beta},
\end{align*}
where $C_a>0$ and $\beta>0$ are some constants. 
\end{corollary}
\begin{proof}
 Consider \eqref{sigma-epsilon} and \eqref{eq:deim-error-L2} and apply \eqref{eq:kolmogorov-width-exp} and 
  \eqref{eq:kolmogorov-width-alg}.  
\end{proof}
%

\subsubsection{\bf ROQ Convergence} \label{sec:rates_ROQ}

Next we present the error estimate for the case in which the reduced basis  
$\{ \widetilde{e}_\ell \}_{\ell=1}^m$ for $\widetilde{\cF}$ are generated through a direct approach via {\em Path \#1} in
\eqref{eq:greedy_for_products}.
\begin{theorem}[approximation error ROQ (direct approach)]   \label{T:approx-scalar-one-step-greedy}
 Let $\epsilon>0$ be an error tolerance and  $\proFtwo{m}{}: \widetilde{\cF} \to\widetilde F_m$ be the projection 
operator defined in Algorithm~\ref{algo:ROQ} ({\it Path \#1}), with accuracy $\epsilon > 0$ 
in the discrete $L^2$-norm $\norm{\cdot}_{\tt d}$.  Furthermore, let 
$\widetilde\cI_m: \widetilde{\cF} \to\widetilde F_m$ be a DEIM operator built using Step (3) (RB-greedy-DEIM)
of Algorithm~\ref{algo:ROQ}. 
Given two arbitrary functions $h_{\mu_i},h_{\mu_j} \in \cF$, the following estimate holds
\begin{align*}
      \left|  I_{\tt d}(i,j) - I_\mathrm{ROQ}(i,j) \right| \lesssim \epsilon \del{| \Omega |_d \triplenorm{\widetilde{\cI}_m}_2 }  \norm{ h_{\mu_i}^* h_{\mu_j} }_{\tt d} \, .
\end{align*}
\end{theorem}
\begin{proof}
We recall that given $h_{\mu_i}, h_{\mu_j} \in \cF$, with $\mu_i, \mu_j \in \cP$, we have a discrete quadrature rule $I_{\tt d}(i,j)$ 
given in Eq.~\eqref{eq:disc_full_L2prod}. On the other hand, we have the following expression for the
ROQ integral
\[
I_{\mathrm{ROQ}}(i,j) = \sum_{k=1}^M \omega_k \widetilde\cI_m[h_{\mu_i}^* h_{\mu_j}](x_k)    \, .                 
\] 
An application of Cauchy-Schwarz implies
\begin{align*}
  \left| I_{\tt d}(i,j) - I_{\mathrm{ROQ}}(i,j)  \right|
  &\le \sum_{k=1}^M \left| \omega_k \del{ h_{\mu_i}^*(x_k) h_{\mu_j}(x_k) - \widetilde{\cI}_m[ h_{\mu_i}^* h_{\mu_j} ](x_k) } 
                                       \right|   \\
  &\le | \Omega |_d \norm{  h_{\mu_i}^* h_{\mu_j} - \widetilde{\cI}_m[ h_{\mu_i}^* h_{\mu_j} ]  }_{\tt d} .
\end{align*}
Next the fact that 
$\widetilde{\cI}_m \left[ \widetilde\cP_m [h_{\mu_i}^* h_{\mu_j}] \right] = \widetilde\cP_m [h_{\mu_i}^* h_{\mu_j}]$, implies
\[
  h_{\mu_i}^* h_{\mu_j} - \widetilde{\cI}_m[ h_{\mu_i}^* h_{\mu_j} ]  =
    \del{ \mathbb{I} - \widetilde{\cI}_m } \left[   h_{\mu_i}^* h_{\mu_j} - \widetilde{\cP}_m[ h_{\mu_i}^* h_{\mu_j} ]  \right]
\]    
Finally using $\triplenorm{ \mathbb{I} - \widetilde{\cI}_m}_2 = \triplenorm{\widetilde{\cI}_m}_2$ 
(cf. \cite{JXu_LZikatanov_2003a,DBSzyld_2006a}), we get
\begin{align*}  
\left| I_{\tt d}(i,j) - I_{\mathrm{ROQ}}(i,j)  \right| 
     \le   | \Omega |_d \triplenorm{\widetilde{\cI}_m}_2 \sigma_{m}\del{\widetilde{\cF}; \cH} \norm{ h_{\mu_i}^* h_{\mu_j} }_{\tt d}  \, .        
\end{align*}
Then the greedy estimate {\em Path \#1} in Algorithm~\ref{algo:ROQ} gives the required estimate. 
\end{proof}

\begin{corollary}[ROQ error bound]\label{cor:roq-err-estimate}
The following practical ROQ error estimate holds for every $\mu_i , \mu_j  \in \cP$
\begin{align} \label{eq:roq-error}
\left|  I_{\tt c}(i,j) - I_\mathrm{ROQ}(i,j) \right| \lesssim \left| I_{\tt c}(i,j) - I_{\tt d}(i,j) \right| + 
\epsilon \del{ | \Omega |_d \triplenorm{\widetilde{\cI}_m}_2 }  \norm{ h_{\mu_i}^* h_{\mu_j} }_{\tt d} .
\end{align}
\end{corollary}

Our numerical examples indicate the ROQ built from a two-step greedy (using {\em Path \#2}) has an error bound as in 
Theorem~\ref{T:approx-scalar-one-step-greedy}, but proving such bound is beyond the scope of this paper.

\subsection{\bf ROQ at new quadrature points} \label{sec:ROQ_newPoints}
Consider the scenario in which the integrands are available at a set of points which is different from the ones used to build the reduced basis. 
A typical example would be the case in which Gaussian quadratures are used in the RB-greedy algorithm but data is given at equally spaced samples.
We continue denoting the set of points and quadrature rule used to build the basis as $\{ x_k, \omega_k \}_{k=1}^M$, and  the second rule $\{ y_i, \tau_i \}_{i=1}^{M'}$. Therefore, it is not unreasonable to assume that $M' > M$. Our goal here is to generate an ROQ with respect to  $\{ y_i, \tau_i \}_{i=1}^{M'}$ under this scenario.

\begin{algorithm}[Construction of ROQ for $\{ y_i, \tau_i \}_{i=1}^{M'}$]  $\quad$\\
$\left[\{ \widetilde{p}_\ell, \tau_\ell^\mathrm{ROQ} \}_{\ell=1}^m\right]$ 
         = ROQ-NEW $\left( \{ \widetilde{\mu}_\ell \}_{\ell=1}^m\text{,}\{ y_i, \tau_i \}_{i=1}^{M'} \right)$  
\begin{enumerate}[(1)]
  \item Orthogonalize the (ordered) product of functions $\{ g_{\widetilde{\mu}_\ell} \}_{\ell=1}^m$
           using the quadrature rule $\{ y_i, \tau_i \}_{i=1}^{M'}$ and let the resulting vectors form the columns of a matrix 
$\widetilde{\bV} = [ \widetilde{\be}_1, \dots, \widetilde{\be}_m ] \in \complex^{M' \times m}$ where, for example, the $\ell^\mathrm{th}$ column is $\widetilde{\be}_\ell = \del{ \widetilde{e}_\ell(y_1), \dots , \widetilde{e}_\ell(y_{M'}) }^T$.
  \item Apply Steps 3 and 4 from Algorithm~\ref{algo:ROQ} to generate the ROQ points 
$\{ \widetilde{p}_\ell \}_{\ell=1}^m \subset \{ y_i \}_{i=1}^{M'}$, interpolation matrix $\widetilde{\bP}$ and ROQ weights $(\boldsymbol{\tau}^\mathrm{ROQ})^T :=  \boldsymbol{\tau}^T \widetilde{\bV} ( \widetilde{\bP}^T \widetilde{\bV})^{-1}$.
\end{enumerate}
\end{algorithm}
Accuracy and conditioning of the resulting ROQ rule are guaranteed provided (i) we are able to carry out the orthogonalization in Step 1 up to a certain tolerance and (ii) the new quadrature rule is accurate enough such that $\norm{ h_\mu - \cP_n h_\mu }_\tau \lesssim \epsilon $.

\section{Numerical examples}   \label{sec:numerics}

Here we discuss two sets of numerical experiments.
As a first example we compare ROQ using orthogonal polynomials on the interval $[-1,1]$ as basis (that is, the basis is not built through a greedy approach) with the well known case of Gaussian quadratures and, in particular, integration of Runge's function. This study reveals several important features: i) ROQ are well conditioned, ii) ROQ and Gaussian quadrature rules have a similar point and weight distribution,
iii) the importance of an accurate quadrature rule $\{x_i,\omega_i\}_{i=1}^M$ (cf. Theorem~\ref{thm:exact-integration}),
iv) the ROQ rule's efficiency, compared to Gaussian quadratures, is expected to increase with the number of spatial dimensions.

Next we turn to a challenging example drawn from gravitational wave physics, namely the calculation of overlaps between ``chirp'' gravitational waves for compact binary coalescences. Details on implementation, error bounds, and a discussion of computational cost are provided, as well as the physical motivation of the problem. In particular, we find that ROQs exhibit exponential asymptotic convergence in the calculation of overlaps, with a profile of the error function in terms of quadrature points qualitatively similar to the greedy error curve of the training space in terms of the number of reduced basis elements. {\em We emphasize that this exponential convergence in computing overlaps is, as expected, present even when the function is sampled at equally spaced quadrature points. In particular, this means that reduced order quadratures can be used as an efficient down-sampling strategy when function evaluations from experimental observations are given at (for example) fixed intervals.We also show the staggering offline savings of our two-step greedy approach compared to a direct one.}

\subsection{\bf Comparison with Gauss-Legendre quadratures} \label{sec:Leg}

In order to highlight the essential features of our approach, in this subsection we use Legendre polynomials 
$\{P_\ell(x)\}_{\ell=1}^m$ defined on $x \in [-1,1]$ as an orthonormal basis (with respect to $W(x)=1$) instead of a reduced basis as 
generated in Step 1 Algorithm~\ref{algo:ROQ}.
Steps 3 and 4 of Algorithm~\ref{algo:ROQ} require specifying an underlying
quadrature rule, which in this subsection will be either an $M$ point Gauss-Legendre or extended trapezoidal rule. 
The resulting ROQ point and weight distribution (see Subsection~\ref{subsec:points_weights}), conditioning 
(see Subsection~\ref{subsec:ROQ_conditioning}) and application to Runge's function 
(see Subsection~\ref{subsec:Runges_function}) are considered next.

\subsubsection{\bf Point and weight distribution of ROQ} \label{subsec:points_weights}
To build the ROQ points and weights an extended trapezoidal rule is used and all $m=24$ Legendre polynomials are 
sampled at $M=1,000$ equidistant points. Fig.~\ref{fig:RBwgts} compares the distribution of points and weights for the Gauss-Legendre quadrature (red dots) and ROQ (blue crosses). Both weight and point distributions are similar, and in particular the points are not equally spaced but instead cluster towards the end points at $x=\pm 1$. 
One of the ROQ weights is equal to $-0.00496089441576999$ at $0.775775775775776$.
Negative weights can be problematic in quadrature rules due to roundoff errors leading to poorly conditioned formula~\cite{Quarteroni2010}. We turn to this issue next.

\begin{figure}[ht]
\centering
\subfigure[ROQ and Gauss-Legendre point and weight distributions.]{
  \includegraphics[width=0.48\linewidth]{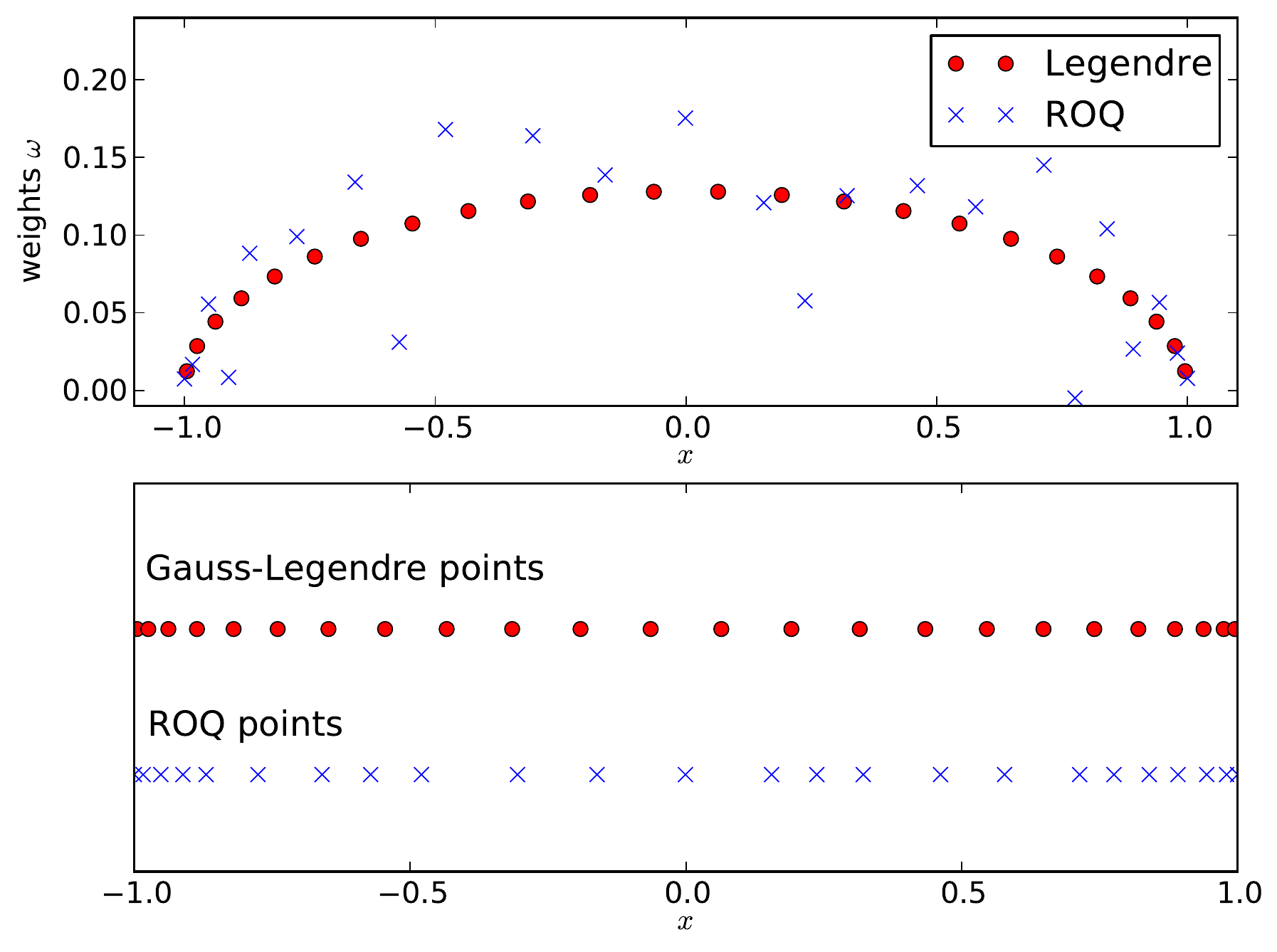}
  \label{fig:RBwgts}}
\subfigure[ROQ (solid) and Gauss-Legendre (dashed) condition number.]{
  \includegraphics[width=0.48\linewidth]{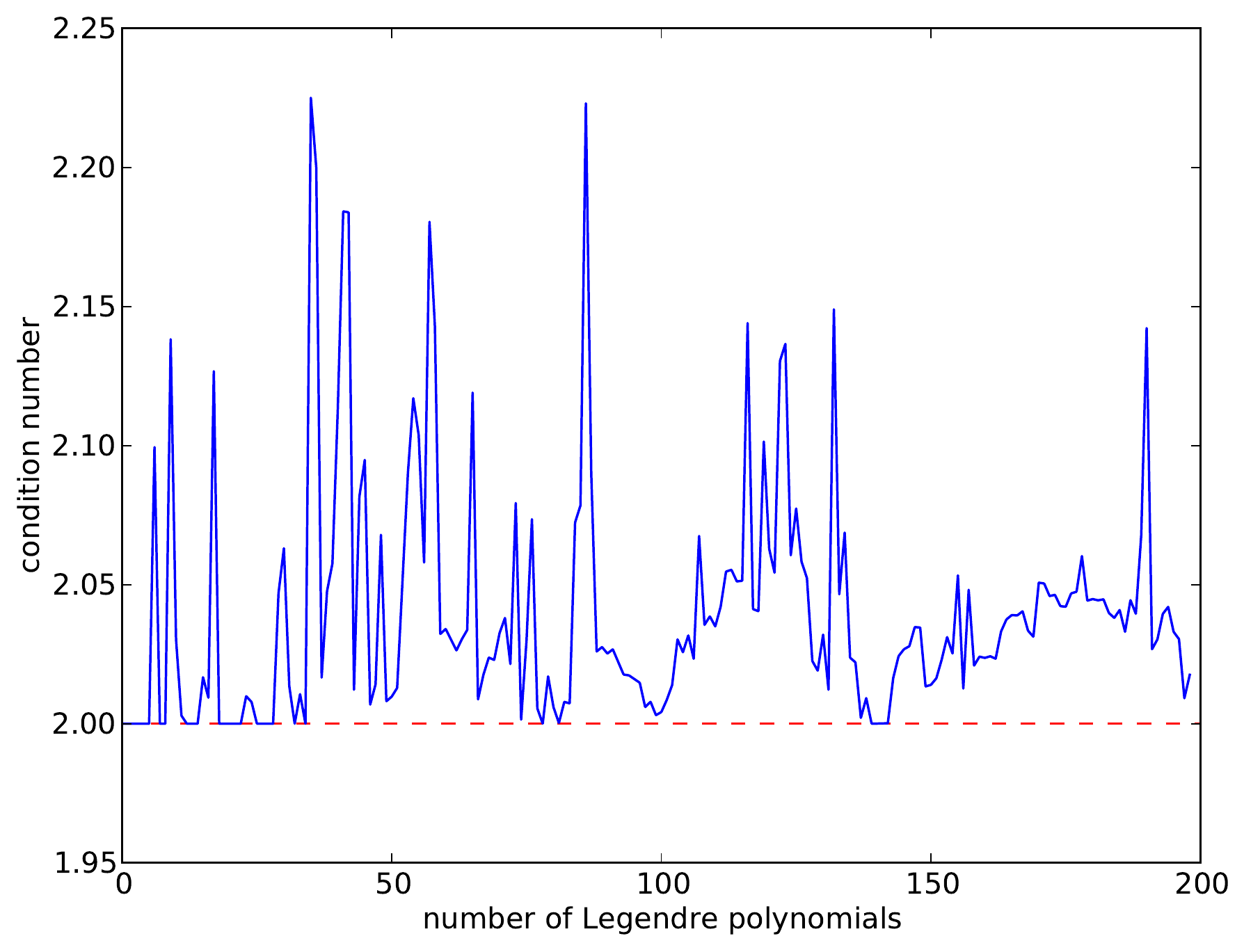}
  \label{fig:RBcond}}
\caption{The top left scatter plot shows the weight $\omega_k$ and point distributions for 
Gauss-Legendre and ROQ using $m=24$ Legendre polynomials as a basis (see Subsection~\ref{subsec:points_weights} 
for more details), whereas the bottom left one shows just the point locations. 
The right plot gives the condition number $\sum_{k=1}^m |\omega_k|$ for an $m$-point Gauss-Legendre and ROQ rule,
 with $m \in [2,200]$ (see Subsection~\ref{subsec:ROQ_conditioning} for more details).}
\label{fig:RBcondwgts}
\end{figure}

\subsubsection{\bf Conditioning of ROQ} \label{subsec:ROQ_conditioning}

Consider the absolute condition number $\sum_k |\omega_k|$ of a quadrature formula which, on the integration domain $x\in [-1,1]$, is exactly $2$ if all the weights $\omega_k$ are positive as can be seen by integration of unity. If some weights are negative then the condition number is necessarily larger than $2$. A classic example of ill-conditioning is the Newton-Cotes rule, where $\sum_k |\omega_k| > 2$ and grows without bound when using more than 8 nodes~\cite{Quarteroni2010}. Fig.~\ref{fig:RBcond} shows the absolute condition number $\sum_k |\omega_k|$ for a Gauss-Legendre (dashed red line) and reduced order (solid blue line) quadrature.
The dashed red line shows the optimal value of $2$, which is achieved by the Gauss-Legendre quadrature case. The condition number for ROQ remains below $2.25$ for the first 200 Legendre polynomials 
(see Subsection~\ref{subsec:points_weights}) and has no noticeable growth trend. 
This plot was generated using $M=1,000$ equally spaced points 
on $x\in [-1,1]$: this resolution is sufficient for demonstrating well conditioning, but the resulting ROQ integration rule would be of low accuracy. We have considered up to $100,000$ points, always obtaining behavior similar to that one of Fig.~\ref{fig:RBcond}.

\subsubsection{\bf Integration of Runge's function} \label{subsec:Runges_function}

\begin{figure}[ht]
\centering
\includegraphics[width=0.5\linewidth]{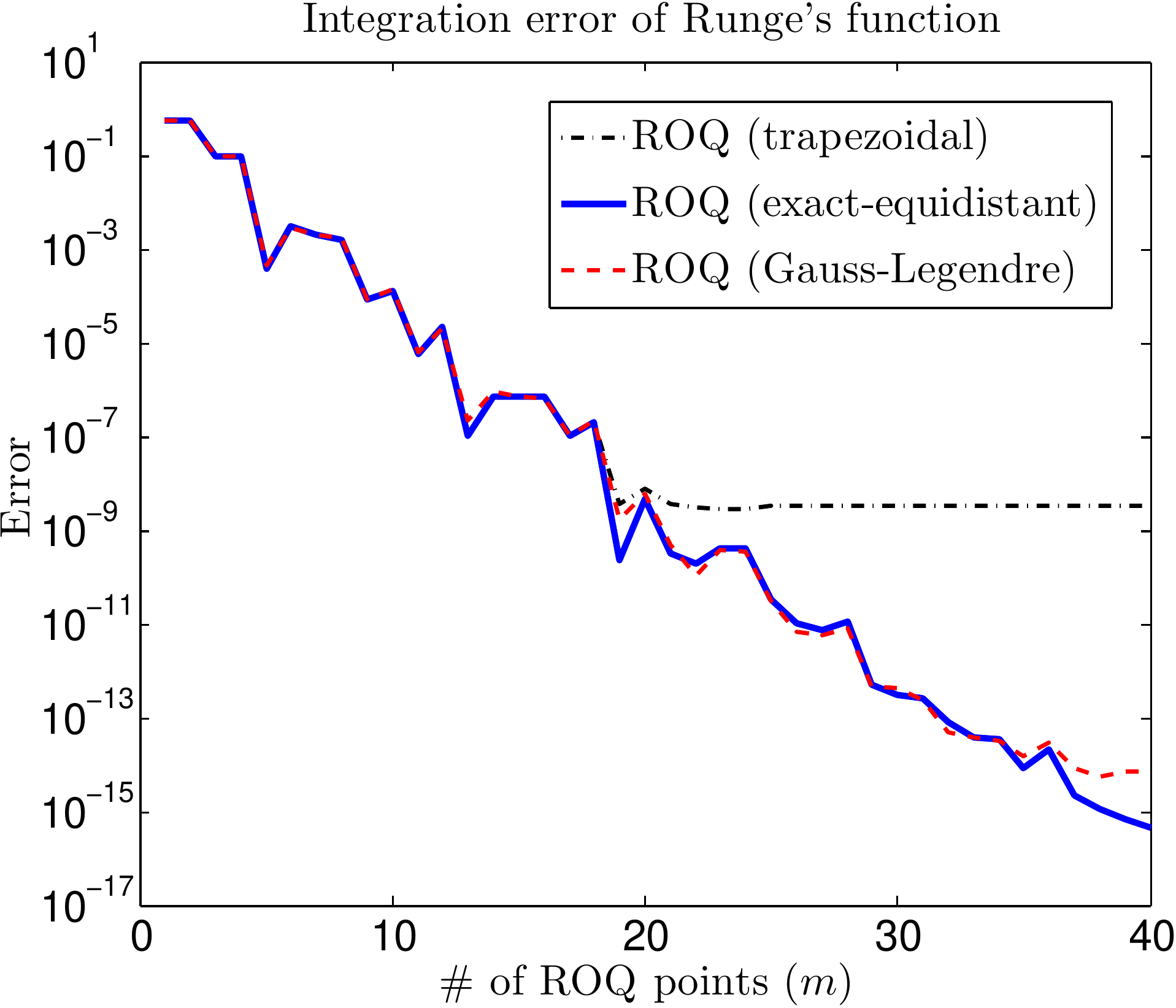}
\caption{
Convergence of the integral $\int_{-1}^1 \left(1+x^2 \right)^{-1} dx $ to $2 \tan^{-1} (1)$ using three different ROQ rules whose generation (using Legendre polynomials as basis) is described in Sec.~\ref{subsec:Runges_function}. Integration error from a ROQ-trapezoidal constructed with a $M_1=10,000$ point extended trapezoidal rule is given by the dash-dot black line. Due to the extended trapezoidal rule's inability to accurately integrate higher order Legendre polynomials this curve flattens out around $10^{-9}$. If we {\em exactly} integrate these basis functions by setting $\bomega^T \widetilde{\bV} = [\sqrt{2},0,\dots,0]$ in Eq.~\eqref{eqn:ROQ_weights} while continuing to use the ROQ-trapezoidal's quadrature points the accuracy is restored as seen from the ROQ-equidistant error curve (solid blue). Integration error from an ROQ constructed with a $M_2=400$ point Gauss-Legendre is given by the dashed red line.}
\label{fig:RungeInt}
\end{figure}

Many general features of ROQ can be ascertained by integration of Runge's function 
\begin{align*}
 \int_{-1}^1 \frac{1}{\left(1+x^2 \right)} \ dx = 2 \tan^{-1} (1) \, ,
\end{align*} 
which is a classic example of oscillatory errors due to high-order interpolation. In the following experiment, integration of Runge's function is carried out with three different ROQ rules built from a basis consisting of $m$ Legendre polynomials. In each case the Legendre polynomials are evaluated at $M_1 = 10,000$ equally spaced and $M_2=400$ Gauss-Legendre points, and the integration errors are computed as $|2 \tan^{-1} (1) - I_\mathrm{ROQ}|$.

First, integration of Runge's function is carried out using an $m$-point ROQ-trapezoidal rule which is built from an $M_1$-point extended trapezoidal rule and the $m$ DEIM point selections. The resulting integration error is given by the dash-dot black line in Fig.~\ref{fig:RungeInt}. Notice that the ROQ-trapezoidal error curve flattens out around $10^{-9}$. Beyond this value the $10,000$ point trapezoidal rule is unable to accurately integrate the Legendre polynomials of degree higher than $\sim 20$, thereby limiting the ROQ accuracy (see Theorem~\ref{thm:exact-integration}). 

Next we show how accuracy may be restored in this particular case. Notice that the factor $\bomega^T \widetilde{\bV}$ in Eq.~\eqref{eqn:ROQ_weights} for the ROQ weights is precisely the numerical integration of $m$ Legendre polynomials. Furthermore, we know that the first Legendre polynomial integrates to $\sqrt{2}$ while all others integrate to zero. Hence, in this particular case we may integrate them ``by hand" by setting $\bomega^T \widetilde{\bV} = [\sqrt{2},0,\dots,0]$, while continuing to generate interpolation points by way of the DEIM algorithm. The resulting ROQ rule, denoted in the figure as ROQ-equidistant, is then used to integrate Runge's function with an error given by the solid blue line in Fig.~\ref{fig:RungeInt}. This highlights the importance of accurate integration of the basis functions which, in a general setting where the basis vectors are selected by a greedy algorithm, should be carried out with a good underlying quadrature rule. 

Lastly, an ROQ rule is constructed from a $M_2=400$ point Gauss-Legendre rule (dashed red line of Fig.~\ref{fig:RungeInt}) and shows excellent agreement with the ROQ-equidistant case. Of particular noteworthiness, all three ROQ rules clearly show exponential convergence with the number of quadrature points and no apparent issues stemming from Runge's phenomena. Notice that the dashed red and solid blue lines are nearly indistinguishable, confirming the expectation that if the basis functions are sampled densely enough (a statement which certainly depends on $m$) and integrated accurately enough by the underlying quadrature rule, the resulting rules perform similarly.

\subsubsection{\bf Higher dimensional integrals} \label{subsec:ROQ_2D}

Let $d\ge 1$ be the space dimension and consider the integrals
\begin{align} \label{eq:Int2d}
I_1 = \int_\Omega \left[ \left( x - \mu_1 \right)^2 + 0.1^2 \right]^{-1/2} \ dx \, , \quad
I_2 = \int_\Omega \left[ \left( x - \mu_1 \right)^2 + \left(y - \mu_2 \right)^2 + 0.1^2 \right]^{-1/2} \ dx \ dy \, ,
\end{align}
for the following two cases: i) $d=1$, $\Omega = [-1,1]$ and $\mu_1 \in  [-.1,.1]$ and ii) $d=2$, $\Omega = [-1,1]^2$ and $\mu = (\mu_1,\mu_2) \in [-.1,.1]^2$. In the first case , numerical integration is carried out using an $M_1 \le 150$ point Gauss-Legendre rule and an ROQ rule built from the $150$-point GQ rule. For the second case, numerical integration is carried out using an $M_2 \le 150^2$ point tensor product Gauss-Legendre rule and an ROQ rule built from the $150^2$-point GQ rule. Basis vectors are selected by the greedy algorithm in both cases. From Fig.~\ref{fig:ROQ_2D} one can see that ROQ provide a factor of $\sim 4$ savings in the $1$-dimensional case for a maximum error below $10^{-4}$, while the savings in the $2$-dimensional case is greater than $\sim 12$. For many problems the expected ROQ savings will continue to increase when compared to quadrature rules arising from tensor product grids whose computational cost (also the cost incurred while building the basis) scales like $M^d$ for $d$ spatial dimensions. On the other hand, an ROQ nodal set is formed by scattered point distributions tailored to the problem. 

\begin{figure}[ht]
\centering
\includegraphics[width=0.5\linewidth]{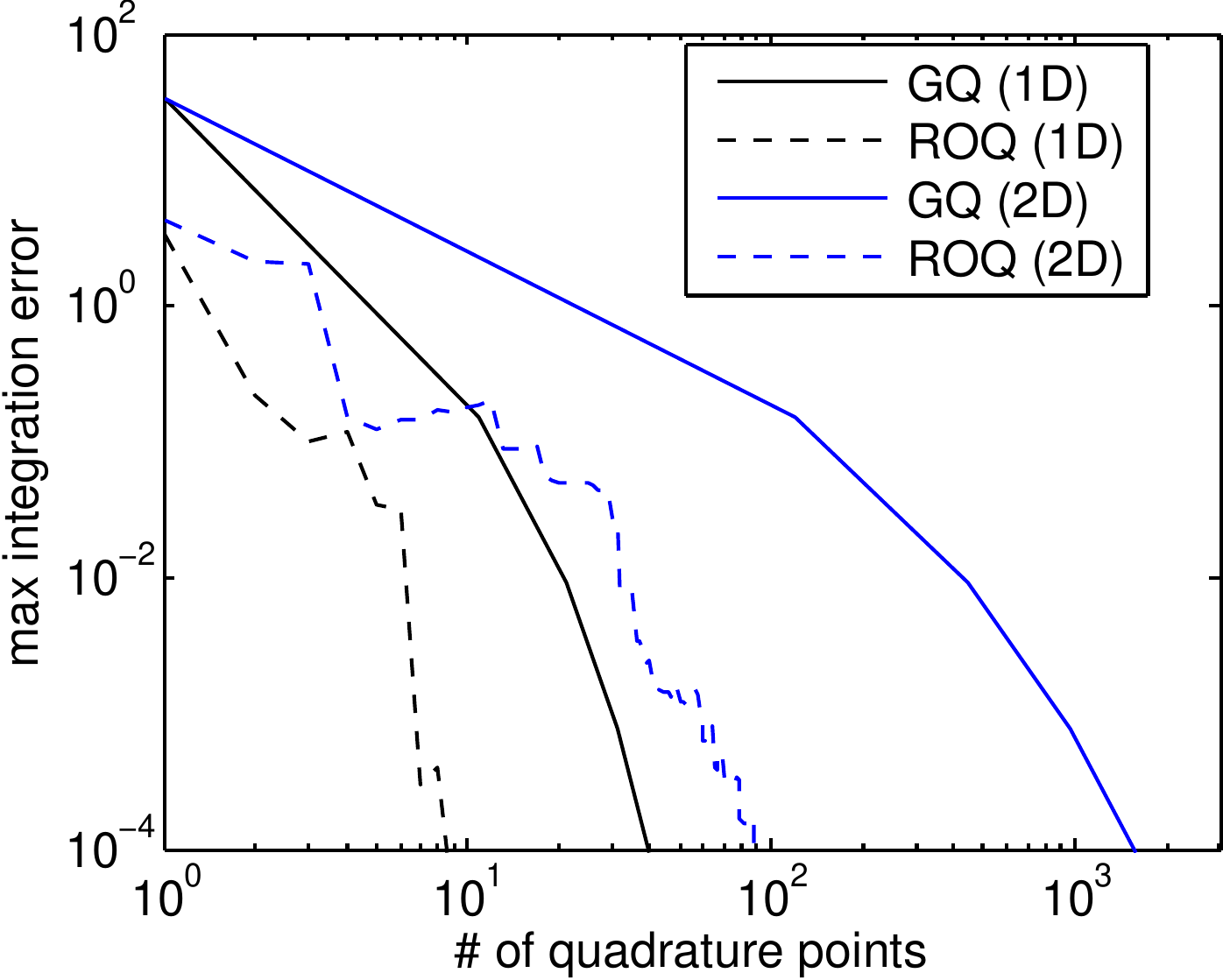}
\caption{
Error curves for the $1$-dimensional (case i) and $2$-dimensional (case ii) numerical integrals of Eq.~\eqref{eq:Int2d} using Gauss-Legendre and ROQ rules. Errors are computed by taking a maximum over the entire training set. The ROQ savings increase from $\sim 4$ to $\sim 12$ as the number of spatial dimensions is increased from one to two. As discussed in Sec.~\ref{subsec:ROQ_2D}, further savings are expected as the number of spatial dimensions increase.}
\label{fig:ROQ_2D}
\end{figure}

\subsection{\bf Gravitational waves} \label{sec:SPAWaveforms}

The inspiral and merger of neutron stars (NSs), black holes (BHs) or mixed pairs, known generically as compact binary coalescences, is believed to be one of the main sources of gravitational waves to be detected by the upcoming generation of earth-based observatories~\cite{Barish:1999vh,Waldmann:2006bm,Hild:2006bk,Acernese:2006bj,Abbott:2007kva,Abadie:2010cfa}. A passing wave distorts the length between any two points separated by a non-zero spatial distance. These distortions are measured by gravitational wave detectors as an effective strain $h_\mu (x)$, where $x$ denotes either time or frequency. In this problem the parameter $\mu$ depends on quantities such as the mass and spin of the compact objects, or detector orientation. In principle $h_{\mu}$ is obtained by numerically solving Einstein's equations, which are a set of parameterized quasilinear hyperbolic-elliptic equations \cite{Centrella:2010mx,Pretorius:2007nq,Baumgarte:2002jm,Sarbach:2012pr}. However, due to the cost of these simulations and the properties of the target system, simplified models are often used. 

Here we focus on testing our proposed ROQ on 
the post-Newtonian (PN) approximation \cite{Blanchet:LRR}, under which the gravitational waves in the leading order Stationary Phase Approximation (SPA) are given by \cite{Blanchet:1995fg,Will:1996zj,Allen:2005fk},
\begin{align} \label{eq:grav_wave_closed_form}
h_{\cM_c}(f) &=  \cA f^{-7/6}  \cdot \mbox{exp}\left(i  \left\{ - \frac{\pi}{4} + \frac{3}{128}\Big(\pi \cdot \frac{G}{c^3} \cdot f \cdot \cM_c \Big)^{-5/3}  \right\}  \right) ,
\end{align}
where $f$ denotes the frequency, $G$ Newton's gravitational constant \cite{Tu:2010zz}, $c$ the speed of light, and $\cA$ an overall amplitude that depends on quantities such as the distance to the compact binary coalescence. We will refer to Eq.~\eqref{eq:grav_wave_closed_form} as a gravitational {\em waveform}, two representative examples are shown in Fig.~\ref{fig:GW}. The single parameter $\mu$ in this model is the {\it chirp mass} $\cM_c:=(m_1 m_2)^{3/5} (m_1+m_2)^{-1/5}$~\cite{Allen:2005fk}, with $m_i$ the individual mass of each object. All quantities here use MKS units. 

\begin{figure}[ht]
\centering
\includegraphics[width=0.46\linewidth,height=.4\linewidth]{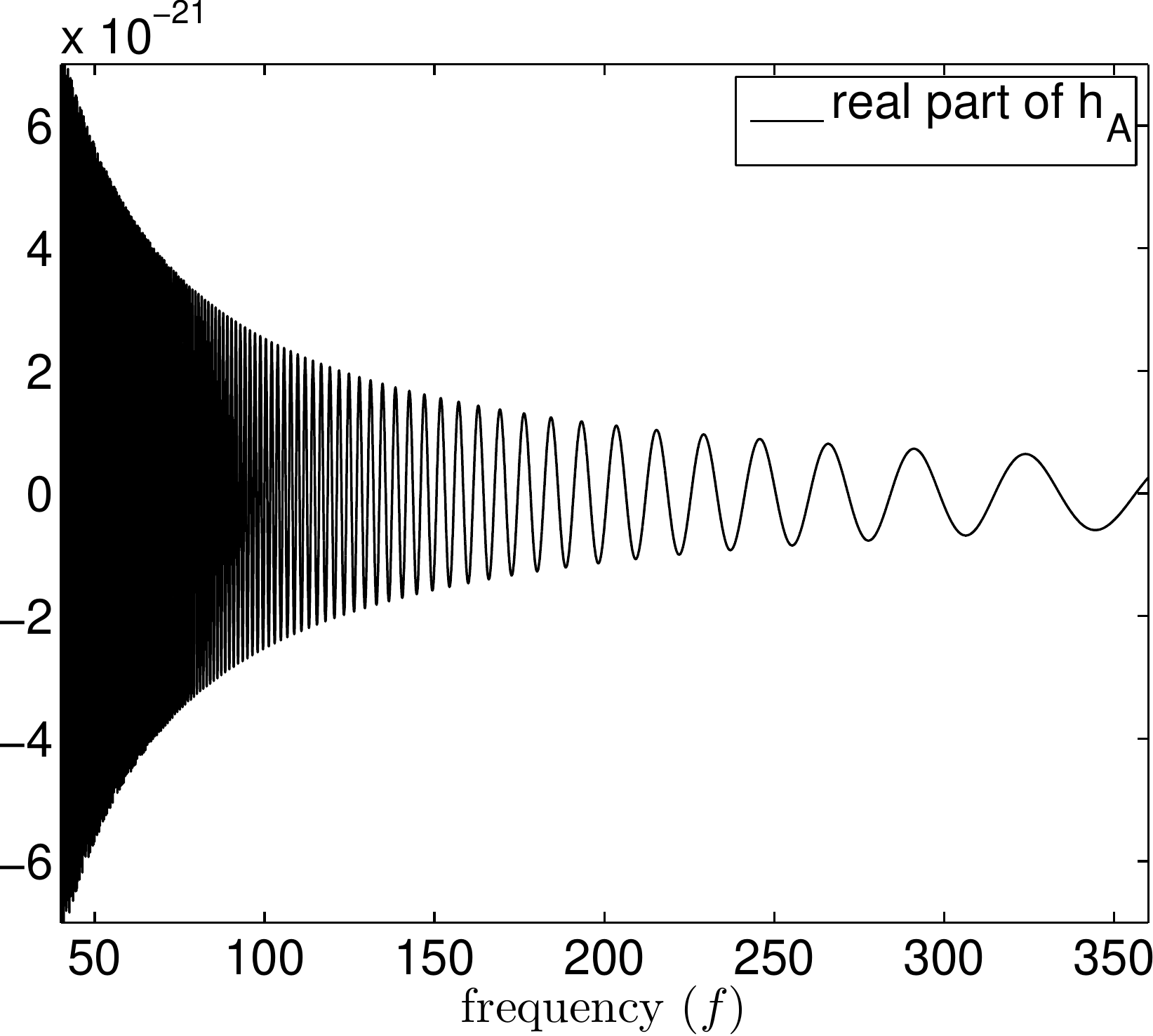} \qquad
\includegraphics[width=0.46\linewidth,height=.4\linewidth]{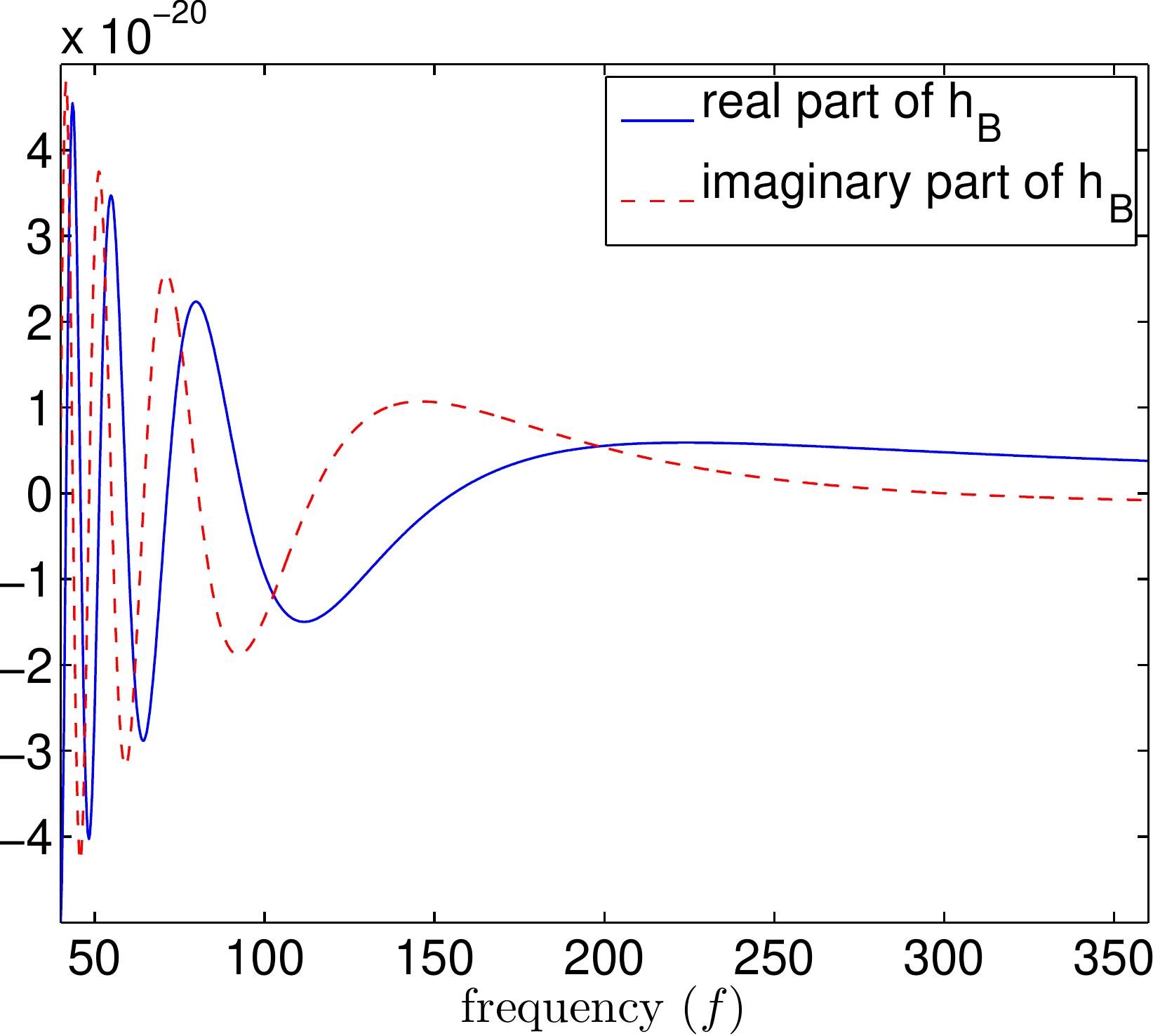}
\caption{Gravitational waveforms $h_A$ (left) and $h_B$ (right) correspond to the smallest ($A = 2.611651689888372 M_{\odot}$) and largest ($B=26.11651689888372 M_{\odot}$) chirp masses of the training set $\cT_K$ defined in \eqref{eqn:GW_training}. Waveforms are complex valued, but for clarity we do not show the imaginary part of $h_A$.}
\label{fig:GW}
\end{figure}

There are various reasons one might need to repeatedly compute inner products between waveforms. One setting where this occurs is in gravitational wave searches using matched filtering ~\cite{Owen_B:96,Owen99,2006CQGra..23.5477B,Allen:2005fk}. Given data $s$ recorded by a detector, a matched filtering search is carried out by computing all possible inner products of the data $s$ against waveforms drawn from a (usually large) catalog. It can be shown~\cite{Finn:1992wt,Maggiore} that the optimal matched filtering strategy requires one to compute inner products with a weight given by the reciprocal of the detector's power spectral density $S(f) = W^{-1}$. As a representative example we use $S(f)$ from the initial Laser Interferometer Gravitational Wave Observatory (LIGO)~\cite{Abbott:2007kv}, which can be modeled by the curve \cite{Ajith:2009fz}
\begin{align} \label{eq:PSD_LIGO}
S(y) = 9\times 10^{-46} \left[ \left(4.49 y \right)^{-56} + 0.16 y^{-4.52} + 0.52 + 0.32 \cdot y^2 \right], \quad y = \frac{f}{150 Hz} .
\end{align}
To ensure that no signals are missed and that the parameters are correctly estimated many inner products are needed for each segment of data~\cite{Field:2011mf,PhysRevD.79.022001,Brady:1998nj,Vallisneri:2011ts}. These significant computational costs motivate a need for faster ways of performing these integrals, often in real-time \cite{Kanner:2008zh,2010AAS...21540606S,Cannon:2011vi}.

In the following series of experiments we consider gravitational waves of the form given by Eq.~\eqref{eq:grav_wave_closed_form}, generated by inspiralling binary black hole (BBH) systems with mass components in the interval $m_i \in [3,30] M_{\odot}$, where $1 M_{\odot} = 1.98892 \times 10^{30} Kg$ is the astrophysical unit of mass known as a {\em solar mass}. Therefore the parameter domain under consideration is $ \cP = [A,B]$, where $A = 2.611651689888372 M_{\odot}$ and $B=26.11651689888372 M_{\odot}$. The physical domain frequency interval is given by $ \Omega = [40, 366.3383434841933] $\,Hz. Here the lower value is set by the Earth's seismic noise which decreases the detector's sensitivity. The upper bound is set by the limitations of the PN approximation which breaks down at frequencies corresponding to the {\it Innermost Stable Circular Orbit}. While this is a mass dependent frequency, for simplicity we choose the upper limit as the maximum over the mass range here considered.

Our primary goal in the forthcoming sections is to describe the construction and performance of an ROQ rule for computing inner products between gravitational waveforms given by Eq.~\eqref{eq:grav_wave_closed_form}. These steps are conveniently summarized in Algorithm \ref{algo:ROQ}. The first step, an approximation for $\widetilde{\cF} \approx \widetilde{F}_m$, can be carried out either by means of a direct greedy (Path \#1) or two-step greedy (Path \#2). Due to a {\em significantly} reduced offline cost we advocate the two-step greedy approach, which in turn requires an intermediate approximation $\cF \approx F_n$. Results for this intermediate approximation are considered in Section~\ref{sec:RB_DEIM_waveforms}. We contemporaneously remark on the performance of the DEIM algorithm when applied to single functions to facilitate a better understanding of this algorithm in a simpler setting. In Sections~\ref{sec:Greedy_comparison} and~\ref{sec:ROQ_for_waveforms} we complete the second half of the two-step greedy approximation, identify the interpolation points and finally generate an ROQ rule. In this section we also answer the fundamental question: {\it Do the direct and two-step greedy approaches give comparable results and is the former much more expensive than the latter?} As we will see, we have strong numerical evidence that, as predicted by our estimates, the two-step greedy gives enormous computational savings without sacrificing the accuracy or compactness of the reduced basis.

\subsubsection{\bf RB-greedy and DEIM for single functions} \label{sec:RB_DEIM_waveforms}

Here we consider experiments focused on approximations of the set of functions $\cF = \{ h_{\cM_c}(f) : \cM_c \in \cP \}$; with $h_{\cM_c}$ as in  Eq.~\eqref{eq:grav_wave_closed_form} and the range $\cP = [A,B]$ for the chirp mass $\cM_c$ --serving here as the parameter $\mu$-- as described above. We find the DEIM points and numerically confirm the DEIM error bounds from Sec.~\ref{sec:rates_ROQ}. This subsection is also precursor to Sections~\ref{sec:Greedy_comparison} and \ref{sec:ROQ_for_waveforms}, where ROQs are constructed for waveform inner products.

\smallskip
\noindent {\bf Reduced Basis:} 
We begin by building a reduced basis space $F_n$ approximating $\cF$ by using the RB-greedy approach. Since the number of cycles is proportional to an inverse power of $\cM_c$, we have found it advantageous to populate the training space with a logarithmic spacing between the samples, thereby clustering more points at low values of $\cM_c$. More precisely, a training set of size $K$ is here given by 
\begin{align} \label{eqn:GW_training}
\cT_K = \left\{ A\left(\frac{B}{A}\right)^\frac{i}{K-1} \ | \  i = 0,\dots,K-1 \right\} \, ,
\end{align}
and an associated training space of normalized waveforms 
\[
     \cF_{\mathrm{train}} =  \{ h_{\cM_c}(f) \ | \ \cM_c \in \cT_K \}. 
\]     
Through numerical experiments we have found that for this problem the number of basis for any given greedy error saturates with at most $K = 3,000$ training space elements, with -- for example -- $178$ RB elements $\{ e_i \}_{i=1}^{178}$ needed to achieve a tolerance\footnote{Our numerical experiments are carried out with double precision arithmetic. This translates into a double precision computation of the quantity $\left\|  h_{\mu} - \proF{n}{} h_\mu \right\|_{\tt d}^2$ found in step 3b of 
Algorithm~\ref{algo:RB-greedy} (RB-Greedy Algorithm), and hence an accuracy of about $10^{-7}$ in the computation of its square root. These observations motivate a choice for the greedy error tolerance to be $\epsilon = 10^{-6}$. Refs.~\cite{Casenave2012539,Canuto:2009:PEA:1654814.1654832} address this issue in greater detail and propose alternatives for improving error computations.} of $\epsilon^2 = 10^{-12}$. Therefore, the results shown below use $3,000$ training space samples. Fig.~\ref{fig:mchirpshist} shows the distribution of points selected by the greedy algorithm. They cluster at low $\cM_c$, which corresponds to more lower-frequency oscillations in Eq.~\eqref{eq:grav_wave_closed_form}. Such clustering is expected, since the number of cycles in a frequency range $f\in [f_{min},f_{max}]$ is given to lowest post-Newtonian approximation order by (graphically seen from Fig~\ref{fig:GW})
$$
{\cal N}_\mathrm{cycles} \left( \cM_c \right) = {\cal N}(f_{min};\cM_c)-{\cal N}(f_{max};\cM_c) \, , 
$$ 
where (see Eq.~(4.23), and Eq.~(5.247) from Ref.~\cite{Maggiore})
$$
{\cal N}(f;\cM_c)=1/(32\pi^{8/3})\left(G\cM_c/c^3\right)^{-5/3} f^{-5/3} \, . 
$$
The truly interesting aspect is that the distribution of greedy points closely matches the functional form of ${\cal N}_\mathrm{cycles} \left( \cM_c \right)$.

The solid black lines labeled ``projection error" in both plots of Figure \ref{fig:EIMInt} show the square of the greedy error $\greedy{n}{M}(\cF_{\mathrm{train}};\cH)$, defined in Eq.~\eqref{eq:RB-error}, over the training space $\cF_{\mathrm{train}}$, as a function of the number of RB elements. After a slow decay, the error decays with a very fast exponential falloff, a feature that we have found in this family of inspiral gravitational waveforms and generalizations thereof 
\cite{Field:2011mf,PhysRevD.86.084046}. 

\begin{figure}[htp]
\centering
\subfigure[Chirp Mass $\cM_c$]{
  \includegraphics[width=0.48\linewidth]{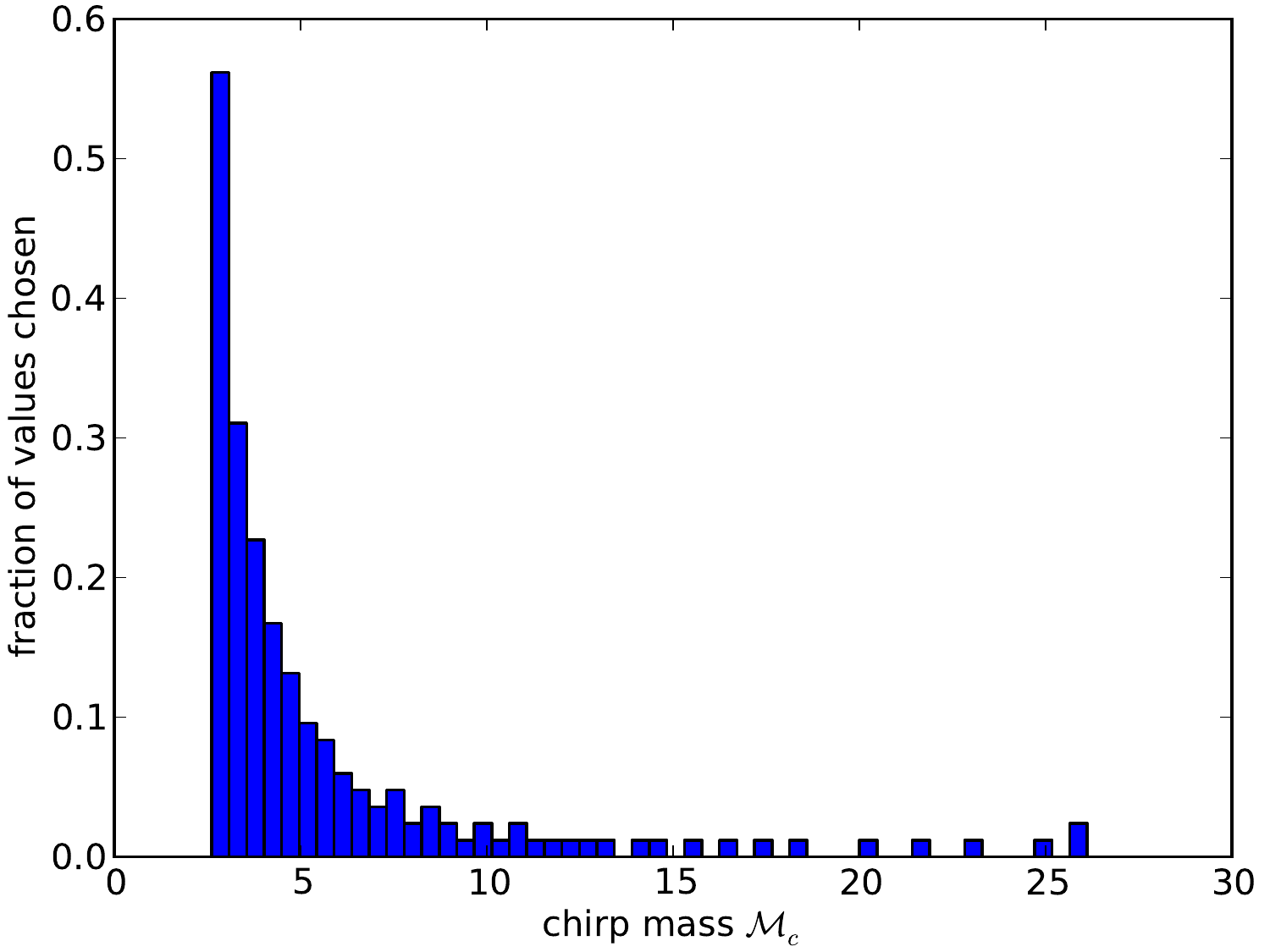}
  \label{fig:mchirpshist}}
\subfigure[Frequency $f$]{
  \includegraphics[width=0.48\linewidth]{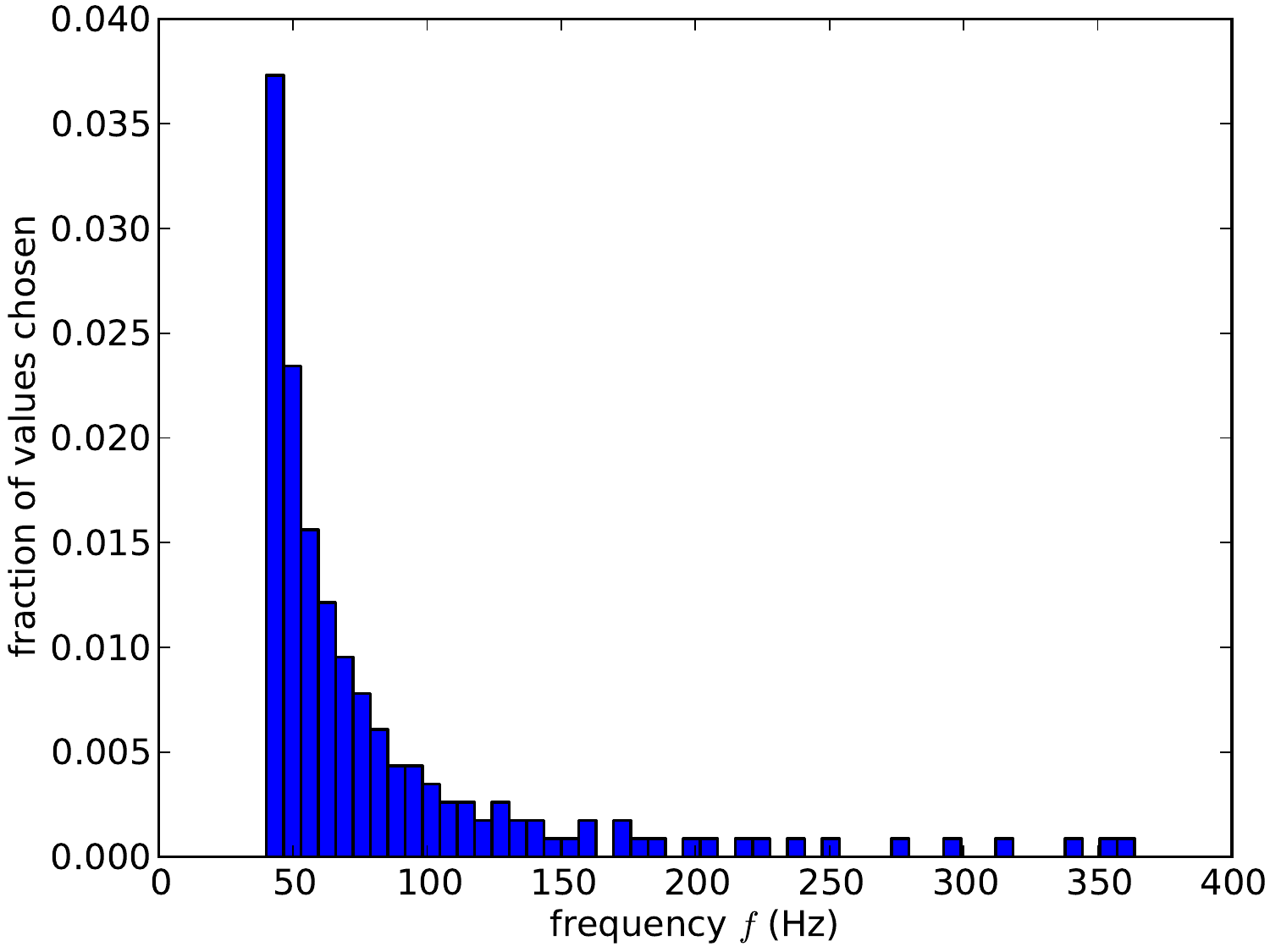}
  \label{fig:frqshist}}
  \caption{
  The left figure~\subref{fig:mchirpshist} shows the parameters $\cM_c$ selected by greedy Algorithm~\ref{algo:RB-greedy}. The right figure~\subref{fig:frqshist} shows the frequency points selected by DEIM Algorithm~\ref{algo:DEIM} out of $20,000$ equidistant frequency sample points.
We see from the histograms that both the chirp masses selected and the frequency interpolating points cluster at small values. This is intuitively expected, because smaller chirp masses correspond to a larger number of waveform cycles in Eq.~\eqref{eq:grav_wave_closed_form}. 
}
\label{fig:mchirpfrqhist}
\end{figure}

\smallskip
\noindent {\bf Empirical Interpolation:} 
Since the reduced basis set $\{ e_i \}_{i=1}^{178}$ generated above is application-based (as opposed to, say, Jacobi polynomials), an appropriate set of interpolation points is {\em in principle} not known, and that is where the DEIM algorithm of Section~\ref{sec:EIM} enters the game. We now generate a hierarchical set of DEIM points for the reduced basis built (these results are not used in the construction of ROQ for inner products), choosing the maximum number of interpolating points which equals the dimension of the RB space. That is, using $n=1,2,3, \ldots ,178$ reduced basis elements, we sequentially computed the equivalent number of DEIM points.  We recall that the RB-greedy approach selects points in parameter space, in this case the chirp mass, while the DEIM generates interpolating points in physical space (here frequency), and that both methods are hierarchical. That is, a seed choice for the first parameter value defines the first RB element, and from it the first interpolating point can be computed. Next, the greedy method chooses a second parameter value and the second RB element. By applying the DEIM to this new set (of so far two basis elements), the second interpolating point can be computed. And so on. Equivalently (the output is exactly the same), one can generate the whole RB first, up to the desired representation error tolerance, and then generate all of the interpolation points. Put differently, if there are $n$ RB elements, up to $n$ DEIM points can be generated, and generating some $n_1 < n$ DEIM only requires the first $n_1$ RB elements. This discussion, though perhaps trivial, might be helpful to keep in mind when later discussing the results from Fig.~\ref{fig:EIMInt}. 

As input to the DEIM algorithm we must provide the RB functions sampled on some set of frequency points. Two cases are here considered: i) the basis vectors are sampled at $1,701$ Gauss-Legendre points, which was the integration rule used to build the reduced basis in Subsection~\ref{sec:RB_DEIM_waveforms}, ii) the basis vectors are sampled at $20,000$ equidistant points \footnote{See the discussion in Section \ref{sec:ROQ_newPoints} related to the subtleties involved when sampling the RB functions at a set different from those used to compute the underlying quadratures to build the RB itself.}. Figure~\ref{fig:frqshist} depicts the frequency distribution of selected points for the equidistant sampling case, which is of particular practical importance when one seeks to downsample experimental data. The overall structure for the Gauss-Legendre points case was found to be essentially identical, and is therefore not shown. 

Next, we randomly pick $10,000$ waveforms, not necessarily in the training space, and represent each of them as a DEIM interpolant (that is, using Eq.~\eqref{eq:eim_disc}). Both sampling at Gauss-Legendre and equidistant points are considered. For each waveform the DEIM error is calculated  
as $\| h_{\mu} - {\cI}_n[h_{\mu}] \|^2_{\tt d}$. Fig.~\ref{fig:EIMInt} compares the largest DEIM error over all $10,000$ waveforms (solid blue line) with the greedy error (solid black line). Notice that the black line lies strictly below the blue line; for waveforms in the training space this is guaranteed by the optimality of the $L^2$ orthogonal projection $\proF{n}{} h_\mu$. In turn, the DEIM interpolation error is bounded, again as expected, from above by the a-priori theoretical error estimates given in by Eq.~\eqref{eq:deim-error-L2} (red line) and Eq.~\eqref{eq:deim-error-L2_2} (purple line). 
In principle these two error estimates strictly hold only for waveforms in the training space, but our numerical results show that
they evidently continue to hold for waveforms between elements of the training space when the latter is sufficiently dense. The red line provides a sharper bound on the error but it is also more expensive to calculate (cf.~Remark \ref{remark:DEIM_bounds}). Finally, notice that, at least for elements in the training space, the asymptotic convergence rate of the DEIM is predicted to be bounded by the RB-greedy representation error (with the Lebesgue constant as a proportionality constant), see Corollary \ref{cor:deim-rb-estimate} . From the figure we can see that the similar asymptotic convergence rates for the DEIM and reduced basis representation continue to hold for waveforms not necessarily in  the training space. 

\begin{figure}[ht]
\includegraphics[width=0.5\linewidth]{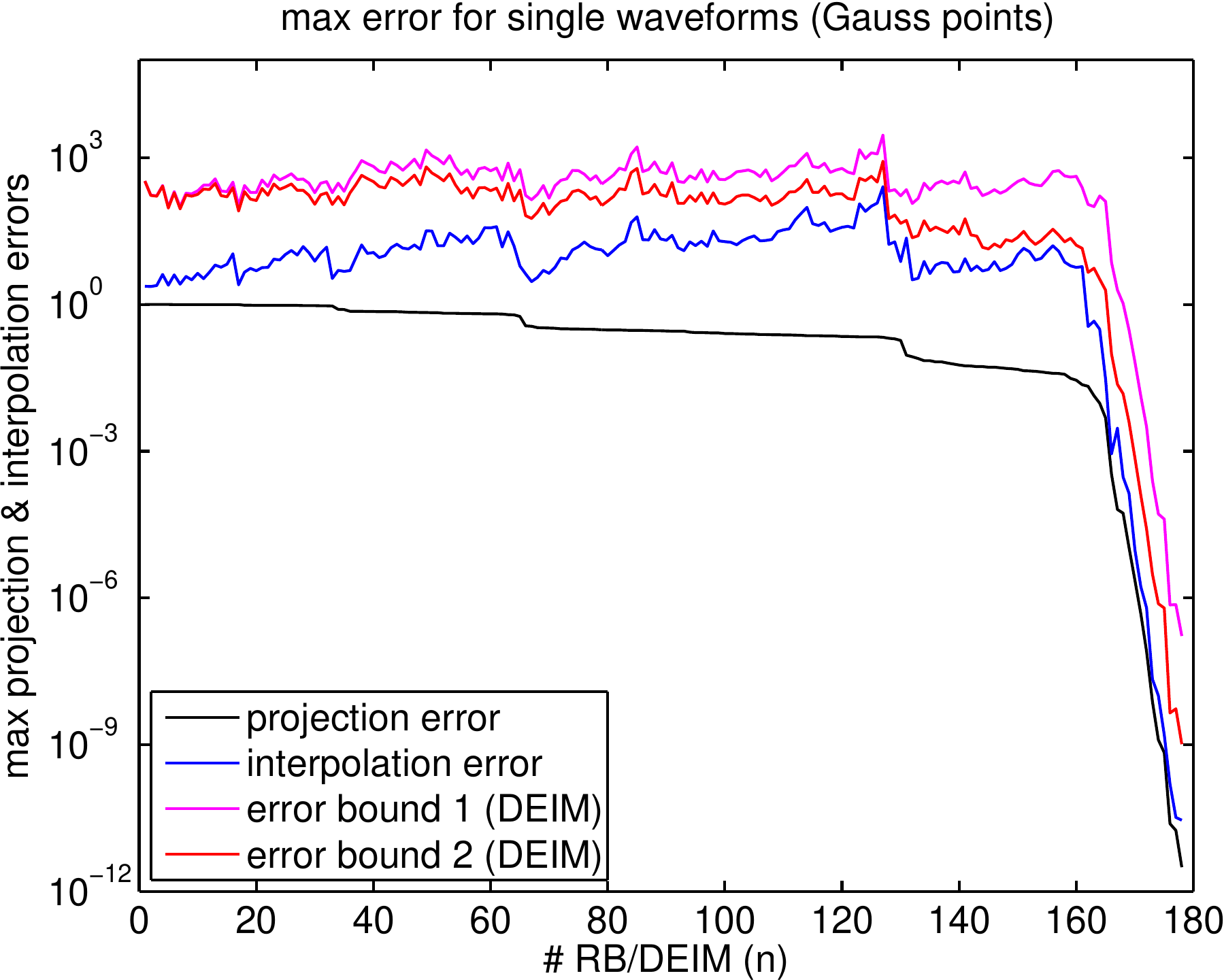}
\includegraphics[width=0.5\linewidth]{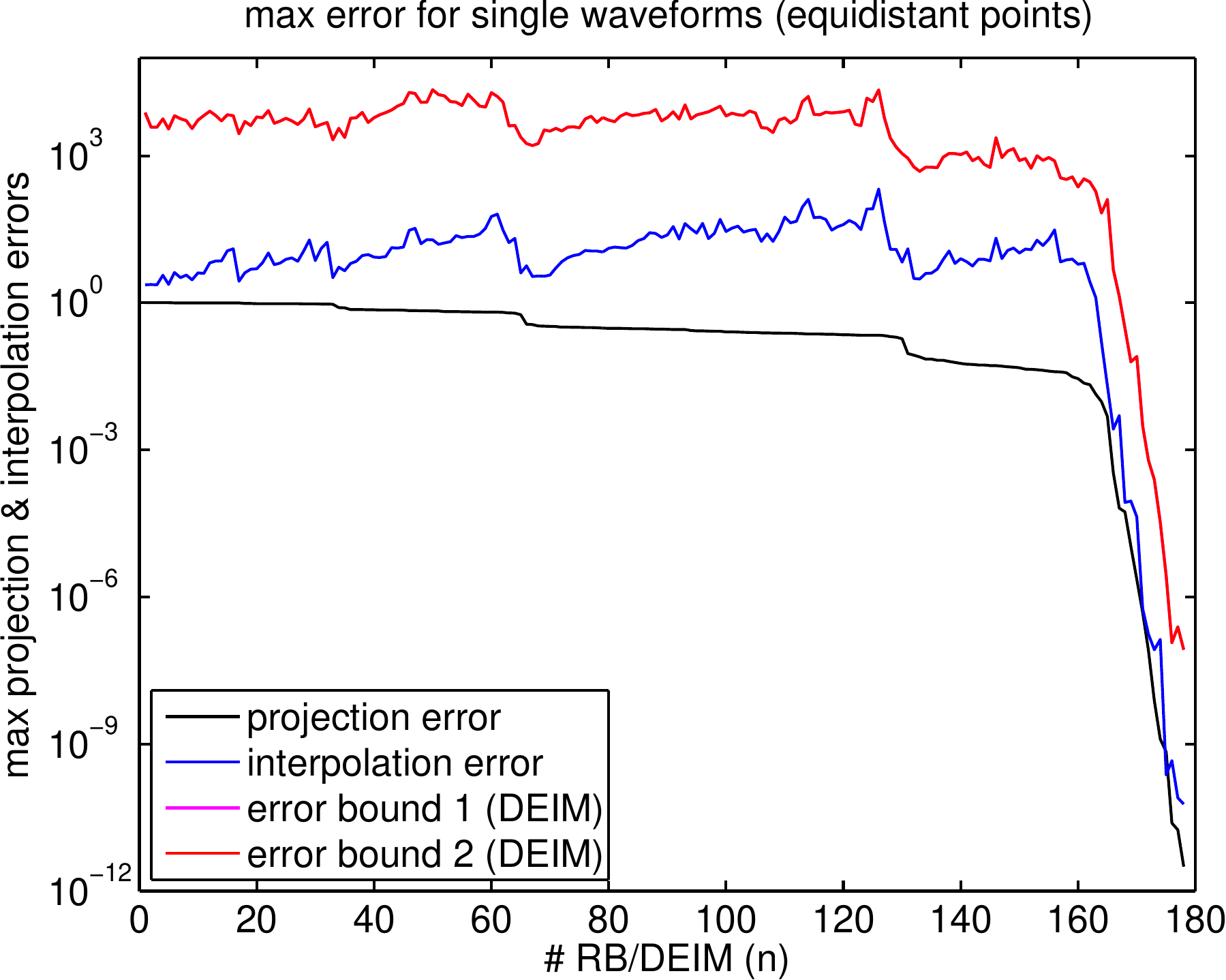}
\caption{Maximum projection and interpolation errors, as well as the predicted error bounds, for single waveforms as a function of the number of reduced basis (equal to the number of DEIM points) used. In both figures ``projection error" plots the square of the greedy error $\greedy{n}{M}(F_\mathrm{train};\cH)$ and ``interpolation error" plots a maximum error $\| h_{\cM_c} - \cI_n\sbr{h_{\cM_c}}\|_{\tt d}^2$ taken over $10,000$ randomly drawn values of $\cM_c$ not necessarily in the training space. The DEIM interpolant error curves found in the left (right) figures correspond to the Gauss-Legendre (equidistant) cases described in Sec.~\ref{sec:RB_DEIM_waveforms}. Error bounds given by Eq.~\eqref{eq:deim-error-L2_2} (error bound 1, magenta) and Eq.~\eqref{eq:deim-error-L2} (error bound 2, red) are identical up to numerical accuracy when equidistant samples are used (cf.~remark \ref{remark:DEIM_bounds}). For both cases the maximum interpolation error is nearly identical, as expected.}
\label{fig:EIMInt}
\end{figure}

\subsubsection{\bf Direct and two-step greedy: results and comparison} \label{sec:Greedy_comparison}

\begin{figure}[htp]
\includegraphics[width=0.5\linewidth]{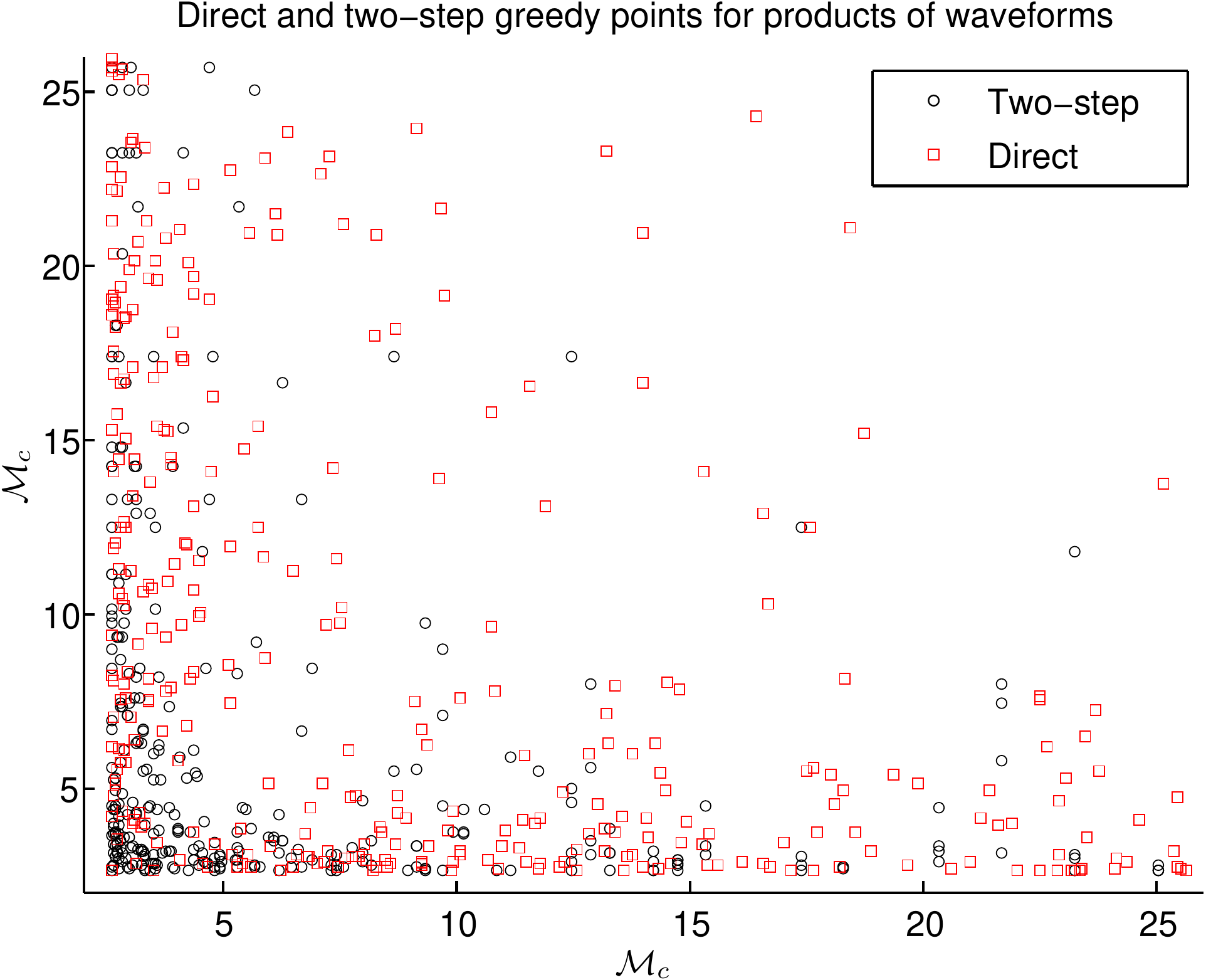}
\includegraphics[width=0.5\linewidth]{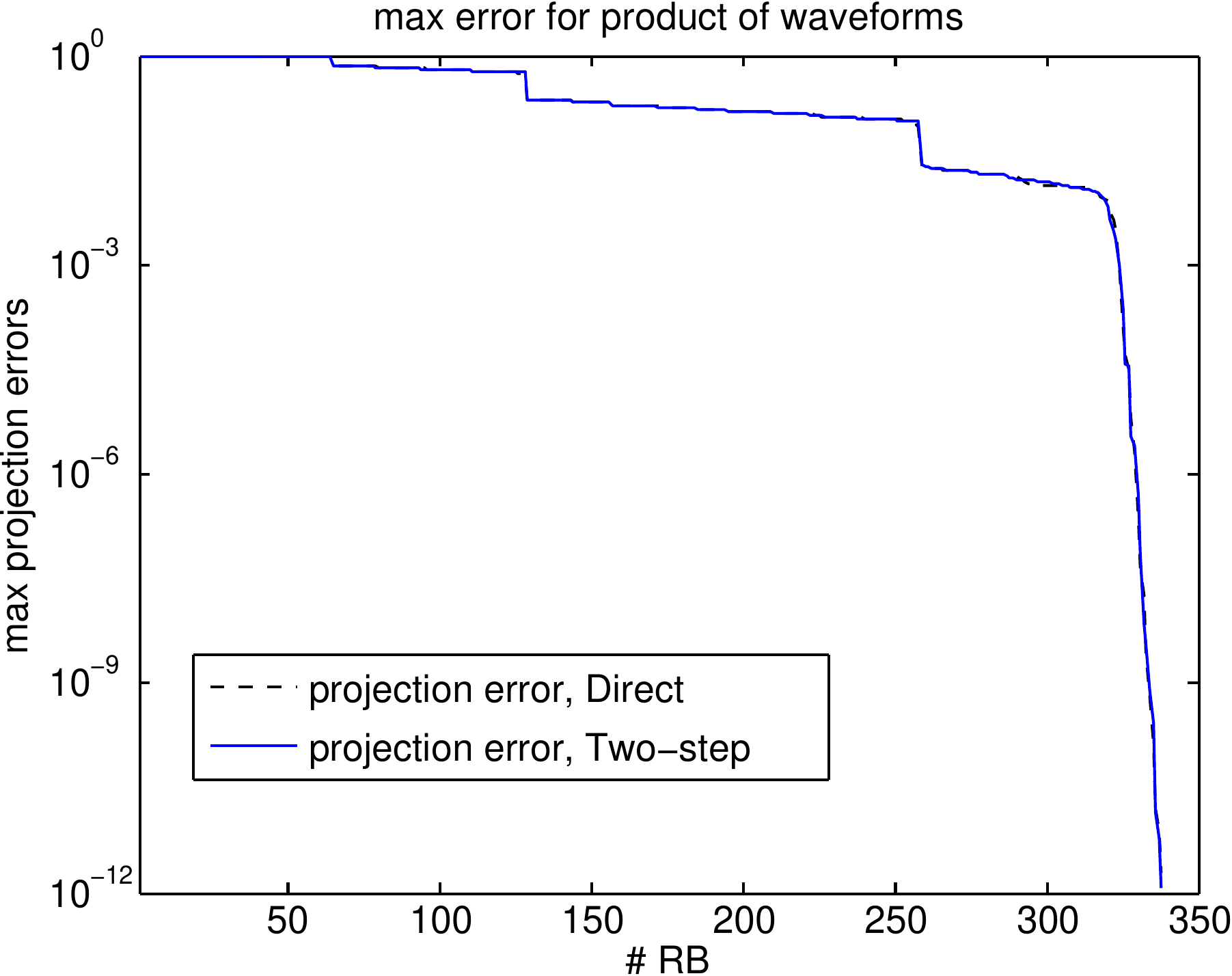}
\caption{Selected chirp mass $\cM_c$ values (left) and square of the greedy error (right) for both the direct and two-step greedy approaches. Each approach seeks to compress a training space which consists of normalized products of waveforms. However, these training spaces are constructed in very different ways. The two-step training space is built from a Cartesian product of greedy points identified while building an RB space for waveforms (as opposed to their products) and thus has $178^2$ members. A direct greedy uses a training space from the Cartesian product of the $\cT_K$ defined in Eq.~\eqref{eqn:GW_training} and has $3000^2$ members. Remarkably, despite these differences, we find that {\em exactly} $339$ basis vectors are needed to achieve an error tolerance of $\epsilon^2 = 10^{-12}$ for both the two-step and direct greedy. The left figure shows the point distribution is slightly different with more clustering at lower masses for the two-step process (black circles) as compared with the direct one (red squares).}
\label{fig:greedycompare}
\end{figure}

In the previous subsection we constructed an approximation space $F_n \approx \cF$. It was found that 178 basis elements are needed to represent any member of a 3000 member training set given by \eqref{eqn:GW_training} with an accuracy better than $\epsilon^2 = 10^{-12}$. 
We continue following Algorithm~\eqref{algo:ROQ} with an aim towards approximating $\widetilde\cF$. Results for both the two-step and direct greedy algorithms are given, followed by a short discussion.

\smallskip
\noindent
{\bf Two-step greedy results:} 
As discussed in Sec.~\ref{sec:Products} a second RB-greedy is used to build a space $\widetilde{F}_m$ which approximates $\widetilde\cF$. Following the prescription outlined there, we begin by building a training set 
$\cT_{n}^2 = \{ ( \mu_i , \mu_j ) \}_{i,j=1}^n$ and a training space $F_{n^2}$ of associated normalized products $\{ h_{\mu_i}^*h_{\mu_j} / \| h_{\mu_i}^*h_{\mu_j} \|_{\tt d} \}_{i,j=1}^n$, where $\{\mu_i \}_{i=1}^n$ are the greedy points identified by the first greedy approximation ${\cal F} \approx F_n$ carried out in Sec.~\ref{sec:RB_DEIM_waveforms}. Our definition of $F_{n^2}$ is not the only choice and, in particular, one may build a training space from normalized products of the orthogonalized basis $\{ e_i \}_{i=1}^{178}$  (cf. Remark~\ref{remark:TwoStepChoices}). 

Through numerical experiments we have found that $339$ reduced basis elements are needed to achieve a tolerance of $\epsilon^2 = 10^{-12}$ for approximating  $F_{n^2}$. Fig.~\ref{fig:greedycompare} (left) shows the distribution of points selected by the greedy algorithm (the two-step greedy results are denoted by black circles). As expected they cluster towards lower values of $\cM_c$ (see Sec.~\ref{sec:RB_DEIM_waveforms}). The solid blue line labeled ``projection error, Two-step" in Fig. \ref{fig:greedycompare} (right) plots the square of the greedy error $\greedy{m}{M}(F_{n^2};\cH)$, defined in Eq.~\eqref{eq:RB-error}, over the training space of $F_{n^2}$, as a function of the number of  reduced basis elements. Furthermore, through Monte Carlo sampling of the continuum we find any waveform to be accurately represented (see ``interpolation error" of Fig.~\ref{fig:EIMInt_Prod}, right) by these same 339 basis elements.

\smallskip
\noindent
{\bf Direct greedy results:} 
One may consider direct approximation of $\widetilde\cF$ through a single RB-greedy (Path \# 1 in Algorithm~\ref{algo:ROQ}). For most problems this will be prohibitively expensive even as an offline computation. Nevertheless, we provide details on it here mainly for comparison with the two-step approach (Path \# 2 in Algorithm~\ref{algo:ROQ}). We take our training set to be a Cartesian product $\cT_{K}^2 = \{ ( \mu_i , \mu_i ) \}_{i,j=1}^K$ of the $K=3000$ element training set given by \eqref{eqn:GW_training} and a training space $F_{K^2}$ to be the associated normalized products $h_{\mu_i}^*h_{\mu_j} / \| h_{\mu_i}^*h_{\mu_j} \|_{\tt d}$. To achieve an approximation tolerance of $\epsilon^2 = 10^{-12}$ we have found that $339$ RB elements were needed. Fig.~\ref{fig:greedycompare} (left) shows the distribution of points selected by the greedy algorithm (the direct greedy results are denoted by red squares). The dashed black line labeled ``projection error, Direct" in Figure \ref{fig:greedycompare} (right) shows the square of the greedy error $\greedy{m}{M}(\cF_{K^2};\cH)$, defined in Eq.~\eqref{eq:RB-error}, over the training space $\cF_{K^2}$, as a function of the number of  RB elements. 

\smallskip
\noindent
{\bf Discussion and comparison of two-step and direct greedy:} 
We first notice that when building $F_{n^2}$ from products of basis waveforms 
(as opposed to products of orthonormal basis vectors (cf. Remark~\ref{remark:TwoStepChoices})) 
one has $F_{n^2} \subset F_{K^2}$. In light of this observation one may view the two-step greedy as a smarter choice of sampling $\widetilde\cF$. Furthermore, to better assist with our comparison, in the numerical experiments carried out above we have initialized both the two-step and direct greedy algorithms with the same seed.

Remarkably, despite the differences in the two approaches in terms of computational cost, we find that the greedy error curves are nearly identical and in both cases {\em exactly} 339 basis vectors are needed to achieve an error tolerance of $\epsilon^2 = 10^{-12}$ (the greedy errors are defined with respect to different training spaces, however). In both cases after a slowly decaying region we observe very fast exponential convergence of the form $C e^{-c_0 n^\alpha}$. For the two-step greedy, $\greedy{m}{M}(F_{n^2};\cH)$ can be fitted by 
\[
    C=4.19 \times 10^{-3}, c_0=0.981, \mbox{ and } \alpha=0.923,
\]
while for the direct one $\greedy{\widehat{m}}{M}(\cF_{K^2};\cH)$ can be fitted by 
\[
    \widehat{C}=3.98 \times 10^{-3}, \widehat{c}_0=1.07 \mbox{ and } \widehat{\alpha}=0.875. 
\]

The number $m$ of basis found to be needed so that $\widetilde{F}_m$ represents $\widetilde\cF$ within machine precision, $m=339$,  is remarkably close to twice that one needed so that $F_n$ represents $\cF$ with the same accuracy, $n=178$. If we were dealing with polynomials, it would be exactly $m=2n$. 
From a practical perspective, with the two-step greedy approach we remove significant redundancy amongst elements of $F_{n^2}$ where the dimensionality of the space to represent products of waveforms is compressed from $n^2$ to $\sim 2n$, with a compression ratio of about $90$ for this problem.

Notice the savings in the offline stage when building the reduced basis space $\widetilde{F}_m$ using our two-step greedy approach, compared to a direct one. For the problem here considered, which is not particularly large in terms of the number of physical parameters (typically an $8$ dimensional parameter space for a faithful description of compact binary coalescences \cite{PhysRevD.86.084046}), we needed $3,000$ training space points to build $F_n$. Using a direct approach to build $\widetilde{F}_m$ one needs $9 \times 10^6$ training space points, compared to the $178^2$ needed in the two-step greedy, {\em with offline savings when building  $\widetilde{F}_m$ of $\sim 284$}. At the end of the day, both basis are able to accurately represent any waveform product and hence perfectly well suited for an ROQ construction. Clearly the two-step is preferable when considering these costs; we refer to Remark~\ref{rmk:flop-count}.

\begin{remark}
The savings for larger problems, for example when using a lower cutoff frequency of $10$Hz, as estimated for the upcoming generation of earth-based detectors, or including spins in the modeling of each compact objects would be considerably larger (even in the absence of precession). For such cases we have typically used $\sim 10^6$ elements in the training space in order to build a high accuracy RB, with less than $\sim 2,000$ RB elements needed to represent 
$F_n$~\cite{Caudill:2011kv,Field:2011mf,PhysRevD.86.084046}. For those cases a direct RB-greedy construction of $\widetilde{F}_m$  would in principle require $\sim 10^{12}$ training space elements -- {\em our two-step greedy would then save the offline cost by a factor of $\sim 10^5$}. 
\end{remark}

\subsubsection{\bf DEIM and ROQ for overlap/inner-product integration} \label{sec:ROQ_for_waveforms}

In the previous subsection we constructed the approximation space $\widetilde{F}_m \approx \widetilde\cF$ using a direct and two-step greedy algorithms. While both approaches identify an accurate and compact reduced basis set $\{ \widetilde{e}_i \}_{i=1}^{339}$ approximating $\widetilde{\cF}$, the two-step one has a much smaller computational cost (see Remark~\ref{rmk:flop-count}). We now continue with Algorithm~\eqref{algo:ROQ}. The results shown are for a reduced basis generated from the two-step greedy.

\smallskip
\noindent
{\bf Discrete Empirical Interpolation:} 
We now generate the corresponding set of interpolation points using the DEIM algorithm of Section~\ref{sec:EIM}. Results for Gauss-Legendre sampling (case 1) are shown in Figure~\ref{fig:EIMInt_Prod}, equidistant sampling is qualitatively similar. A distribution of selected DEIM points is depicted in the leftmost plot. Notice that the overall structure is very similar to the single waveform case shown in Figure~\ref{fig:frqshist}.

For $20,000$ randomly sampled normalized waveform products $g \in \widetilde{\cF}$, not necessarily in the training space, the DEIM interpolant is evaluated using Eq.~\eqref{eq:eim_disc}. The DEIM interpolation error is calculated 
as $\| g - \widetilde{\cI}_m[g] \|^2_{\tt d}$. Fig.~\ref{fig:EIMInt_Prod} (right) compares the largest DEIM interpolation error over all $20,000$ products (solid blue line) with the square of the greedy error $\sigma_m\del{F_{n^2};\cH}$ (solid black line). Notice that the DEIM interpolation error is bounded from above by the a-priori error estimates Eq.~\eqref{eq:deim-error-L2} (red line) and Eq.~\eqref{eq:deim-error-L2_2} (purple line). These two error estimates strictly hold for waveforms in the training space, and evidently continue to hold for waveforms outside of the training space when the latter is sufficiently dense. The red line provides a sharper bound on the error, but its also more expensive to compute (cf.~Remark \ref{remark:DEIM_bounds}).

\begin{figure}[ht]
\includegraphics[width=0.5\linewidth]{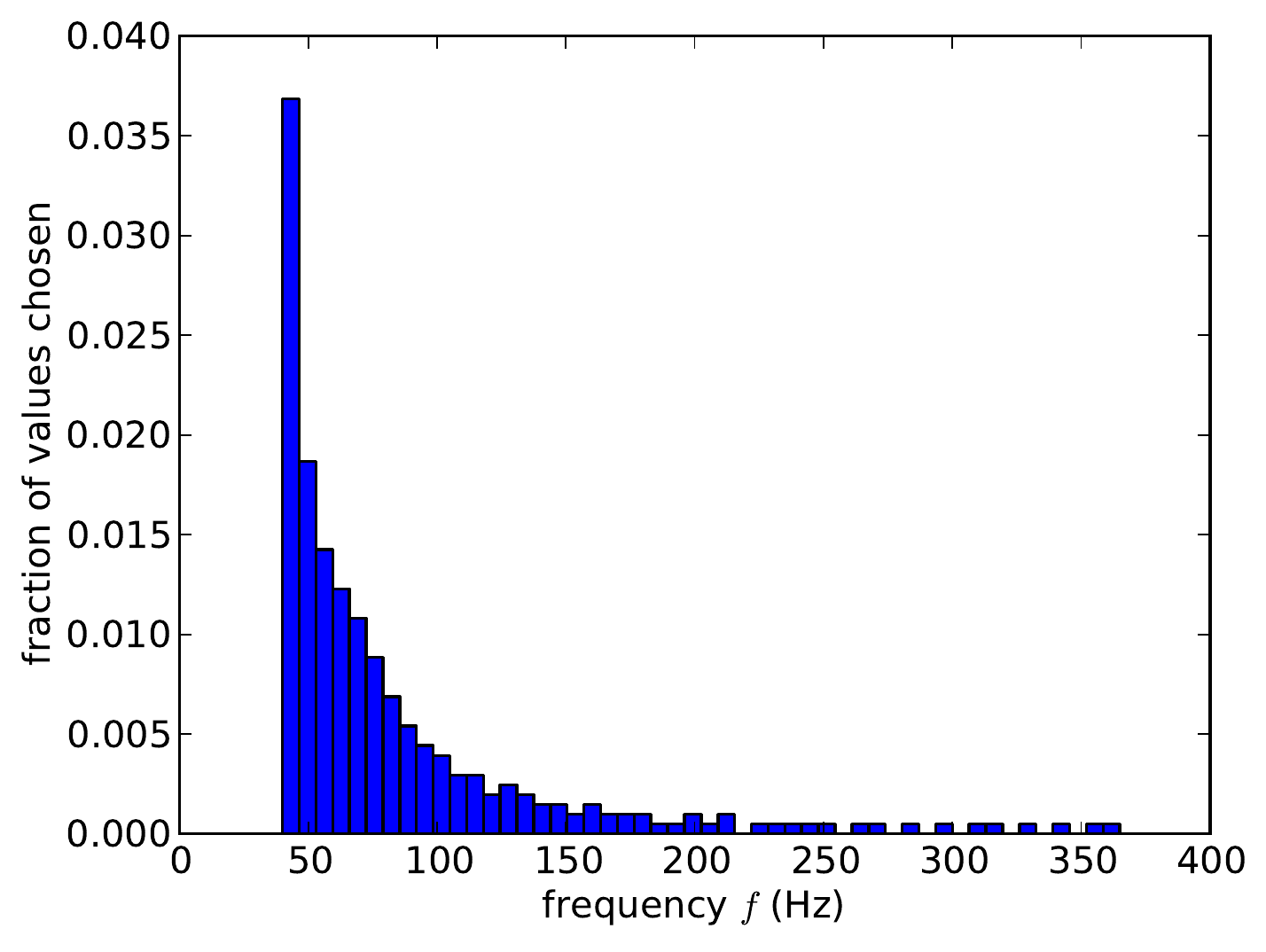}
\includegraphics[width=0.5\linewidth]{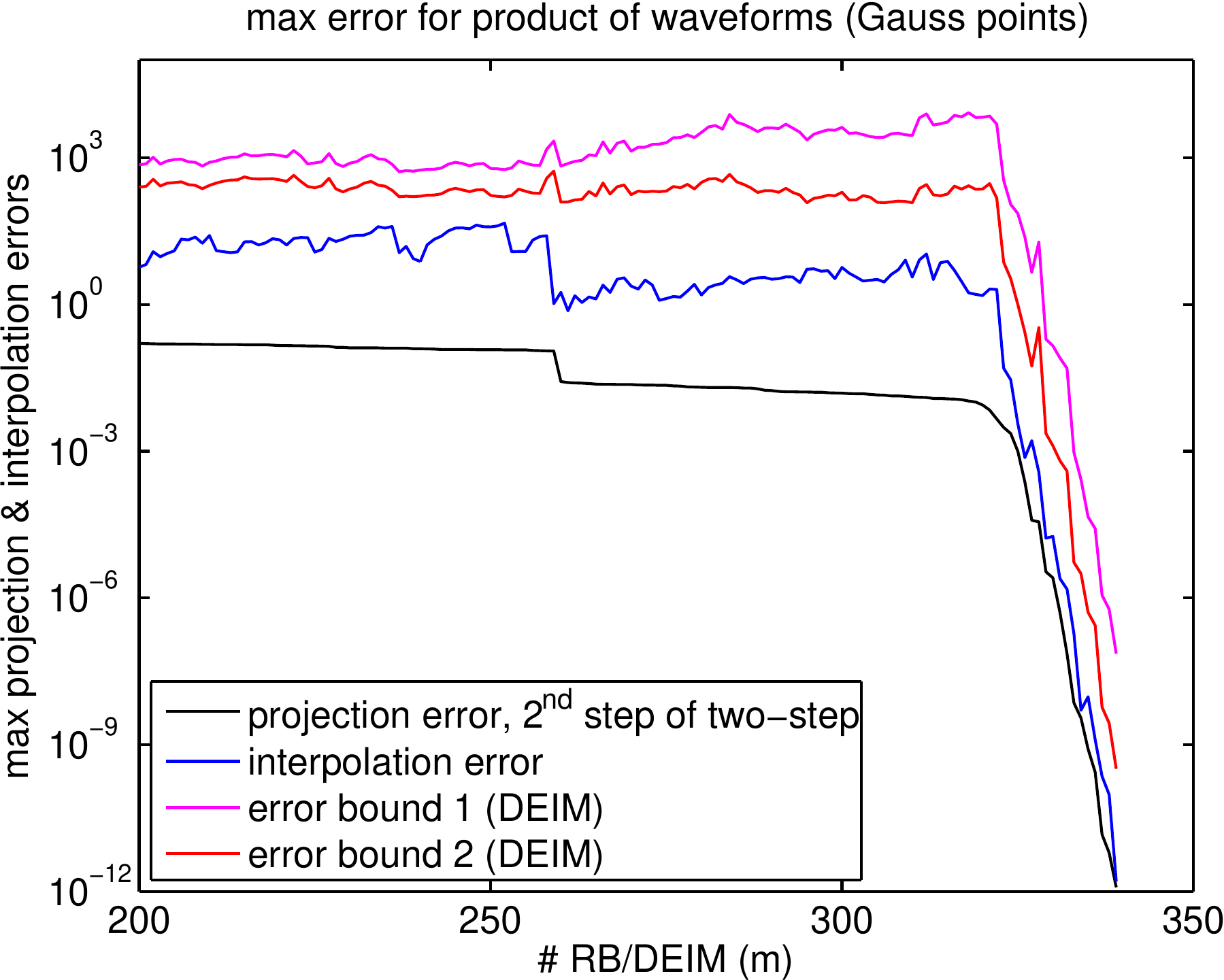}
\caption{The left figure shows the distribution of empirical interpolation points selected by the DEIM algorithm with the basis vectors sampled at $1,701$ Gauss-Legendre points. The right figure shows maximum projection and interpolation error (and error bounds) for products of waveforms as a function of the number of reduced basis used; notice that the profile is similar to that one of Fig.~\ref{fig:EIMInt}. On the right figure ``projection error" plots the square of the greedy error $\greedy{m}{M}(F_{n^2};\cH)$ and ``interpolation error" plots a maximum error $\|g - \widetilde\cI_m\sbr{g}\|^2_{\tt d}$ taken over $20,000$ randomly drawn normalized 
products $g \in \widetilde{\cF}$.
A-priori interpolation error bounds given by Eq.~\eqref{eq:deim-error-L2} (error bound 2) and Eq.~\eqref{eq:deim-error-L2_2} (error bound 1) are also shown. Results for equidistant sampling are qualitatively similar to the Gauss-Legendre case.}
\label{fig:EIMInt_Prod}
\end{figure}

\smallskip
\noindent
{\bf Reduced Order Quadratures:} 
Having assembled a set of basis vectors and empirical interpolation points for products, we compute the reduced order quadrature weights $\omega_i^\mathrm{ROQ}$ with Eq.~\eqref{eqn:ROQ_weights} to complete the ROQ rule. Like the previous experiments of this section, we randomly sample $20,000$ inner products (with normalized integrands) to compute, and monitor the maximum error from this computation. 

Figure~\ref{fig:RBO_Quad} compares accuracy versus computational degrees of freedom (number of quadrature nodes) for a variety of integration schemes. The black and blue lines denote trapezoidal and Gauss-Legendre quadratures, respectively. Two cases are considered within reduced order quadratures. The ROQ leading to the red line stems from a reduced basis space $\{ \widetilde{e}_i \}_{i=1}^{339}$ and a $1,701$ point Gauss-Legendre quadrature rule. The magenta one, in turn, uses a $20,000$ point trapezoidal rule. One can see that ROQ have a factor of $\sim 2$ of savings when compared to Gauss-Legendre points for a maximum error below $10^{-2}-10^{-1}$, while the savings compared to the extended trapezoidal rule at high resolutions are greater than $50$. The Gauss-Legendre comparison provides a benchmark test against the best quadrature for smooth functions and already provides benefits for applications that may take days or weeks to run (e.g. parameter estimation studies with Markov chain Monte Carlo). The more relevant comparison, however, for data driven applications is with the trapezoidal rule. Real data (e.g.~measurements taken at GW observatories or, more broadly, signal detection) will not be given at Gauss-Legendre points but rather equally spaced ones and in this setting ROQ significantly outperforms its counterpart.

\begin{figure}[ht]
\centering
\includegraphics[width=0.49\linewidth]{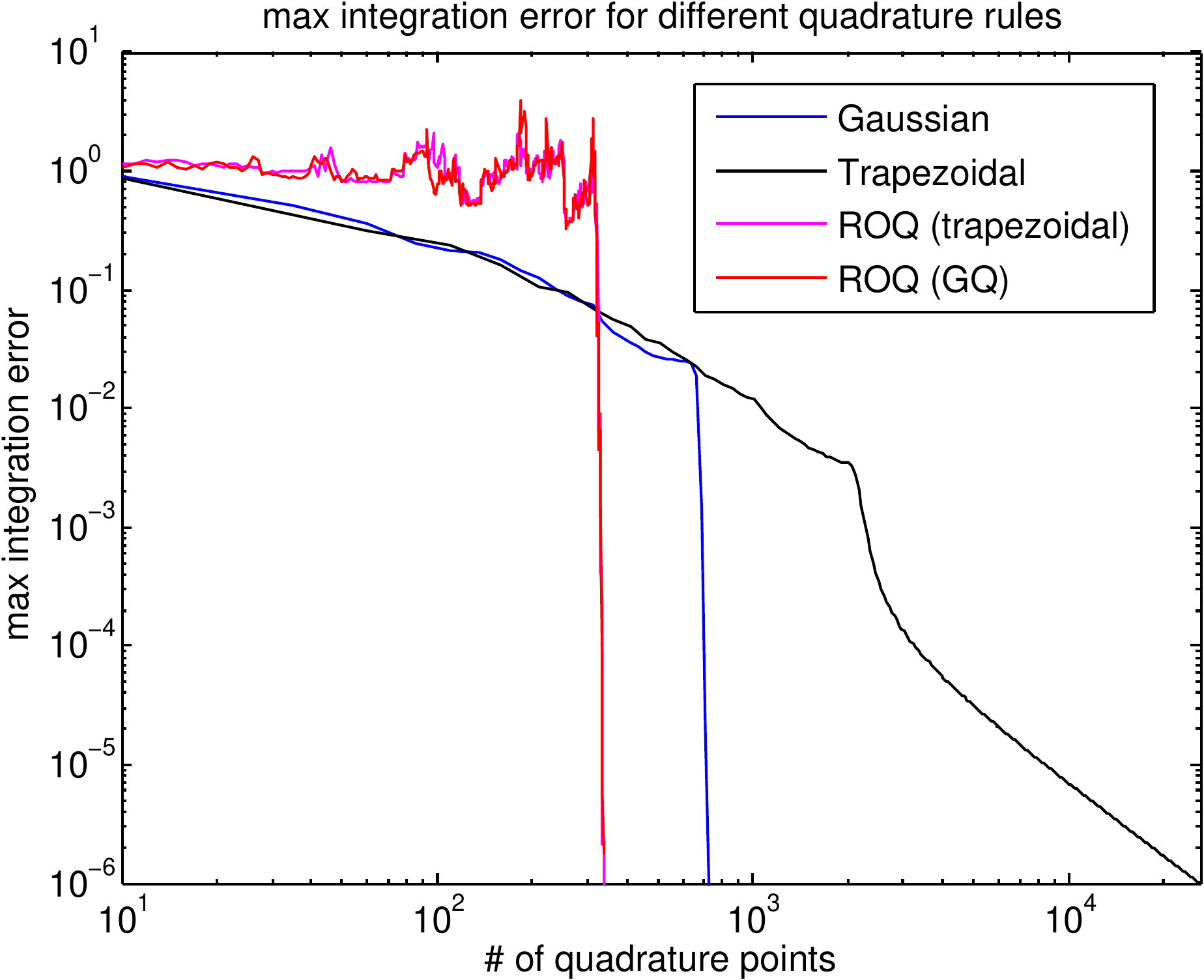} \hspace{5pt}
\includegraphics[width=0.49\linewidth]{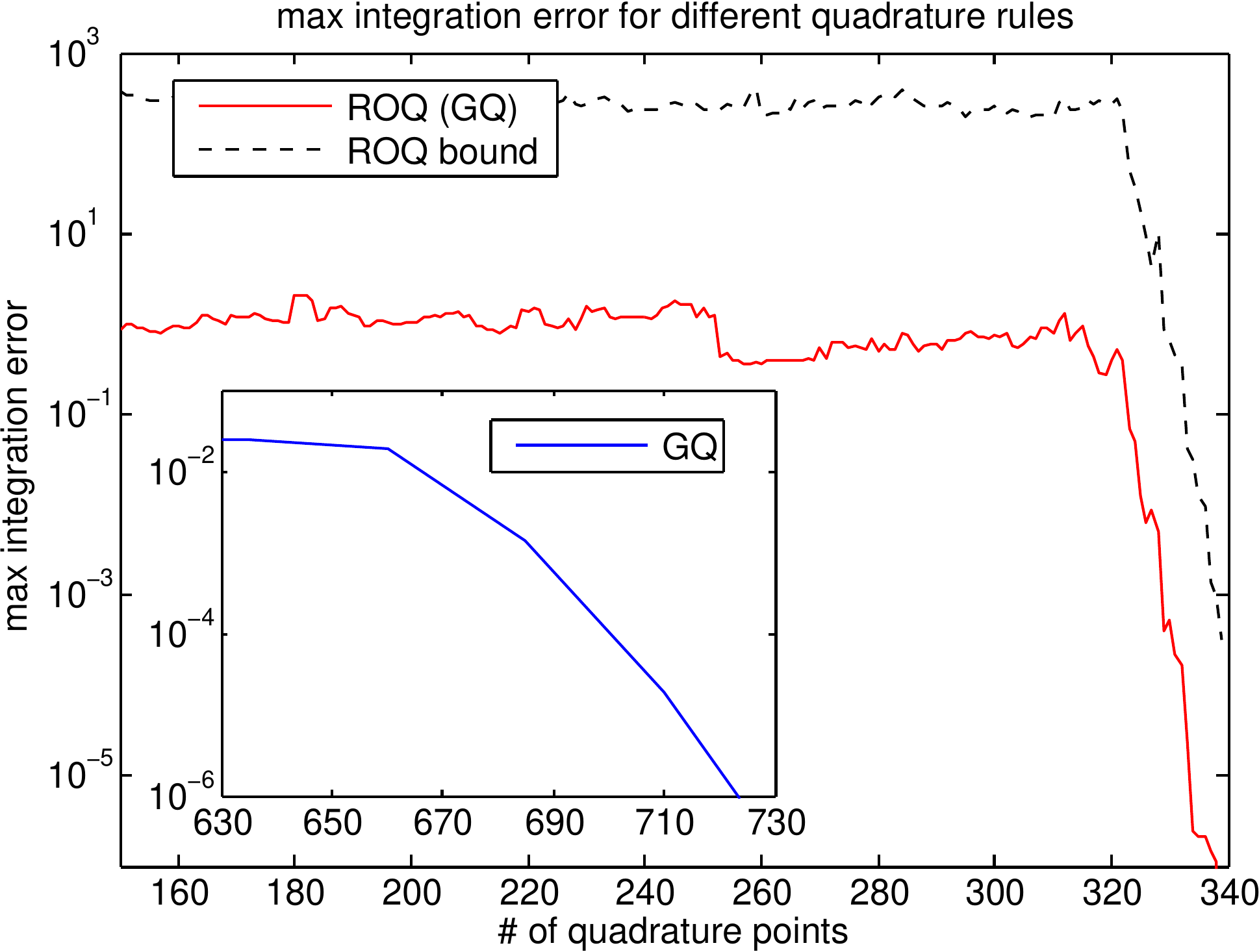}
\caption{We randomly draw $20,000$ normalized pair products $g \in \widetilde{\cF}$ and for each compute an accurate inner product (whose error is smaller than $10^{-6}$) which we take to be its exact value $I_c$. Next, for each product we compute an inner product using a i) Gauss-Legendre quadrature (blue line), ii) trapezoidal rule (black line), iii) ROQ built from Gauss-Legendre quadratures (red line), and iv) ROQ built from the trapezoidal rule (pink line). In each of the four cases we monitor the maximum errors $|I_c-I_d|$ (for the Gauss-Legendre and trapezoidal rules) and $|I_c-I_\mathrm{ROQ}|$ (for the ROQ accelerated Gauss-Legendre and trapezoidal rules). Once the underlying product is well resolved by the empirical interpolant we see very fast exponential convergence of $I_\mathrm{ROQ} \rightarrow I_c$ (compare with ``interpolation error" plotted in the right panel of Figure~\ref{fig:EIMInt_Prod}). Notice that both ROQ rules perform similarly, which is to be expected whenever the underlying integration scheme is able to accurately integrate the reduced basis (cf. the discussion following Algorithm~\ref{algo:ROQ} and Fig.~\ref{fig:RungeInt}). After an error of about $10^{-1}$ both ROQ rules outperform their discrete counterparts and in some cases significantly so. The right figure is a semi-log ``zoom-in" plot to clearly show the exponential convergence along with error bounds given by Eq.~\ref{eq:roq-error}. Results for ROQ (trapezoidal) are qualitatively similar and omitted for clarity on the right figure.}
\label{fig:RBO_Quad}
\end{figure}

\section{Final remarks}   \label{sec:conclusion}

Numerical integration is a well studied topic (see Refs.~\cite{Press92,Quarteroni2010} for excellent introductions). One may wish to identify when Reduced Order Quadratures (ROQ) are likely to be a competitive option over alternative, more standard integration methods. First and foremost, the set of functions to be integrated should be well approximated by a relatively compact basis. Although we have here focused on bases obtained through the Reduced Basis-greedy approach, ROQ apply unchanged to any other basis sets; for example, those obtained through a Proper Orthogonal/Singular Value decomposition. Second, since the cost of building the reduced basis either by a RB-greedy or POD/SVD is potentially large, it should be done {\em offline}. For larger problems, however, a direct reduced order modeling approach might be unfeasible, even when building the basis is an offline and parallelizable calculation. To overcome this we proposed a two-step greedy targeted towards approximation of products of functions (such as those appearing in weighted inner products) which, for the problem considered here, accelerated offline ROQ computations by two orders of magnitude. Finally, and crucially, one should determine which {\em online} costs to reduce. Benefits of ROQ are greatest when multiple fast evaluations of weighted inner products (or perhaps simply weighted integrals) are required between parametrized functions. Are the functional evaluations themselves costly? Quadrature rules which are designed for high-order integration of polynomials might lead to significantly more functional evaluations compared to an ROQ rule. Can one sample the function at arbitrary points or is the data sampling specified? In our numerical experiments with gravitational waves the savings were moderate (a factor of $\sim$ 2) when we were able to dictate the sampling location whereas they were much greater (a factor of $\sim 50$) when the function was sampled at equally spaced points.
This suggests the possibility of using ROQ in data analysis applications such as 
matched filtering and Bayesian parameter estimation.

There are many potential applications of ROQ which deserve further consideration. As pointed out in Ref.~\cite{Maday_2009}, empirical interpolation is remarkably versatile in its handling of multi-dimensional problems on irregular domains, and a natural application of ROQ would be to such problems. Discontinuous functions (in physical space) with smooth variation with respect to parameters might admit a reduced basis approximating to the full space with fast convergence, and if so it might be possible to generate fast converging ROQ for families of discontinuous or noisy functions. Furthermore, as pointed out in Sec.~\ref{sec:ROQ}, the hierarchical construction of ROQ node locations allows for very natural application-specific nested quadrature rules of arbitrary order and depth.

\section*{Acknowledgments}
This work has been supported in part by NSF Grants DMS-0807811, PHY0801213, PHY1005632, and DMS-1109325 to the University of Maryland, and PHY1125915 to the University of California at Santa Barbara. We thank Priscilla Canizares 
for comments on the manuscript and Chad Galley for useful discussions. MT and SEF thank the Kavli Institute for Theoretical Physics, University of California at Santa Barbara, where this work was completed, for its hospitality. 
H.\,A. and M.\,T. thank Tryst DC, where parts of this work were done, for its hospitality.

\appendix

\section{Greedy and EIM algorithms} \label{app:algs}

\subsection{Greedy algorithm}

For any $\epsilon > 0$ and training set $\cT_K$ the greedy algorithm to build a reduced basis is as follows:
\begin{algorithm} {(\bf RB Greedy Algorithm)}  \label{algo:RB-greedy}  $\quad$\\
$\left[\{e_\ell\}_{\ell=1}^n, \{\mu_\ell\}_{\ell=1}^n\right]$ = RB-Greedy ($\epsilon$, $\cT_K$)
\begin{itemize}
\setlength{\itemsep}{2pt}
\item[(1)] Set $n=0$ and define $\greedy{0}{M}(\cF_{\mathrm{train}};\cH) := 1$
\item[(2)] Choose an arbitrary $\mu_1 \in \cT_K$ and set $e_1 := h_{\mu_1}$ $\quad$ {\bf Comment:} $\|h_{\mu_1}\|_{\tt d} = 1$
\item[(3)] do, while $\greedy{n}{M}(\cF_{\mathrm{train}};\cH) \ge \epsilon$
\begin{itemize}
\setlength{\itemsep}{2pt}
\item[(a)] $n = n + 1$
\item[(b)] $\greedy{n}{M}(h_{\mu}) := \left\|  h_{\mu} - \proF{n}{} h_\mu \right\|_{\tt d}$ for all $\mu \in \cT_K$
\item[(c)] $\greedy{n}{M}(\cF_{\mathrm{train}};\cH) =  \sup_{ \mu \in \cT_K} \left\{ \greedy{n}{M}(h_{\mu}) \right\}$
\item[(d)] $ \mu_{n+1} :=  \argsup_{\mu \in \cT_K}  \left\{ \greedy{n}{M} (h_\mu) \right\} \quad \mbox{(greedy sweep)} $
\item[(e)] $e_{n+1} := h_{\mu_{n+1}} - \proF{n}{} h_{\mu_{n+1}} \quad \mbox{(Gram-Schmidt Orthogonalization)} $
\item[(f)] $e_{n+1} := e_{n+1}/ \| e_{n+1} \|_{\tt d}  \quad \mbox {(normalization)}$
\end{itemize}
\end{itemize}
\end{algorithm}
\vspace{10pt}

\begin{remark}
In step $3$b the error is computed exactly as $\left\|  h_{\mu} - \proF{n}{} h_\mu \right\|_{\tt d}$. In the setting of PDEs, RB methods avoid exact error computations by using parametric error estimators and without computing full solutions.
In the context of the empirical interpolation method Ref.~\cite{Maday_2009} suggests using $\hat{\sigma}( h_{\mu} ) = \| h_{\mu} - I_n [h_{\mu} ]\|_{\tt d}$, where $I_n [h_{\mu}]$ is the empirical interpolant. Indeed, for applications limited by computational resources using $\hat{\sigma}( h_{\mu} )$ could be desirable. Within the Hilbert space setting described in this paper, however, an exact error $\sigma( h_{\mu} )$ computation results in a reduced basis space which will more accurately approximate the full space (either $\cF$ or $\widetilde{\cF}$) and should be used whenever possible.
\end{remark}

\subsection{EIM algorithm}

\begin{algorithm} [Selection of DEIM Points]  \label{algo:DEIM} $\quad$  \\
$\left[\bP, \{ p_i \}_{i=1}^n\right]$ = DEIM ($\bV$, $\{ x_k \}_{k=1}^M$)
{\bf Comment:} the column vectors of $\bV$ must be linearly independent 
\begin{itemize}
 \item[(1)] $j = \mathrm{arg}\hspace{-1pt} \max | \be_1 | $ $\quad$ {\bf Comment:} here $\mathrm{arg}\hspace{-1pt} \max$ takes a vector and returns the {\em index} of its largest entry
 \item[(2)] Set $\bU = [\be_1]$, $\bP = [\hat{\be}_j]$, $p_1 = x_j$ $\quad$ {\bf Comment:} $\hat{\be}_j$ is a unit column vector with a single unit entry at index $j$
 \item[(3)] for $i = 2, \dots , n$ do
 \begin{itemize}
        \item[(4)] Solve $( \bP^T \bU ) \bc = \bP^T \be_i$ for $\bc$
        \item[(5)] $\br = \be_i - \bU \bc$ 
        \item[(6)] $j = \mathrm{arg}\hspace{-1pt} \max | \br | $
        \item[(7)] Set $\bU = [\bU \quad\hspace{-2pt} \br]$, $\bP = [\bP \quad\hspace{-2pt} \hat{\be}_j]$, 
                     $p_i = x_j$ 
\end{itemize}
\end{itemize}
\end{algorithm}

The DEIM algorithm described above is nearly identical to the one given in Ref.~\cite{chaturantabut:2737} with the exception of Step 7 which is $\bU = [\bU \quad \be_i]$ in that reference. In Appendix~\ref{app:DEIM_Gauss_elimination} we show these algorithms to be equivalent and, furthermore, we show a relationship between DEIM and Gauss-elimination with partial pivoting. Owing to the lower triangular form of $\bP^T \bU$, in Appendix~\ref{app:flop-count} we show that Algorithm \ref{algo:DEIM} has a reduced computational cost. 


\section{Asymptotic FLOP Count} \label{app:flop-count}

\subsection{DEIM FLOP Count}
In Algorithm~\ref{algo:DEIM} as $\bP^T\bU$ is a lower triangular matrix therefore (after $m$ iterations) Step (4) costs $\bigo{ m^3 }$, Step (5) costs $\bigo{Mm}$ subtractions and $\bigo{M m^2}$ matrix vector multiplications, where latter is the dominant one. 
The improvement in our implementation of  Algorithm~\ref{algo:DEIM} as compared to \cite{chaturantabut:2737} is that
the our implementation scales as $\bigo{m^3}$ as compared to $\bigo{m^4}$ in \cite{chaturantabut:2737}. Therefore
the total DEIM cost is $\bigo{Mm^2+m^3}$.

\subsection{ROQ FLOP Count}
The dominant cost (assuming $M > m$) of computing reduced order quadrature weights 
$\del{\bomega^\mathrm{ROQ}}^T = \bomega^T \widetilde{\bV}\del{ \widetilde\bP^T \widetilde\bV}^{-1}$ in 
\eqref{eqn:ROQ_weights} is $\bigo{Mm^2}$. Here we used the fact that the operation 
$\widetilde{\bP}^T \widetilde{\bV}$ is equivalent to selecting $m$ rows corresponding to DEIM points 
$\{ \widetilde{p}_\ell \}_{\ell=1}^m$.

Then the overall cost to compute ROQ using one-step (Algorithm~\ref{algo:RB-greedy} [{\em Path \#1}] with 
modified {\em Gram-Schmidt} in greedy is:
\[
    \bigo{K^2 M \widehat{m}  + M \widehat{m}^2 } 
\]
and two-step (Algorithm~\ref{algo:RB-greedy} [{\em Path \#2}] is:
\[
    \bigo{n^2 M m + K M n  + M m^2 } \, .
\]
Here $m$, $\widehat{m}$ denote the number of reduced basis used in approximation of $\widetilde{\cF}$ via 
two-step and direct approach respectively and $n$ is the number of reduced basis used to approximate $\cF$.

\section{DEIM and LU Decomposition with Partial Pivoting} \label{app:DEIM_Gauss_elimination}

The original DEIM described in  \cite{chaturantabut:2737} has as Step 7 
\[
     \bU = [ \bU \quad \be_i], \quad i = 2, \dots n \, .
\] 
Since the columns of $\bU$ are linearly independent, $\bP^T\bU$ is always invertible, but
in this form the matrix $\bP^T \bU$ can be dense. 
Next we show that with a slight modification to the original DEIM, we get Algorithm~\ref{algo:DEIM} and
at every iteration $i=2, \dots, n$ with $\br_i := \br$,  
\[
     \bU := [ \bU \quad \br_i]  \, ,
\]
gives the same result. One of the 
advantages of looking at DEIM in this format is that the matrix $\bP^T \bU$ is lower triangular, therefore the system can be solved for $\bc$ with $\bigo{n^2}$
operations, as compared to using $\be_i$, which gives a dense matrix 
requiring $\bigo{n^3}$ operations.

Next we prove that we get the same result using both formats.

\begin{proposition}  \label{prop:gauss-elimination}
In \cite{chaturantabut:2737}, Step 7 of Algorithm~\ref{algo:DEIM} was presented as 
$\bU = [ \bU \quad \be_i]$, $i=2,\dots,n$, we can replace it by
\[
    \bU = [ \bU \quad \br_i] , \qquad i = 2, \dots, n ,
\]
where $\br_i := \br$, at every iteration, and $\br$ is as given in Step $5$ of Algorithm~\ref{algo:DEIM}.
\end{proposition}
\begin{proof}
Step $4$ of Algorithm~\ref{algo:DEIM}, at $i=n$, gives the coefficient matrix  
\begin{align} \label{eq:pu}
\bP^T \bU =
   \left(  \begin{array}{cccc}   
              \be_1(p_1)  & \be_2(p_1)  & \cdots & \be_{n-1}(p_1)      \\
              \be_1(p_2)  & \be_2(p_2)  & \cdots & \be_{n-1}(p_2)      \\
              \be_1(p_3)  & \be_2(p_3)  & \cdots & \be_{n-1}(p_3)      \\              
              \vdots            & \vdots            & \ddots & \vdots                       \\
              \be_1(p_{n-1})  & \be_2(p_{n-1})  & \cdots & \be_{n-1}(p_{n-1})   \\               
             \end{array}
   \right) .
\end{align}
Next we write the row reduced echelon form using forward Gaussian elimination for $( \bP^T\bU )^T$, where 
the first step implies
\[ 
   \left(  \begin{array}{cccc}   
              \be_1(p_1)  &   0                     & \cdots & 0      \\
              \be_1(p_2)  &  \br_2(p_2)     & \cdots & \be_{n-1}(p_2) - \frac{\be_{n-1}(p_1)}{\be_1(p_1)}\be_1(p_2)      \\
              \be_1(p_3)  &  \br_2(p_3)     & \cdots & \be_{n-1}(p_3) - \frac{\be_{n-1}(p_1)}{\be_1(p_1)}\be_1(p_3)       \\              
              \vdots            & \vdots               & \ddots & \vdots                       \\
              \be_1(p_{n-1})  & \br_2(p_{n-1})    & \cdots & \be_{n-1}(p_{n-1}) - \frac{\be_{n-1}(p_1)}{\be_1(p_1)}\be_1(p_{n-1})    \\               
             \end{array}
   \right) .
\]
The final row reduced echelon form for $( \bP^T\bU )^T$ is
\begin{align} \label{eq:forward-elimination}
   \left(  \begin{array}{cccc}   
              \be_1(p_1)  &   0                     & \cdots & 0      \\
              \be_1(p_2)  &  \br_2(p_2)     & \cdots & 0     \\
              \be_1(p_3)  &  \br_2(p_3)     & \cdots & 0       \\              
              \vdots            & \vdots               & \ddots & \vdots                       \\
              \be_1(p_{n-1})  & \br_2(p_{n-1})    & \cdots & \br_{n}(p_{n-1})    \\               
             \end{array}
   \right) .
\end{align}
Hence the proposition.
\end{proof}

\begin{remark}
Algorithm~\ref{algo:DEIM} has a flavor of Gauss elimination with row pivoting. 
We showed the elimination part in Proposition~\ref{prop:gauss-elimination}. 
Now we  show that the DEIM points $\{ p_1, \dots, p_n \}$ are in fact the pivots.
Our goal is to arrive at \eref{eq:forward-elimination} row echelon form of $( \bP^T\bU )^T$ with 
with partial pivoting. Consider the matrix \eqref{eq:pu}.
Let the location of the first pivot, i.e., the maximum of the absolute value of the first column be $p_1$
(this is the same as Step 1 in Algorithm \ref{algo:DEIM}). Apply the first step of forward elimination
as in the proof of Proposition~\ref{prop:gauss-elimination}. Next compute the second pivot $p_2$, 
which is the maximum of the absolute value of the so-obtained
second column. Following in this manner we arrive at the matrix in \eref{eq:forward-elimination}.
\end{remark}



\bibliographystyle{elsarticle-num}
\bibliography{../bibtex-references/references}

\begin{thebibliography}{10}
\expandafter\ifx\csname url\endcsname\relax
  \def\url#1{\texttt{#1}}\fi
\expandafter\ifx\csname urlprefix\endcsname\relax\def\urlprefix{URL }\fi
\expandafter\ifx\csname href\endcsname\relax
  \def\href#1#2{#2} \def\path#1{#1}\fi

\bibitem{PateraBook}
A.~Patera, G.~Rozza, Reduced Basis Approximation and A Posteriori Error
  Estimation for Parametrized Partial Differential Equations, to appear in
  (tentative rubric) MIT Pappalardo Graduate Monographs in Mechanical
  Engineering. http://augustine.mit.edu/methodology/methodology\_bookPartI.htm.

\bibitem{Maday_2009}
Y.~Maday, N.~C. Nguyen, A.~T. Patera, S.~H. Pau, A general multipurpose
  interpolation procedure: the magic points, Communications on Pure and Applied
  Analysis 8 (2009) 383--404.

\bibitem{Barrault2004667}
M.~Barrault, Y.~Maday, N.~C. Nguyen, A.~T. Patera, An empirical interpolation
  method: application to efficient reduced-basis discretization of partial
  differential equations, Comptes Rendus Mathematique 339 (2004) 667--672.
\newblock \href {http://dx.doi.org/10.1016/j.crma.2004.08.006}
  {\path{doi:10.1016/j.crma.2004.08.006}}.

\bibitem{chaturantabut:2737}
S.~Chaturantabut, D.~C. Sorensen,
  \href{http://link.aip.org/link/?SCE/32/2737/1}{Nonlinear model reduction via
  discrete empirical interpolation}, SIAM Journal on Scientific Computing
  32~(5) (2010) 2737--2764.
\newblock \href {http://dx.doi.org/10.1137/090766498}
  {\path{doi:10.1137/090766498}}.
\newline\urlprefix\url{http://link.aip.org/link/?SCE/32/2737/1}

\bibitem{Chaturantabut5400045}
S.~Chaturantabut, D.~Sorensen, Discrete empirical interpolation for nonlinear
  model reduction, in: Decision and Control, 2009 held jointly with the 2009
  28th Chinese Control Conference. CDC/CCC 2009. Proceedings of the 48th IEEE
  Conference on, 2009, pp. 4316 --4321.
\newblock \href {http://dx.doi.org/10.1109/CDC.2009.5400045}
  {\path{doi:10.1109/CDC.2009.5400045}}.

\bibitem{MHinze_SVolkwein_2005a}
M.~Hinze, S.~Volkwein, Proper orthogonal decomposition surrogate models for
  nonlinear dynamical systems: error estimates and suboptimal control, in:
  Dimension reduction of large-scale systems, Vol.~45 of Lect. Notes Comput.
  Sci. Eng., Springer, Berlin, 2005, pp. 261--306.

\bibitem{manolakis2000statistical}
D.~Manolakis, V.~Ingle, S.~Kogon,
  \href{http://books.google.com/books?id=3RQfAQAAIAAJ}{Statistical and adaptive
  signal processing: spectral estimation, signal modeling, adaptive filtering,
  and array processing}, Artech House signal processing library, Artech House,
  2000.
\newline\urlprefix\url{http://books.google.com/books?id=3RQfAQAAIAAJ}

\bibitem{Wainstein_L:1962}
L.~A. Wainstein, V.~D. Zubakov, Extraction of signals from noise,
  Prentice-Hall, London, 1962.

\bibitem{Eftang:2011}
J.~L. Eftang, B.~Stamm, \href{http://dx.doi.org/10.1002/nme.3327}{Parameter
  multi-domain empirical interpolation}, International Journal for Numerical
  Methods in Engineering 90~(4) (2012) 412--428.
\newblock \href {http://dx.doi.org/10.1002/nme.3327}
  {\path{doi:10.1002/nme.3327}}.
\newline\urlprefix\url{http://dx.doi.org/10.1002/nme.3327}

\bibitem{Aanonsen2009}
T.~O. Aanonsen,
  \href{http://ntnu.diva-portal.org/smash/record.jsf?pid=diva2:348771}{Empiric%
al interpolation with application to reduced basis approximations}, Ph.D.
  thesis, Norwegian University of Science and Technology (2009).
\newline\urlprefix\url{http://ntnu.diva-portal.org/smash/record.jsf?pid=diva2:%
348771}

\bibitem{Pinkus}
A.~Pinkus, N-widths in approximation theory, Springer, Amsterdam, 1985.

\bibitem{Binev10convergencerates}
P.~Binev, A.~Cohen, W.~Dahmen, R.~A. DeVore, G.~Petrova, P.~Wojtaszczyk,
  \href{http://dblp.uni-trier.de/db/journals/siamma/siamma43.html#BinevCDDPW11%
}{Convergence rates for greedy algorithms in reduced basis methods.}, SIAM J.
  Math. Analysis 43~(3) (2011) 1457--1472.
\newline\urlprefix\url{http://dblp.uni-trier.de/db/journals/siamma/siamma43.ht%
ml#BinevCDDPW11}

\bibitem{DeVore2012}
R.~DeVore, G.~Petrova, P.~Wojtaszczyk, Greedy algorithms for reduced bases in
  banach spaces, Arxiv preprint arXiv:1204.2290.

\bibitem{Eftang:2010}
J.~L. Eftang, A.~T. Patera, E.~M. Ronquist, An hp certified reduced basis
  method for parametrized elliptic partial differential equations, SIAM J. Sci.
  Comput. 32 (2010) 3170--3200.

\bibitem{Bui:2007}
T.~Bui-Thanh, Model-constrained optimization methods for reduction of
  parameterized large-scale systems, PhD thesis, Massachusetts Institute of
  Technology.

\bibitem{Haasdonk:2010}
B.~Haasdonk, M.~Dihlmann, M.~Ohlberger, A training set and multiple bases
  generation approach for parametrized model reduction based on adaptive grids
  in parameter space, Tech. Report, SRC SimTech.

\bibitem{Caudill:2011kv}
S.~Caudill, S.~E. Field, C.~R. Galley, F.~Herrmann, M.~Tiglio, {Reduced Basis
  representations of multi-mode black hole ringdown gravitational waves},
  Class. Quant. Grav. 29 (2012) 095016.
\newblock \href {http://arxiv.org/abs/1109.5642} {\path{arXiv:1109.5642}}.

\bibitem{Field:2011mf}
S.~E. Field, C.~R. Galley, F.~Herrmann, J.~S. Hesthaven, E.~Ochsner, M.~Tiglio,
  {Reduced basis catalogs for gravitational wave templates}, Phys. Rev.Lett.
  106 (2011) 221102.
\newblock \href {http://arxiv.org/abs/1101.3765} {\path{arXiv:1101.3765}},
  \href {http://dx.doi.org/10.1103/PhysRevLett.106.221102}
  {\path{doi:10.1103/PhysRevLett.106.221102}}.

\bibitem{PhysRevD.86.084046}
S.~E. Field, C.~R. Galley, E.~Ochsner,
  \href{http://link.aps.org/doi/10.1103/PhysRevD.86.084046}{Towards beating the
  curse of dimensionality for gravitational waves using reduced basis}, Phys.
  Rev. D 86 (2012) 084046.
\newblock \href {http://dx.doi.org/10.1103/PhysRevD.86.084046}
  {\path{doi:10.1103/PhysRevD.86.084046}}.
\newline\urlprefix\url{http://link.aps.org/doi/10.1103/PhysRevD.86.084046}

\bibitem{Cannon:2010qh}
K.~Cannon, A.~Chapman, C.~Hanna, D.~Keppel, A.~C. Searle, et~al., {Singular
  value decomposition applied to compact binary coalescence gravitational-wave
  signals}, Phys. Rev. D82 (2010) 044025.
\newblock \href {http://arxiv.org/abs/1005.0012} {\path{arXiv:1005.0012}},
  \href {http://dx.doi.org/10.1103/PhysRevD.82.044025}
  {\path{doi:10.1103/PhysRevD.82.044025}}.

\bibitem{Cannon:2011xk}
K.~Cannon, C.~Hanna, D.~Keppel, {Efficiently enclosing the compact binary
  parameter space by singular-value decomposition}, Phys.Rev. D84 (2011)
  084003.
\newblock \href {http://arxiv.org/abs/1101.4939} {\path{arXiv:1101.4939}},
  \href {http://dx.doi.org/10.1103/PhysRevD.84.084003}
  {\path{doi:10.1103/PhysRevD.84.084003}}.

\bibitem{JXu_LZikatanov_2003a}
J.~Xu, L.~Zikatanov, \href{http://dx.doi.org/10.1007/s002110100308}{Some
  observations on babu\v ska and brezzi theories}, Numer. Math. 94~(1) (2003)
  195--202.
\newblock \href {http://dx.doi.org/10.1007/s002110100308}
  {\path{doi:10.1007/s002110100308}}.
\newline\urlprefix\url{http://dx.doi.org/10.1007/s002110100308}

\bibitem{DBSzyld_2006a}
D.~B. Szyld, \href{http://dx.doi.org/10.1007/s11075-006-9046-2}{The many proofs
  of an identity on the norm of oblique projections}, Numer. Algorithms
  42~(3-4) (2006) 309--323.
\newblock \href {http://dx.doi.org/10.1007/s11075-006-9046-2}
  {\path{doi:10.1007/s11075-006-9046-2}}.
\newline\urlprefix\url{http://dx.doi.org/10.1007/s11075-006-9046-2}

\bibitem{Quarteroni2010}
A.~Quarteroni, R.~Sacco, F.~Saleri, Numerical Mathematics, Springer, Berlin,
  2010.

\bibitem{Barish:1999vh}
B.~C. Barish, R.~Weiss, {LIGO and the detection of gravitational waves}, Phys.
  Today 52N10 (1999) 44--50.

\bibitem{Waldmann:2006bm}
S.~J. Waldmann, {Status of LIGO at the start of the fifth science run}, Class.
  Quantum Grav. 23 (2006) S653--S660.

\bibitem{Hild:2006bk}
S.~Hild, {The status of GEO 600}, Class. Quantum Grav. 23 (2006) S643--S651.

\bibitem{Acernese:2006bj}
F.~Acernese, et~al., {The Virgo status}, Class. Quantum Grav. 23 (2006)
  S635--S642.

\bibitem{Abbott:2007kva}
B.~Abbott, et~al., {LIGO: The Laser Interferometer Gravitational-Wave
  Observatory}\href {http://arxiv.org/abs/0711.3041} {\path{arXiv:0711.3041}}.

\bibitem{Abadie:2010cfa}
J.~Abadie, et~al., {Predictions for the Rates of Compact Binary Coalescences
  Observable by Ground-based Gravitational-wave Detectors}, Class. Quantum
  Grav. 27 (2010) 173001.
\newblock \href {http://arxiv.org/abs/1003.2480} {\path{arXiv:1003.2480}},
  \href {http://dx.doi.org/10.1088/0264-9381/27/17/173001}
  {\path{doi:10.1088/0264-9381/27/17/173001}}.

\bibitem{Centrella:2010mx}
J.~Centrella, J.~G. Baker, B.~J. Kelly, J.~R. van Meter, {Black-hole binaries,
  gravitational waves, and numerical relativity}, Rev.Mod.Phys. 82 (2010) 3069.
\newblock \href {http://arxiv.org/abs/1010.5260} {\path{arXiv:1010.5260}},
  \href {http://dx.doi.org/10.1103/RevModPhys.82.3069}
  {\path{doi:10.1103/RevModPhys.82.3069}}.

\bibitem{Pretorius:2007nq}
F.~Pretorius, {Binary Black Hole Coalescence}\href
  {http://arxiv.org/abs/0710.1338} {\path{arXiv:0710.1338}}.

\bibitem{Baumgarte:2002jm}
T.~W. Baumgarte, S.~L. Shapiro, Numerical relativity and compact binaries,
  Physics Reports 376~(2) (2003) 41--131.
\newblock \href {http://arxiv.org/abs/gr-qc/0211028}
  {\path{arXiv:gr-qc/0211028}}.

\bibitem{Sarbach:2012pr}
O.~Sarbach, M.~Tiglio, {Continuum and Discrete Initial-Boundary-Value Problems
  and Einstein's Field Equations}To appear in Living Reviews in Relativity, 191
  pages.
\newblock \href {http://arxiv.org/abs/1203.6443} {\path{arXiv:1203.6443}}.

\bibitem{Blanchet:LRR}
L.~Blanchet, \href{http://www.livingreviews.org/lrr-2006-4}{Gravitational
  radiation from post-newtonian sources and inspiralling compact binaries},
  Living Reviews in Relativity 9~(4).
\newline\urlprefix\url{http://www.livingreviews.org/lrr-2006-4}

\bibitem{Blanchet:1995fg}
L.~Blanchet, T.~Damour, B.~R. Iyer, {Gravitational waves from inspiralling
  compact binaries: Energy loss and wave form to second postNewtonian order},
  Phys. Rev. D51 (1995) 5360.
\newblock \href {http://arxiv.org/abs/gr-qc/9501029}
  {\path{arXiv:gr-qc/9501029}}, \href
  {http://dx.doi.org/10.1103/PhysRevD.51.5360}
  {\path{doi:10.1103/PhysRevD.51.5360}}.

\bibitem{Will:1996zj}
C.~M. Will, A.~G. Wiseman, Gravitational radiation from compact binary systems:
  gravitational waveforms and energy loss to second post-{N}ewtonian order,
  Phys. Rev. D 54 (1996) 4813--4848.
\newblock \href {http://arxiv.org/abs/gr-qc/9608012}
  {\path{arXiv:gr-qc/9608012}}.

\bibitem{Allen:2005fk}
B.~Allen, W.~G. Anderson, P.~R. Brady, D.~A. Brown, J.~D. Creighton,
  {FINDCHIRP: An Algorithm for detection of gravitational waves from
  inspiraling compact binaries}\href {http://arxiv.org/abs/gr-qc/0509116}
  {\path{arXiv:gr-qc/0509116}}.

\bibitem{Tu:2010zz}
L.-C. Tu, Q.~Li, Q.-L. Wang, C.-G. Shao, S.-Q. Yang, et~al., {New determination
  of the gravitational constant G with time-of-swing method}, Phys.Rev. D82
  (2010) 022001.
\newblock \href {http://dx.doi.org/10.1103/PhysRevD.82.022001}
  {\path{doi:10.1103/PhysRevD.82.022001}}.

\bibitem{Owen_B:96}
B.~J. Owen, Search templates for gravitational waves from inspiraling binaries:
  Choice of template spacing, Phys. Rev. D 53 (1996) 6749.

\bibitem{Owen99}
B.~J. Owen, B.~S. Sathyaprakash, Matched filtering of gravitational waves from
  inspiraling compact binaries: Computational cost and template placement,
  Phys. Rev. D 60 (1999) 022002.
\newblock \href {http://arxiv.org/abs/gr-qc/9903108}
  {\path{arXiv:gr-qc/9903108}}.

\bibitem{2006CQGra..23.5477B}
S.~{Babak}, R.~{Balasubramanian}, D.~{Churches}, T.~{Cokelaer}, B.~S.
  {Sathyaprakash}, {A template bank to search for gravitational waves from
  inspiralling compact binaries: I. Physical models}, Classical and Quantum
  Gravity 23 (2006) 5477--5504.
\newblock \href {http://arxiv.org/abs/arXiv:gr-qc/0604037}
  {\path{arXiv:arXiv:gr-qc/0604037}}, \href
  {http://dx.doi.org/10.1088/0264-9381/23/18/002}
  {\path{doi:10.1088/0264-9381/23/18/002}}.

\bibitem{Finn:1992wt}
L.~S. Finn, {Detection, measurement and gravitational radiation}, Phys.Rev. D46
  (1992) 5236--5249.
\newblock \href {http://arxiv.org/abs/gr-qc/9209010}
  {\path{arXiv:gr-qc/9209010}}, \href
  {http://dx.doi.org/10.1103/PhysRevD.46.5236}
  {\path{doi:10.1103/PhysRevD.46.5236}}.

\bibitem{Maggiore}
M.~Maggiore, Gravitational Waves Volume 1: Theory and Experiments, Oxford
  University Press, Oxford, 2008.

\bibitem{Abbott:2007kv}
B.~Abbott, et~al., {LIGO: The Laser Interferometer Gravitational-Wave
  Observatory}, Rept. Prog. Phys. 72 (2009) 076901.
\newblock \href {http://arxiv.org/abs/0711.3041} {\path{arXiv:0711.3041}},
  \href {http://dx.doi.org/10.1088/0034-4885/72/7/076901}
  {\path{doi:10.1088/0034-4885/72/7/076901}}.

\bibitem{Ajith:2009fz}
P.~Ajith, S.~Bose, {Estimating the parameters of non-spinning binary black
  holes using ground-based gravitational-wave detectors: Statistical errors},
  Phys.Rev. D79 (2009) 084032.
\newblock \href {http://arxiv.org/abs/0901.4936} {\path{arXiv:0901.4936}},
  \href {http://dx.doi.org/10.1103/PhysRevD.79.084032}
  {\path{doi:10.1103/PhysRevD.79.084032}}.

\bibitem{PhysRevD.79.022001}
B.~Abbott, et~al.,
  \href{http://link.aps.org/doi/10.1103/PhysRevD.79.022001}{Einstein@home
  search for periodic gravitational waves in ligo s4 data}, Phys. Rev. D 79
  (2009) 022001.
\newblock \href {http://dx.doi.org/10.1103/PhysRevD.79.022001}
  {\path{doi:10.1103/PhysRevD.79.022001}}.
\newline\urlprefix\url{http://link.aps.org/doi/10.1103/PhysRevD.79.022001}

\bibitem{Brady:1998nj}
P.~R. Brady, T.~Creighton, {Searching for periodic sources with LIGO. 2.
  Hierarchical searches}, Phys.Rev. D61 (2000) 082001.
\newblock \href {http://arxiv.org/abs/gr-qc/9812014}
  {\path{arXiv:gr-qc/9812014}}, \href
  {http://dx.doi.org/10.1103/PhysRevD.61.082001}
  {\path{doi:10.1103/PhysRevD.61.082001}}.

\bibitem{Vallisneri:2011ts}
M.~Vallisneri, {Beyond Fisher: exact sampling distributions of the
  maximum-likelihood estimator in gravitational-wave parameter estimation},
  Phys. Rev. Lett. 107 (2011) 191104.
\newblock \href {http://arxiv.org/abs/1108.1158} {\path{arXiv:1108.1158}},
  \href {http://dx.doi.org/10.1103/PhysRevLett.107.191104}
  {\path{doi:10.1103/PhysRevLett.107.191104}}.

\bibitem{Kanner:2008zh}
J.~Kanner, et~al., {LOOC UP: Locating and observing optical counterparts to
  gravitational wave bursts}, Class. Quantum Grav. 25 (2008) 184034.
\newblock \href {http://arxiv.org/abs/0803.0312} {\path{arXiv:0803.0312}},
  \href {http://dx.doi.org/10.1088/0264-9381/25/18/184034}
  {\path{doi:10.1088/0264-9381/25/18/184034}}.

\bibitem{2010AAS...21540606S}
P.~S. {Shawhan}, {LIGO Scientific Collaboration}, {Virgo Collaboration},
  {LOOC-UP: Seeking Optical Counterparts to Gravitational-Wave Signal
  Candidates}, in: Bulletin of the American Astronomical Society, Vol.~42 of
  Bulletin of the American Astronomical Society, 2010, pp. 230--+.

\bibitem{Cannon:2011vi}
K.~Cannon, R.~Cariou, A.~Chapman, M.~Crispin-Ortuzar, N.~Fotopoulos, et~al.,
  {Toward Early-Warning Detection of Gravitational Waves from Compact Binary
  Coalescence}, Astrophys.J. 748 (2012) 136.
\newblock \href {http://arxiv.org/abs/1107.2665} {\path{arXiv:1107.2665}}.

\bibitem{Casenave2012539}
F.~Casenave,
  \href{http://www.sciencedirect.com/science/article/pii/S1631073X12001483}{Ac%
curate a posteriori error evaluation in the reduced basis method}, Comptes
  Rendus Mathematique 350~(9–10) (2012) 539 -- 542.
\newblock \href {http://dx.doi.org/10.1016/j.crma.2012.05.012}
  {\path{doi:10.1016/j.crma.2012.05.012}}.
\newline\urlprefix\url{http://www.sciencedirect.com/science/article/pii/S16310%
73X12001483}

\bibitem{Canuto:2009:PEA:1654814.1654832}
C.~Canuto, T.~Tonn, K.~Urban, \href{http://dx.doi.org/10.1137/080724812}{A
  posteriori error analysis of the reduced basis method for nonaffine
  parametrized nonlinear pdes}, SIAM J. Numer. Anal. 47~(3) (2009) 2001--2022.
\newblock \href {http://dx.doi.org/10.1137/080724812}
  {\path{doi:10.1137/080724812}}.
\newline\urlprefix\url{http://dx.doi.org/10.1137/080724812}

\bibitem{Press92}
W.~H. Press, B.~P. Flannery, S.~A. Teukolsky, W.~T. Vetterling, Numerical
  Recipes, 2nd Edition, Cambridge University Press, New York, 1992.

\end{thebibliography}

\end{document}